\documentclass[a4paper,11pt]{amsart}

\pdfoutput=1

\usepackage[english]{babel}

\usepackage{ifthen}
\usepackage{lmodern}
\usepackage[T1]{fontenc}
\usepackage{microtype}

\usepackage{typearea}
\setboolean{@twoside}{false}
\typearea{0}
\setboolean{@twoside}{true}

\usepackage{amsmath}
\usepackage{amsthm}
\usepackage{amssymb}
\usepackage{mathtools}
\usepackage{mathrsfs}
\usepackage{dsfont}
\usepackage{xcolor}
\usepackage{scalerel}
\usepackage{verbatim}

\usepackage{multirow}
\usepackage{enumitem}
\usepackage[autostyle]{csquotes}

\usepackage[all]{xy}
\usepackage{url}

\usepackage{color}

\DeclareFontEncoding{OT2}{}{} % to enable usage of cyrillic fonts
\DeclareTextFontCommand{\textcyr}{\fontencoding{OT2}
    \fontfamily{wncyr}\fontseries{m}\fontshape{n}\selectfont}
\newcommand{\Sha}{\textcyr{Sh}}

\newcommand{\loc}{{\rm loc}}

\allowdisplaybreaks
\numberwithin{equation}{section}

\newcommand{\rleft}{\mathopen{}\mathclose\bgroup\left}
\newcommand{\rright}{\aftergroup\egroup\right}

\theoremstyle{plain}

\newtheorem{theorem}[equation]{Theorem}
\newtheorem{proposition}[equation]{Proposition}
\newtheorem{prop}[equation]{Proposition}

\newtheorem{lemma}[equation]{Lemma}
\newtheorem{corollary}[equation]{Corollary}

\newtheorem{question}[equation]{Question}

\newtheorem{main-theorem}[equation]{Theorem}

\theoremstyle{definition}

\newtheorem{remark}[equation]{Remark}
\newtheorem{example}[equation]{Example}
\newtheorem{definition}[equation]{Definition}
\newtheorem{definitions}[equation]{Definitions}

\newtheorem{construction}[equation]{Construction}
\newtheorem{subsec}[equation]{}

\DeclareMathOperator{\cone}{cone}

\DeclareMathOperator{\Ind}{Ind}

\DeclareMathOperator{\Br}{Br}

\newcommand{\Ad}{\mathds{A}}
\newcommand{\Cd}{\mathds{C}}
\newcommand{\Hd}{\mathds{H}}
\newcommand{\Qd}{\mathds{Q}}

\newcommand{\Vd}{\mathds{V}}
\newcommand{\Zd}{\mathds{Z}}
\newcommand{\Rd}{\mathds{R}}

\newcommand{\Gd}{\mathds{G}}
\newcommand{\Pd}{\mathds{P}}

\newcommand{\Am}{\mathcal{A}}

\newcommand{\Vm}{\mathcal{V}}
\newcommand{\Cm}{\mathcal{C}}
\newcommand{\Fm}{\mathcal{F}}
\newcommand{\Dm}{\mathcal{D}}

\newcommand{\Pm}{\mathcal{P}}

\newcommand{\Mm}{\mathcal{M}}

\newcommand{\Um}{\mathcal{U}}
\newcommand{\Om}{\mathcal{O}}
\newcommand{\Sm}{\mathcal{S}}

\newcommand{\Xm}{\mathcal{X}}

\newcommand{\Mfunc}{\Mm_\Gamma}

\newcommand{\gal}{\mathcal{G}}

\newcommand{\C}{{\Cd}}
\newcommand{\R}{{\Rd}}
\newcommand{\Q}{{\Qd}}
\newcommand{\Z}{{\Zd}}
\newcommand{\G}{{\Gd}}
\newcommand{\A}{{\Ad}}

\newcommand{\sX}{{\mathcal X}}
\newcommand{\sV}{{\mathcal V}}
\newcommand{\sP}{{\mathcal P}}
\newcommand{\sD}{{\mathcal{D}}}
\newcommand{\sC}{{\mathcal C}}
\newcommand{\sF}{{\mathcal F}}
\newcommand{\sA}{{\mathcal A}}

\newcommand{\sN}{{\mathcal{N}}}
\newcommand{\sU}{{\mathcal{U}}}

\newcommand{\sG}{{\mathcal G}}
\newcommand{\sQ}{{\mathcal Q}}
\newcommand{\sH}{{\mathcal H}}

\newcommand{\vs}{\varsigma}
\newcommand{\ve}{{\varepsilon}}

\newcommand{\vk}{\varkappa}

\newcommand{\into}{\hookrightarrow}
\newcommand{\isoto}{\xrightarrow{\sim}}
\newcommand{\onto}{\twoheadrightarrow}
\newcommand{\labelto}[1]{\xrightarrow{\makebox[1.5em]{\scriptsize ${#1}$}}}
\newcommand{\longisoto}{\xrightarrow{\makebox[1.5em]{\scriptsize ${\sim}$}}}

\newcommand{\Hom}{{\rm Hom}}
\newcommand{\Spec}{{\rm Spec\,}}
\newcommand{\SAut}{{\rm SAut}}
\newcommand{\Aut}{{\rm Aut}}
\newcommand{\Val}{{\rm Val}}
\newcommand{\val}{{\rm val}}
\newcommand{\Gal}{{\rm Gal}}
\newcommand{\Stab}{{\rm Stab}}

\newcommand{\siga}{{\sigma_\gamma}}

\newcommand{\ctil}{{\tilde c}}
\newcommand{\upgam}{{}^\gamma}
\newcommand{\upgamM}{\hs^{\gamma^{-1}}\hm}
\newcommand{\qs}{{q}}

\newcommand{\diag}{{\rm diag}}

\newcommand{\spl}{{\rm spl}}

\newcommand{\scc}{{\rm sc}}

\renewcommand{\gg}{{\mathfrak{g}}}

\newcommand{\Gbar}{{\overline{G}}}
\newcommand{\Ztil}{{\smash{\widetilde{Z}}}}

\newcommand{\emm}{\bfseries}

\newcommand{\X}{{{\sf X}}}

\newcommand{\im}{\mathrm{im}}
\newcommand{\inn}{\mathrm{inn}}
\newcommand{\Inn}{\mathrm{Inn}}

\newcommand{\BRD}{{\rm BRD}}
\newcommand{\Dyn}{{\rm Dyn}}
\newcommand{\id}{{\rm id}}

\newcommand{\NGH}{{\sN_G(H)}}

\newcommand{\veg}{{\ve_\gamma}}

\newcommand{\Gtil}{{\smash{\widetilde{G}}}}

\newcommand{\MGz}{(M_\Gamma)_0}
\newcommand{\CXz}{(C_X)_0}

\newcommand{\GL}{{\rm GL}}
\newcommand{\SO}{{\rm SO}}

\newcommand{\SL}{{\rm SL}}
\newcommand{\SU}{{\rm SU}}

\newcommand{\Lie}{{\rm Lie}}

\newcommand{\EE}{{\sf E}}
\newcommand{\FF}{{\sf F}}
\newcommand{\GG}{{\sf G}}
\newcommand{\Af}{{\sf A}}
\newcommand{\BB}{{\sf B}}
\newcommand{\CC}{{\sf C}}
\newcommand{\DD}{{\sf D}}

\newcommand{\vktil}{{\widetilde \vk}}

\newcommand{\Bbar}{{\overline{B}}}

\newcommand{\dmd}{{\scriptscriptstyle \diamondsuit}}

\newcommand{\cbar}{{c}}

\newcommand{\Tbar}{{\overline T}}
\newcommand{\Hbar}{{\overline H}}

\newcommand{\tr}{{\rm tr}}

\newcommand{\deltatil}{{\tilde \delta}}

\newcommand{\ste}[1]{{\scriptscriptstyle (#1)}}

\newcommand{\Sp}{{\rm Sp}}

\newcommand{\emphb}[1]{\textbf{#1}}

\newcommand{\Omt}{\Omega^{\ste{2}}}
\newcommand{\Omone}{\Omega^{\ste{1}}}

\newcommand{\mug}{{\mu_\gamma}}

\newcommand{\Gm}{{\mathcal{G}}}

\newcommand{\lra}{\longrightarrow }

\newcommand{\Out}{{\rm Out}}

\newcommand{\CF}{{\rm CF}}

\newcommand{\alphA}{a}
\newcommand{\ztil}{{\tilde z}}
\newcommand{\kk}{{\ker}}

\newcommand{\Dbar}{{\overline{D}}}
\newcommand{\half}{{\tfrac{1}{2}}}

\newcommand{\Sigmat}{{\Sigma^{\ste{2}}}}
\newcommand{\Sigmac}{{\Sigma^{\rm sc}}}
\newcommand{\Sigmattf}{{\Sigma_{\rm abc}}}

\newcommand{\Pf}{{\sf P}}
\newcommand{\Qf}{{\sf Q}}

\newcommand{\Pl}{{\rm Pl}}
\newcommand{\Qbar}{{\overline{\Q}}}
\newcommand{\Hm}{{\mathcal{H}}}
\newcommand{\Hnul}{\widehat{H}^0}

\newcommand{\cupdot}{\cup}

\newcommand{\hs}{\kern 0.8pt}
\newcommand{\hsss}{\kern 1.2pt}
\newcommand{\hm}{\kern -0.8pt}
\newcommand{\hmm}{\kern -1.2pt}

\newcommand{\stY}{\smash{\widetilde Y}}
\newcommand{\stH}{\smash{\widetilde H}}
\newcommand{\stsV}{\smash{\widetilde \sV}}
\newcommand{\stsX}{\smash{\widetilde \sX}}
\newcommand{\stSigma}{\smash{\widetilde \Sigma}}
\newcommand{\stsD}{\smash{\widetilde \sD}}

\makeatletter
\newcases{dlrcases}{\quad}{%
  $\m@th\displaystyle{##}$\hfil}{$\m@th\displaystyle{##}$\hfil}{\lbrace}{\rbrace}
\newcases{lrcases}{\quad}{%
  $\m@th{##}$\hfil}{{##}\hfil}{\lbrace}{\rbrace}
\makeatother

\begin{document}

\title[Models of spherical varieties]
{Existence of equivariant models\\ of spherical varieties and other $G$-varieties}

\author{Mikhail Borovoi}
\address{Borovoi: Raymond and Beverly Sackler School of Mathematical Sciences,
  Tel Aviv University, 6997801 Tel Aviv, Israel}
\email{borovoi@tauex.tau.ac.il}

\author{Giuliano Gagliardi}
\address{Gagliardi: Raymond and Beverly Sackler School of Mathematical Sciences,
  Tel Aviv University, 6997801 Tel Aviv, Israel}
\curraddr{Institut f\"ur Algebra, Zahlentheorie und Diskrete Mathematik,
Leibniz Universit\"at Hannover,
Welfengarten 1, 30167 Hannover, Germany}
\email{gagliardi@math.uni-hannover.de}

\thanks{This research was partially supported by the Israel Science Foundation (grant No. 870/16)}

\keywords{Algebraic group, spherical variety, affine spherical variety, horospherical variety, equivariant model,
     inner form, pure inner form, Tits class. }

\subjclass{Primary%
: 14M27; % Algebraic geometry, Compactifications; symmetric and spherical varieties
  secondary: %
  14M17% Algebraic geometry, Homogeneous spaces and generalizations
, 14G20% Algebraic geometry, Local ground fields
, 14G25% Algebraic geometry, Global ground fields
, 14G27% Algebraic Geometry, Other nonalgebraically closed ground fields
, 20G15% Group theory and generalizations,  Linear algebraic groups over arbitrary fields
, 12G05% Galois cohomology
}

%\date{\today}

\begin{abstract}
  Let $k_0$ be a field of characteristic $0$ with algebraic
  closure $k$. Let $G$ be a connected reductive  $k$-group, and let $Y$
  be a spherical variety over $k$ (a spherical homogeneous space or a spherical embedding).
  Let $G_0$ be a $k_0$-model ($k_0$-form) of~$G$.
  We give necessary and sufficient conditions for the existence
  of a $G_0$-equivariant $k_0$-model of $Y$.
\end{abstract}

\maketitle

\tableofcontents

\setcounter{section}{-1}

\section{Motivation}
\label{s:motivation}

\begin{subsec}\label{ss:mod-Y}
Throughout the article, $k_0$ is a field of characteristic $0$,
and  $k$ is  a fixed algebraic closure of $k_0$.
We write  $\sG=\Gal(k/k_0)$ for the Galois group of $k$ over $k_0$.
If $X_0$ is an algebraic variety over $k_0$,
we write $X_{0,k}$ or $(X_0)_k$ for $X_0\times_{k_0} k$.

  Let $Y$ be an algebraic variety over $k$.
  A  \emph{$k_0$-model} (or \emph{$k_0$-form}) of $Y$ is a pair $(Y_0, \nu_Y)$,
  where $Y_0$ is an algebraic variety over $k_0$ and
  \[
    \nu_Y\colon Y_{0,k}\coloneqq Y_0\times_{k_0} k\, \longisoto\, Y
  \]
  is an isomorphism of $k$-varieties.
  By abuse of language, we say just that $Y_0$ is a
  $k_0$-model of $Y$. If $(Y_0',\nu_Y')$ is another $k_0$-model, then by
  an {\em isomorphism} $(Y_0, \nu_Y)\isoto (Y_0', \nu_Y')$
  we mean an isomorphism of $k_0$-varieties
  $ Y_0\isoto Y_0'$.

  It is a classical problem of Algebraic Geometry and Galois
  Cohomology to describe the set of isomorphism classes of
  $k_0$-models of a $k$-variety $Y$. If there exists one such model
  $(Y_0,\nu_Y)$, then it defines an injective map from the set of all
  isomorphism classes of $k_0$-models into the first Galois cohomology
  set $H^1(k_0,\Aut(Y_0))$. Here we denote by $\Aut(Y)$ the group of
  automorphisms of $Y$ (regarded as an abstract group) and we write
  $\Aut(Y_0)$ for the group $\Aut(Y)$ equipped with the Galois action
  induced by the $k_0$-model $Y_0$ of $Y$. If $Y$ is a
  \emph{quasi-projective} variety, then this injective map is
  bijective; see Serre \cite[III.1.3, Proposition~5]{ser97}.
\end{subsec}

\begin{subsec}\label{ss:model-G}
  When $Y$ carries an additional structure, we wish $Y_0$ to carry
  this structure as well. For example, if $G$ is an algebraic group
  over $k$, we define a $k_0$-model of $G$ as a pair $(G_0,\nu_G)$
  where $G_0$ is an algebraic group over $k_0$ and
  \[
    \nu_G\colon\, G_{0,k}\coloneqq  G_0\times_{k_0} k\, \longisoto\,  G
  \]
  is an isomorphism of algebraic $k$-groups. We define isomorphisms of
  $k_0$-models of $G$
  in a natural way.
  If there exists one $k_0$-model
  $(G_0,\nu_G)$, then it defines a bijection between the set of
  isomorphism classes of $k_0$-models of $G$ and the first Galois
  cohomology set $H^1(k_0,\Aut(G_0))$; see Serre \cite[III.1.3,
  Corollary of Proposition~5]{ser97}.
  \end{subsec}

\begin{remark}
We use the term ``$k_0$-model'' rather than ``$k_0$-form''
because the term ``form'' is overloaded.
However, we do use  the classical term ``inner (twisted) form''.
\end{remark}

\begin{subsec}
  A $k_0$-torus $T_0$ is called {\em split} if it is isomorphic to $(\G_{m,k_0})^n$
  for some natural number $n$, where $\G_{m,k_0}$
  is the multiplicative group over $k_0$.
  A reductive $k_0$-group $G_0$ is called {\em split}
  if there is  a split maximal torus $T_0\subset G_0$.
  A reductive $k_0$-group $G_0$ is called {\em quasi-split}
  if there is a Borel subgroup $B_0\subset G_0$.
  Clearly, any split $k_0$-group is quasi-split.

  It is well known (see, for instance, Milne \cite[Corollary~23.57]{Milne})
  that any reductive $k$-group $G$
  admits a {\em split} $k_0$-model $G_\spl$ (also called the
  \emph{Chevalley form}). Then one can classify $k_0$-models of $G$ by
  the cohomology classes in $H^1(k_0,\Aut(G_\spl))$. One can also
  classify $k_0$-models of {\em classical} groups using algebras with
  involutions; see Tits \cite{Tits} and \enquote{The Book of Involutions}
  \cite{KMRT}.
\end{subsec}

\begin{subsec}\label{ss:G-k}
  Let $G$ be a (connected) reductive
  group over $k$ and $Y$ be a $k$-variety equipped with a $G$-action
  \[
    \theta\colon G\times_{k} Y\to Y\text{.}
  \]
  This means that $\theta$ is a morphism of $k$-varieties
  such that certain natural diagrams commute;
  see Milne \cite[Section 1.f]{Milne}.
  We say then that $Y$ is a $G$-$k$-variety. By a {\em $k_0$-model of
    $(G,Y,\theta)$} we mean a triple $(G_0,Y_0,\theta_0)$ as above, but
  over $k_0$, together with two isomorphisms
  \[
    \nu_G\colon G_{0,k}\isoto G\text{,} \quad \nu_Y\colon Y_{0,k}\isoto Y\text{,}
  \]
  where $\nu_G$ is an isomorphism of algebraic $k$-groups
  and $\nu_Y$ is an isomorphism of $k$-varieties, such that the following diagram commutes:
  \begin{equation}\label{e:G-Y}
    \begin{gathered}
    \xymatrix{
      G_{0,k}\times _k Y_{0,k}\ar[r]^-{\theta_{0,k}}\ar[d]_{\nu_G\times \nu_Y}
                                                 & Y_{0,k}\ar[d]^{\nu_Y} \\
      G\times _k Y\ar[r]^-\theta                      &Y\mathrlap{\text{.}}
    }
    \end{gathered}
  \end{equation}
  We define isomorphisms of $k_0$-models of $(G,Y,\theta)$ in a natural
  way. We wish to classify isomorphism classes of $k_0$-models
  $(G_0,Y_0,\theta_0)$ of triples $(G,Y,\theta)$.
\end{subsec}

\begin{subsec}
  One can classify $k_0$-models of triples
  $(G,Y,\theta)$ in two steps:
  \begin{enumerate}
  \item One classifies $k_0$-models $G_0$ of  $G$ (which is mostly known);
  \item For a given model $G_0$ of $G$, one classifies
  $k_0$-models $(G_0,Y_0,\theta_0)$  with this given $G_0$\hs.
  \end{enumerate}
  If $(G_0,Y_0,\theta_0)$ is a $k_0$-model of  $(G,Y,\theta)$ with given $G_0$, we say
  that $(Y_0,\theta_0)$ is a {\em $G_0$-equivariant $k_0$-model of $(Y,\theta)$.}
  Let $G$ be a (connected) reductive $k$-group and let $(Y,\theta)$ be a
  $G$-$k$-variety. For a given $k_0$-model $(G_0,\nu_G)$ of $G$, we
  wish to classify $G_0$-equivariant $k_0$-models of $Y$ up to
  isomorphism.

  Assume for simplicity that the variety $Y$ is quasi-projective. A
  $G_0$-equivariant $k_0$-model $(Y_0,\theta_0)$ of $(Y,\theta)$, {\em if
  it exists,} defines a bijection between the set of isomorphism
  classes of all such models and the first Galois cohomology set
  $H^1(k_0,\Aut^{G_0} (Y_0))$; see Serre \cite[III.1.3, Proposition
  5]{ser97}. Here we denote by $\Aut^G(Y)$ the group of
  $G$-equivariant automorphisms of $Y$ (regarded as an abstract
  group), and we write $\Aut^{G_0}(Y_0)$ for the group $\Aut^G(Y)$
  equipped with the Galois action induced by the $k_0$-model $Y_0$ of
  $Y$.

  We see that the main problem is to find {\em at least one} $G_0$-equivariant
  $k_0$-model of $Y_0$. We arrive at the following question:

\begin{question}\label{q:ACP}
  Let $G$ be a (connected) reductive $k$-group and let $Y$ be a
  $G$-$k$-variety. For a given $k_0$-model $G_0$ of $G$, does there
  exist a $G_0$-equivariant $k_0$-model of $Y$?
\end{question}

Our  Theorem \ref{t:sphfull'} and Theorem \ref{t:main2'}
together with Propositions \ref{p:BG-Hur} and \ref{p:nonqp}
completely answer Question \ref{q:ACP} in the case when $Y$ is a spherical
variety (a spherical homogeneous space or a spherical embedding).
\end{subsec}

\begin{remark}
Question \ref{q:ACP} was inspired by the articles of Akhiezer and Cupit-Foutou
\cite{acf14}, \cite{akh15}, and \cite{cf15}.
This section was inspired by the introduction of Moser-Jauslin and Terpereau \cite{MJTB}.
\end{remark}

\section{Introduction}
\label{s:intro}

\begin{subsec}\label{ss:intro-models}
  Let $k_0$, $k$, and $\sG$ be as in Subsection~\ref{ss:mod-Y}.
  In this article, by an algebraic variety over~$k$ we mean a separated
  and reduced scheme of finite type over $k$. By a (linear) algebraic group
  over $k$ we mean an affine group scheme of finite type over $k$.
\end{subsec}

\begin{subsec}\label{ss:intro-twist}
  Let $G$ be a connected linear algebraic group over $k$, and let $Y$
  be a $G$-$k$-variety. We consider the group $\Aut(G)$ of
  automorphisms of $G$. We regard $\Aut(G)$ as an abstract group. Any
  $g\in G(k)$ defines an \emph{inner automorphism}
  \[i_g\colon G\to G,\quad x\mapsto gxg^{-1}\text{ for }x\in G(k).\]
  We obtain a homomorphism
  \[i\colon G(k)\to \Aut(G).\] We denote by $\Inn(G)\subset \Aut(G)$
  the image of the homomorphism $i$ and we say that $\Inn(G)$ is the
  group of inner automorphisms of $G$. We identify $\Inn(G)$ with
  $\Gbar(k)$, where
  \[\Gbar = G/Z(G)\]
  and $Z(G)$ is the center of $G$.

  Let $G_\dmd$ be a $k_0$-model of $G$. We write $Z_\dmd$ for the
  center of $G_\dmd$; then $\Gbar_\dmd\coloneqq G_\dmd/Z_\dmd$ is a
  $k_0$-model of $\Gbar$.
   Let $\cbar\colon \sG\to \Gbar_\dmd(k)$ be a
  \emph{1-cocycle}, that is, a locally constant map
  for which  the following cocycle condition is satisfied:
\begin{equation}\label{e:cocycle-cond}
\cbar_{\gamma \beta  }=\cbar_\gamma\cdot\upgam\hm\cbar_{\beta  }
                  \quad \text{for all }\gamma,\beta  \in\sG,
\end{equation}
where we denote by $(\gamma,g)\mapsto\upgam \hm g$
the action of $\gamma\in\Gm$ on $g\in \Gbar_\dmd(k)$.
We denote the set of such 1-cocycles by $Z^1(\sG,\Gbar_\dmd(k))$ or by $Z^1(k_0,\Gbar_\dmd)$.
For $c\in Z^1(k_0,\Gbar)$ one can define the \emph{$\cbar$-twisted inner form}
$_\cbar(G_\dmd)$ of $G_\dmd$; see Subsection~\ref{ss:tw} below.
For simplicity we write $_\cbar G_\dmd$ for $_\cbar(G_\dmd)$.
\end{subsec}

\begin{subsec}\label{ss:intro-qs}
  Let $Y$ be a $G$-$k$-variety; see Subsection~\ref{ss:G-k}.
  It is well known that if $G$ is a  reductive $k$-group,
  then any $k_0$-model $G_0$ of $G$ is an inner form of a quasi-split
  model; see, for instance, \enquote{The Book of Involutions} \cite[Proposition (31.5)]{KMRT}.
  In other words, there exist a quasi-split model $G_\qs$ of $G$ and a
  1-cocycle $\cbar\in Z^1(k_0, \Gbar_\qs)$ such that
  $G_0= {}_\cbar G_\qs$. In some cases it is clear that the $G$-$k$-variety $Y$ admits a
  $G_\qs$-equivariant $k_0$-model. For example, assume that $Y=G/U$,
  where $U=R_u(B)$ is the unipotent radical of a Borel subgroup $B$ of
  $G$. Since $G_\qs$ is a \emph{quasi-split} model, there exists a
  Borel subgroup $B_\qs\subset G_\qs$ (defined over $k_0$). Set
  $U_\qs=R_u(B_\qs)$ (the unipotent radical of $B_\qs$); then $G_\qs/U_\qs$ is a $G_\qs$-equivariant
  $k_0$-model of $Y=G/U$.
\end{subsec}

\begin{subsec}\label{ss:intro-we-ask}
In the setting of Subsections~\ref{ss:intro-models} and \ref{ss:intro-twist},
let $G_\dmd$ be a $k_0$-model of $G$,
and let $G_0= {}_\cbar G_\dmd$, where $\cbar\in Z^1(k_0, \Gbar_\dmd)$.
Motivated by Subsection~\ref{ss:intro-qs}, \emph{we assume}
that $Y$ admits a $G_\dmd$-equivariant $k_0$-model $Y_\dmd$,
and we ask whether $Y$ admits a $G_0$-equivariant $k_0$-model $Y_0$.

We consider the short exact sequence
\[1\to Z_\dmd\to G_\dmd\to \Gbar_\dmd\to 1\]
and the connecting map
\[\delta\colon H^1(k_0,\Gbar_\dmd)\to H^2(k_0,Z_\dmd);\]
see Serre \cite[I.5.7, Proposition 43]{ser97}. If
$\cbar\in Z^1(k_0,\Gbar_\dmd)$, we write $[\cbar]$ for the
corresponding cohomology class in $H^1(k_0,\Gbar_\dmd)$. By abuse of
notation, we write $\delta[\cbar]$ for~$\delta([\cbar])$.

We consider the group $\sA\coloneqq\Aut^G(Y)$ of $G$-equivariant
automorphisms of $Y$, regarded as an abstract group. The
$G_\dmd$-equivariant $k_0$-model $Y_\dmd$ of $Y$ defines a
$\sG$-action on $\sA$, see \eqref{e:ggg-action} below, and we
denote the obtained $\sG$-group by $\sA_\dmd$. One can define the
second Galois cohomology set $H^2(\sG,\sA_\dmd)$. See Springer
\cite[Section 1.14]{spr66} for a definition of $H^2(\sG,\sA_\dmd)$ in
the case when the $\sG$-group $\sA_\dmd$ is nonabelian.

For all $z\in Z_\dmd(k)$ we consider the $G$-equivariant automorphism
 $\vk(z)$
of $Y$ defined by $y\mapsto z\cdot y$.
We obtain a $\sG$-equivariant homomorphism
\[\vk\colon Z_\dmd(k)\to\sA_\dmd\hs ,\]
which induces a map
\[\vk_*\colon H^2(k_0,Z_\dmd)\to H^2(\sG,\sA_\dmd).\]
\end{subsec}

\begin{theorem}[Theorem \ref{t:twist}]
\label{t:twist'}
Let $k_0$, $k$, and $\sG$ be as in Subsection \ref{ss:mod-Y}.
Let $G$ be a connected linear algebraic $k$-group,
and let $Y$ be a $G$-$k$-variety. Let $G_\dmd$ be
a $k_0$-model of $G$, and assume that $Y$ admits a
$G_\dmd$-equivariant $k_0$-model $Y_\dmd$\hs.
We also assume that $Y$ is quasi-projective.
With the above notation, let $\cbar\in Z^1(k_0,\Gbar_\dmd)$ be a 1-cocycle,
and consider its class $[\cbar]\in H^1(k_0,\Gbar_\dmd)$. Set
$G_0=\hs _c G_\dmd$ (the inner twisted form of $G_\dmd$ defined by the
1-cocycle $c$). Then the $G$-variety $Y$ admits a $G_0$-equivariant
$k_0$-model if and only if the cohomology class
\[\vk_*(\delta[c])\in H^2(\sG, \sA_\dmd)\]
is neutral.
\end{theorem}

\begin{remark}
In the case when $\sA$ is abelian, the condition \enquote{$\vk_*(\delta[\cbar])$ is neutral}
means that $\vk_*(\delta[\cbar])=1$; see Subsection \ref{ss:nonab-H2-0} below for details.
\end{remark}

\begin{corollary}
If either $Z=\{1\}$ or $\Am=\{1\}$, then $Y$ does admit a $G_0$-equivariant $k_0$-model
(because then $\vk_*(\delta[c])=1$).
\end{corollary}

\begin{subsec}\label{ss:sph-qs}
  In the setting of Subsections~\ref{ss:intro-models}
  and \ref{ss:intro-twist}, we now consider
  the case when $G$ is a (connected) reductive group and $Y = G/H$ is
  a \emph{spherical homogeneous space}, which means that for a Borel
  subgroup $B \subset G$ there exists an open $B$-orbit in $G/H$.
  We say then that $H$ is a \emph{spherical subgroup of} $G$.
  Spherical varieties (spherical homogeneous spaces and spherical
  embeddings) were considered in works of Luna, Vust, Brion, Knop,
  Losev, and others. The classification of spherical homogeneous
  spaces over algebraically closed fields of characteristic $0$ was
  completed in the works of Losev \cite{los09a} and
  Bravi and Pezzini \cite{bp14}, \cite{bp15}, \cite{bp16},
  as well as Cupit-Foutou \cite{cf2}. For general surveys on the
  theory of spherical varieties we refer to Timashev \cite{tim11} and Perrin \cite{per14}.
\end{subsec}

 Motivated by Subsection~\ref{ss:intro-qs} and Theorem~\ref{t:twist'},
 we assume that $G_0$ is quasi-split, and we ask whether $G/H$ admits
 a $G_0$-equivariant $k_0$-model.
We use the notation of Subsections \ref{ss:Luna-Vust} and \ref{ss:Om1-Om2}.

\begin{theorem}[Theorems \ref{t:sphqs} and  \ref{t:k0-point}]
  \label{t:sphqs'}
  Let $k_0$, $k$, and $\sG$ be as in Subsection \ref{ss:mod-Y}.
  Let $G$ be a (connected) reductive $k$-group.
  Let $G/H$ be a spherical homogeneous space of $G$,
  and let $(\sX, \Vm, \Omone,\Omt)$ be
  its combinatorial invariants. Let $G_0$ be a
  \emphb{quasi-split} $k_0$-model of $G$.
  If the $\Gm$-action on $\X^*(B)$ and $\Sm$ defined by $G_0$
  preserves the combinatorial invariants  $(\sX, \Vm, \Omone,\Omt)$ of $G/H$,
   then the spherical homogeneous space
  $G/H$ admits a $G_0$-equivariant $k_0$-model $Y_0$.
  Moreover, then $G/H$ admits a $G_0$-equivariant $k_0$-model $Y_0$
  having a $k_0$-rational point,
  that is, $Y_0=G_0/H_0$, where $H_0\subset G_0$ is a $k_0$-subgroup
  such that $(H_0)_k$ is conjugate to $H$ in $G$.
\end{theorem}

We explain the idea of the proof. First, we consider the case
  where $G/H$ is quasi-affine, that is, an open subvariety of an
  affine variety.
  Let $\mathrm{H}^0(G/H, \Om_{G/H})$ denote the ring of regular functions on $G/H$.
  Then the affine closure $X \coloneqq\Spec\hs \mathrm{H}^0(G/H, \Om_{G/H})$
  of $G/H$ is an affine spherical $G$-variety,
  whose $G$-equivariant isomorphism class therefore
  corresponds to a certain subscheme $C_X^\circ$ in the invariant
  Hilbert scheme $M_\Gamma$ defined by Alexeev and Brion \cite[Section 1.3]{ab05}.
  The fact that the $\sG$-action preserves the combinatorial
  invariants of $G/H$ allows us to define a $\sG$-action on
  $M_\Gamma$. We show that $C_X^\circ$ contains a $\sG$-fixed point,
  from which we then obtain a $k_0$-model of $G/H$.

  In the general case (when $G/H$ is not necessarily quasi-affine)
   we use a construction based on Brion
  \cite[Section~4.1]{bri97} to define a torus $C$ and a quasi-affine
  spherical homogeneous space $G'/H' \coloneqq (G \times C)/H'$ such
  that $G/H$ is the quotient of $G'/H'$ by the torus $C$. The
  combinatorial invariants of $G'/H'$ are obtained straightforwardly
  from the combinatorial invariants of $G/H$. Moreover, we construct a
  $k_0$-model $C_0$ of the torus $C$ such that the
  $\sG$-action on $\X^*(B)\oplus\X^*(C)$ and $\Sm$
  defined by the $k_0$-model $G_0 \times_{k_0}C_0$ of $G'=G\times C$ preserves the
  combinatorial invariants of~$G'/H'$. In this way the general case is
  reduced to the quasi-affine case.

\begin{subsec}\label{ss:Aq}
  In the setting of Subsection~\ref{ss:sph-qs}, write
  $G_0= {}_c G_\qs$ as in Subsection~\ref{ss:intro-qs}, where $G_\qs$
  is a quasi-split model of $G$ and $c\in Z^1(k_0,\Gbar_\qs)$. Let
  $G/H$ be a spherical homogeneous space.
  We ask whether $G/H$ admits a $G_0$-equivariant $k_0$-model.

  Assume that the $\sG$-action
  on $(\X^*(B),\Sm)$ defined by the $k_0$-model $G_0$ of $G$ preserves
  the combinatorial invariants $(\sX, \Vm, \Omone,\Omt)$ of $G/H$.
  Then the $\sG$-action defined by the $k_0$-model $G_\qs$ preserves
  the combinatorial invariants $(\sX, \Vm, \Omone,\Omt)$ of $G/H$ as
  well, because it is the same $\sG$-action on $(\X^*(B),\Sm)$. By
  Theorem \ref{t:sphqs'}, the spherical homogeneous space $G/H$ admits
  a $G_\qs$-equivariant $k_0$-model of the form $G_q/H_q$\hs, where
  $H_q\subset G_q$ is a subgroup defined over $k_0$. Set
  \[A_q=\sN_{G_q}(H_q)/H_q\hs,\] which is a $k_0$-model of the abelian $k$-group
  $A\coloneqq \sN_G(H)/H$.
  Then
  \[ A_q(k)=A(k)=\Aut^G(G/H)\text{;}\]
see, for instance, \cite[Corollary 4.3]{BG}.
Let $[G_0, G_0]$ denote the commutator subgroup of $G_0$,
and let $\Gtil_0$ denote the universal cover of the connected semisimple $k_0$-group $[G_0, G_0]$.
Similarly, we denote by $\Gtil_\qs$ the universal cover of $[G_\qs,G_\qs]$.
We may and shall identify the centers  $Z(\Gtil_0)$ and $Z(\Gtil_\qs)$.
Consider the composite homomorphism of $k_0$-groups
 \begin{equation}\label{e:vktil-2}
 \vktil\colon\hs  Z(\Gtil_0) = Z(\Gtil_q)\to Z(G_q) \into \sN_{G_q}(H_q)\to A_q
 \end{equation}
 and the induced homomorphism on cohomology
 \[\vktil_*\colon H^2(k_0, Z(\Gtil_0))\to H^2(k_0,A_q)\text{.}\]

Let $t(G_0) \in H^2(k_0,Z(\Gtil_0))$ denote the {\em Tits class of $G_0$}\hs.
The definition of the Tits class is given in Subsection~\ref{ss:Tits} below.
We shall consider
\[\vktil_*(t(G_0))\in H^2(k_0,A_q)\text{.}\]
The following theorem gives necessary and sufficient conditions for the existence of
a $G_0$-equivariant $k_0$-model of $G/H$.
 \end{subsec}

 \begin{main-theorem} \label{t:sphfull'}
 Let $k_0$, $k$, and $\sG$ be as in Subsection \ref{ss:mod-Y}.
  Let $G$ be a (connected) reductive $k$-group.
  Let $G/H$ be a spherical homogeneous space of $G$,
  and let $(\sX, \Vm, \Omone,\Omt)$ be
  its combinatorial invariants.
   Let $G_0$ be a  $k_0$-model of $G$.
Assume that the $\gal$-action defined by $G_0$ preserves
    $(\sX, \Vm, \Omone,\Omt)$. Then $G/H$ admits a $G_0$-equivariant
    $k_0$-model if and only if
\begin{equation}\label{e:vktil*}
  \vktil_*(t(G_0))=1\in H^2(k_0,A_q)
\end{equation}
  with the notation of \ref{ss:Aq}.
\end{main-theorem}

Theorem \ref{t:sphfull'} will be proved in Subsection \ref{ss:proof-main}.
The idea of proof is that the  theorem follows from Theorem \ref{t:sphqs'}
together with Proposition \ref{p:tits},
which is a special case of Theorem \ref{t:twist'}.

\begin{remark}
The Tits classes $t(G_0)\in H^2(k_0,Z(\Gtil_0))$ were computed
in ``The Book of Involutions'' \cite[Examples (31.8)--(31.11)]{KMRT} for all classical groups $G_0$.
See also the tables of the Tits classes for all simple $\R$-groups
in Borovoi's  appendix to Moser-Jauslin and Terpereau \cite{MJTB}.
\end{remark}

\begin{remark}
  Theorem~\ref{t:sphqs'} generalizes a result of Snegirov
  \cite[Theorem~1.1]{sni18}, who assumes also that $k_0$ is a {\em large field}
  and that $\sN_G(\sN_G(H))=\sN_G(H)$, where $\sN_{G}(H)$ denotes the
  normalizer of $H$ in $G$.
  His result  in turn generalizes a result of
  Akhiezer and Cupit-Foutou \cite[Theorem~4.4]{acf14}, who considered
  the case of a \emph{split} group $G_0$ defined over $k_0=\Rd$.
Theorem \ref{t:sphfull'} generalizes Theorem 1.5 of Snegirov \cite{sni18}.
It also generalizes Theorem 3.19 of Moser-Jauslin and Terpereau \cite{MJTB},
where the authors considered the case  when $k_0=\R$
and $H$ is a horospherical subgroup.
\end{remark}

\begin{remark}
In the paper \cite[Theorem 9.2]{BG},
the first-named author proved  a version of Theorem \ref{t:sphfull'}
with the sufficient condition that the spherical subgroup $H$ is {\em spherically closed,}
instead of the necessary and sufficient condition \eqref{e:vktil*}.
\end{remark}

\begin{subsec}
We consider spherical embeddings.
Let $G/H\into Y^e$ be a spherical embedding, that is, a
$G$-equivariant open embedding of $G/H$ into a normal irreducible
$G$-variety $Y^e$. With a spherical embedding $G/H\into Y^e$, the
Luna-Vust theory \cite{lv83} associates its {\em colored fan}
$\CF(Y^e)$, see Knop \cite[Theorem~3.3]{kno91},
Timashev \cite[Theorem~15.4]{tim11}, Perrin \cite[Theorem~3.1.10]{per14}, or
Subsection \ref{ss:emb-1} below. The following theorem gives necessary
and sufficient conditions for the existence of a $G_0$-equivariant
$k_0$-model of a spherical embedding $G/H\into Y^e$.
\end{subsec}

\begin{theorem}[Theorem  \ref{t:main2} and Proposition \ref{p:nonqp}]
\label{t:main2'}
Let $G/H\into Y^e$ be a spherical embedding with colored fan $\CF(Y^e)$.
Let $G_0$ be a $k_0$-model of $G$ and write $G_0=\hs_c G_\qs$,
where $G_\qs$ is a quasi-split inner form of $G_0$.
Assume that the $\Gm$-action defined by $G_0$ preserves the combinatorial invariants of $G/H$.
Then $Y^e$ admits a $G_0$-equivariant  $k_0$-model if and only if
the  following three conditions are satisfied:
\begin{enumerate}
\item[{\rm (i)}]The $\sG$-action on $\Omega$ defined by $G_0$
can be lifted to a continuous $\sG$-action $\alpha^\Dm$ on $\Dm$ such that
the colored fan $\CF(Y^e)$ is $\sG$-stable with respect to $(G_0\hs,\alpha^\Dm)$
with the notation of Subsection \ref{ss:G-stable} below;
\item [{\rm (ii)}]   $\vktil_*(t(G_0))=1\in H^2(k_0,A_q)$.
\item[{\rm (iii)}] The colored fan $\CF(Y^e)$ is a union of
    quasi-projective colored subfans of $\CF(Y^e)$ that are $\sG$-stable with
    respect to $(G_0\hs,\alpha^\Dm)$.
\end{enumerate}
\end{theorem}

See Section~\ref{s:embeddings} for the notion of a quasi-projective colored subfan.
Condition (iii) is trivially satisfied if the variety $Y^e$ is quasi-projective.
Theorem~\ref{t:main2'} will be deduced in Section~\ref{s:embeddings}
from our Theorem~\ref{t:sphfull'} and Theorem~2.23 of Huruguen \cite{Huruguen}.

\begin{subsec}
We consider horospherical homogeneous spaces.
In the setting of Subsections~\ref{ss:intro-models}
and \ref{ss:intro-twist}, we  consider
the case when $G$ is a reductive $k$-group
and $H\subset G$ is a {\em horospherical} subgroup,
that is, $H$ contains a maximal unipotent subgroup of $G$
(the unipotent radical of a Borel subgroup).
Then $G/H$ is called a horospherical homogeneous space.

In the case when $G$ is a simply connected simple $k$-group,
and $k_0$ is a {\em local} field
(that is, $\R$ or a $p$-adic field), using Tate duality, we write explicitly
the cohomological condition $\vk_*(t(G_0))=1$ in  Theorem \ref{t:sphfull'}  (here $\vktil_*=\vk_*$);
see Subsection \ref{ss:simple-R-p} and Theorem \ref{t:horo-types}.
For a horospherical homogeneous space of a simply connected absolutely simple
$k_0$-group $G_0$ over a  {\em number field} $k_0$,
we prove a local-global principle for the condition $\vk_*(t(G_0))=1$;
see Theorem \ref{t:number-HP}.
\end{subsec}

\begin{subsec}
The plan for the rest of the article is as follows. In
Section~\ref{s:pre} we recall basic definitions and results. In
Section~\ref{s:existence} we prove Theorem~\ref{t:twist'}. In
Section~\ref{s:tits} we apply Theorem~\ref{t:twist'} to homogeneous
spaces of a reductive group.
In Section~\ref{s:sphaff} we study the existence of models
of affine spherical varieties. In Section~\ref{s:spharb} we use
results of Section~\ref{s:sphaff} to prove Theorem~\ref{t:sphqs'}
and we deduce Theorem~\ref{t:sphfull'}.
In Section~\ref{s:embeddings} we consider equivariant models of spherical
embeddings and prove Theorem~\ref{t:main2'}.
In Sections \ref{s:horo} and \ref{s:horo-l-n} we consider equivariant models
of horospherical homogeneous spaces.
In particular, in Section  \ref{s:horo-l-n} we consider equivariant models
over local fields and number fields.
In Section~\ref{s:H x H} we consider equivariant models
of homogeneous varieties of the form $H^n/\Delta$,
where $H$ is a linear algebraic $k$-group and $\Delta\subset H^n$ is the diagonal,
that is, $H$ embedded into $H^n$ diagonally; in general these varieties are not spherical.
In Section~\ref{s:examples} we give examples.
In Appendix~\ref{s:G/U} we
compare the sets of $k_0$-rational points of a $k_0$-variety $X_0$ on
which a unipotent $k_0$-group $U_0$ acts and of the quotient variety
$X_0/U_0$.
In Appendix~\ref{s:AqAq} we describe the group $A_\qs$ in terms of
combinatorial data.
In Appendix~\ref{s:A} we give an alternative proof
of the important Proposition~\ref{l:inj}.
In Appendix~\ref{s:A-TN} we consider Tate duality.
\end{subsec}

\section{Preliminaries}
\label{s:pre}

\begin{subsec} \label{ss:semi-linear}
Let $k_0$, $k$, and $\sG$ be as in Subsection \ref{ss:mod-Y}.
For $\gamma \in \sG$, denote by $\gamma^{*} \colon \Spec k \to \Spec k$
the morphism of schemes induced by $\gamma$.
Notice that $(\gamma_1 \gamma_2)^{*}=\gamma_2^{*} \circ \gamma_1^{*}$.

Let $(Y,p_Y \colon Y\to \Spec k)$ be a $k$-scheme.
A \emph{$k/k_0$-semilinear automorphism of $Y$} is a pair $(\gamma,\mu)$,
where $\gamma\in \sG$ and $\mu\colon Y\to Y$
is an isomorphism of schemes, such that the following diagram commutes:
\begin{equation*}
\xymatrix@R=25pt@C=40pt{
Y\ar[r]^\mu \ar[d]_{p_Y}          & Y\ar[d]^{p_Y} \\
\Spec k\ar[r]^-{(\gamma^*)^{-1}}   & \Spec k\mathrlap{\text{.}}
}
\end{equation*}
In this case we also say that $\mu$ is a $\gamma$-semilinear automorphism of $Y$.
We shorten ``$\gamma$-semilinear automorphism'' to ``$\gamma$-semi-automorphism''.
Note that if $(\gamma,\mu)$ is a semi-automorphism of $Y$,
then $\mu$ uniquely determines $\gamma$;
see \cite[Lemma 1.6]{BG}.

We denote by $\SAut_{k/k_0}(Y)$ or just by $\SAut(Y)$ the group
of all $\gamma$-semi-automorphisms $\mu$ of $Y$,
where $\gamma$ runs over $\sG=\Gal(k/k_0)$.
By a \emph{semilinear action} of $\sG$ on $Y$ we mean a homomorphism of groups
\[\mu \colon \sG \rightarrow \SAut(Y), \quad \gamma \mapsto \mu_\gamma\]
such that for each $\gamma \in \sG$ the automorphism $\mu_\gamma$ is $\gamma$-semilinear.

If we have a $k_0$-scheme $Y_0$, then the formula
\begin{equation}\label{e:s-action}
\gamma \mapsto \id_{Y_0} \times  (\gamma^{*})^{-1}
\end{equation}
defines a semilinear action of $\sG$ on
\[Y_{0,k}= Y_0 \underset{\Spec k_0}{\times}\Spec k\text{.}\]
Thus a $k_0$-model of $Y$ induces a semilinear action of $\sG$ on $Y$.

\begin{definition}\label{d:algebraic}
Let $Y$ be a $k$-variety.
We say that a semilinear action $\mu$ of $\sG$ on $Y$ is {\em algebraic}
if there exists a finite Galois extension
$k_1/k_0$ in $k$ and a $k_1$-model $Y_1$ of $Y$ inducing the restriction
of $\mu$ to $\Gal(k/k_1)$.
\end{definition}

The semilinear $\sG$-action on $Y$ coming from a $k_0$-model is clearly algebraic.
Conversely:

\begin{lemma}[Galois descent for varieties]
\label{l:semilinear}
Let $Y$ be a $k$-variety.
An algebraic semilinear action $\mu$ of $\sG$ on $Y$
comes from a $k_0$-model of $Y$ if and only if
$Y$ admits a covering by $\sG$-stable affine open subvarieties,
or equivalently, if and only if
$Y$ admits a covering by $\sG$-stable quasi-projective open subvarieties.
In particular, if $Y$ is quasi-projective, then any algebraic
semilinear $\sG$-action on $Y$ comes from a $k_0$-model.
\end{lemma}

\begin{proof}
See Jahnel \cite[Proposition 2.5 and its proof]{Jahnel}.
See also  Borel and Serre \cite[Lemma 2.12]{BS}.
\end{proof}

\begin{subsec}
Let $(G,p_G \colon G \to \Spec k)$ be a linear algebraic group over $k$.
A $k/k_0$-semilinear automorphism of $G$ is a pair
$(\gamma,\tau)$ where $\gamma\in \sG$ and $\tau\colon G\to G$
is a morphism of schemes such that the diagram
\begin{equation*}
\xymatrix@R=25pt@C=40pt{
G\ar[r]^\tau \ar[d]_{p_G}          & G\ar[d]^{p_G} \\
\Spec k\ar[r]^-{(\gamma^*)^{-1}}   & \Spec k
}
\end{equation*}
commutes and that the $k$-morphism
\[\tau_\natural\colon\gamma_*G\to G\]
is an isomorphism of algebraic groups over $k$; see \cite[Definition 2.2]{BG}
for the notations $\tau_\natural$ and $\gamma_* G$.

We denote  by $\SAut_{k/k_0}(G)$, or just by $\SAut(G)$,
the group of all $\gamma$-semilinear automorphisms
$\tau$ of $G$, where $\gamma$ runs over $\sG=\Gal(k/k_0)$.
By a semilinear action of $\sG$ on $G$ we mean a homomorphism
\[\sigma \colon \sG \to \SAut(G), \quad \gamma \mapsto \siga\]
such that for all $\gamma \in \sG$ the semi-automorphism $\siga$ is $\gamma$-semilinear.
As above, a $k_0$-model $G_0$ of $G$ induces an algebraic semilinear action of $\sG$ on $G$.
Conversely, any algebraic semilinear action $\sigma$
of $\sG$ on $G$ comes from a $k_0$-model of $G$.
Indeed, since $G$ is an affine variety,
by Lemma \ref{l:semilinear} the action $\sigma$ comes
from some $k_0$-model $G_0$ of $G$ as a variety.
Since $\sigma_\gamma\in \SAut(G)$ for all $\gamma\in\sG$, we see that
the composition law, the inversion map, and the unit element in $G$
are ``defined over $k_0$'', that is, come from $k_0$-morphisms.
See  Jahnel \cite[Proposition 2.8]{Jahnel}.
\end{subsec}

Let $G$ be a linear algebraic group over $k$, and let $Y$ be a $G$-$k$-variety.
Let $G_0$ be a $k_0$-model of $G$. It gives rise to a semilinear action
$\sigma \colon \sG \rightarrow \SAut(G)$, $\gamma \mapsto \siga$.
Let $Y_0$ be a $G_0$-equivariant $k_0$-model of $Y$.
It gives rise to an algebraic semilinear action $\mu \colon \sG \to \SAut(Y)$ such that
\begin{equation}\label{e:sigma-equi}
\mu_\gamma(g \cdot y)=\sigma_\gamma(g) \cdot \mu_\gamma(y)
      \quad\text{for all }\gamma\in\sG,\  y\in Y(k),\, g\in G(k).
\end{equation}
\end{subsec}

\begin{definition}\label{d:semi-linear}
Let $Y$ be a $G$-variety, and let $G_0$ be a $k_0$-model of $G$
with semilinear action $\sigma\colon\sG\to\SAut(G)$.
Let $\mu\colon\sG\to\SAut(Y)$ be a semilinear action.
We say that $\mu$ is {\em $\sigma$-equivariant} if $\eqref{e:sigma-equi}$ holds.
\end{definition}

If $\mu$ comes from a $G_0$-equivariant $k_0$-model $Y_0$ of $Y$,
then it is $\sigma$-equivariant and algebraic. Conversely:

\begin{lemma}[Galois descent for $G$-varieties]
\label{p:lemma-5-4}
Let $G$, $Y$, $G_0$, and $\sigma$ be as in Definition~\ref{d:semi-linear}.
Let $\mu$ be a $\sigma$-equivariant algebraic semilinear action of $\sG$ on $Y$.
The following two assertions are equivalent:
\begin{enumerate}
\item[\rm (i)] $\mu$ comes from a $G_0$-equivariant $k_0$-model $Y_0$ of $Y$;
\item[\rm (ii)] $Y$ admits a covering by $\sG$-stable quasi-projective open subvarieties.
\end{enumerate}
\end{lemma}

\begin{proof}[Idea of proof]
Clearly, (i) implies (ii). Conversely, if (ii) holds, then by Lemma \ref{l:semilinear}
the semilinear action $\mu$ comes from a $k_0$-model $Y_0$ of $Y$.
Since $\mu$ is $\sigma$-equivariant, the $k$-morphism $G\times_k Y\to Y$
comes from a $k_0$-morphism $G_0\times_{k_0} Y_0\to Y_0$;
see Jahnel \cite[Proposition 2.8]{Jahnel}. See also \cite[Lemma 5.4]{BG}.
\end{proof}

Note that if $Y$ is a {\em homogeneous} $G$-variety, then it is quasi-projective
(see Borel \cite[Theorem 6.8]{bor91} or Springer \cite[Corollary 5.5.6]{Springer-book})
and therefore (ii) is satisfied.

\begin{subsec}\label{ss:tw}
Let $G$ be a group scheme over $k$.
We have an exact sequence
\[ 1\to \Aut(G)\to \SAut(G)\to \sG.\]
Let $G_0$ be a $k_0$-model of $G$; it defines a semilinear action
\[ \sigma\colon \sG\to \SAut(G)\text{.}\] Since $\Aut(G)$ is a normal
subgroup of $\SAut(G)$, this action induces an action of $\sG$ on the
group $\Aut(G)$ regarded as an abstract group.
Recall that a map
\[c\colon \sG\to \Aut(G)\] is called a \emph{1-cocycle} if the map $c$
is locally constant and satisfies the cocycle condition
\eqref{e:cocycle-cond}. The set of such 1-cocycles is denoted by
$Z^1(\sG,\Aut(G))$ or $Z^1(k_0,\Aut(G))$. For every $c\in Z^1(k_0,\Aut(G))$,
we consider the $c$-twisted semilinear action
\[\sigma'\colon \sG\to \SAut(G),\quad \gamma\mapsto c_\gamma\circ \sigma_\gamma.\]
Then, clearly, $\sigma'_\gamma$ is
a $\gamma$-semi-automorphism of $G$ for any $\gamma\in\sG$.
It follows from the cocycle condition \eqref{e:cocycle-cond} that
\[\sigma'_{\gamma \beta  }=\sigma'_\gamma\circ\sigma'_{\beta  }\
     \text{ for all }\gamma,\beta  \in\sG.\]
Since $G$ is an algebraic group, the semilinear action $\sigma'$
comes from some $k_0$-model $G_0'$ of $G$;
see Serre \cite[III.1.3, Corollary of Proposition 5]{ser97}.
We write $G_0'= {}_c G_0$ and say that $G_0'$ is
the \emph{twisted form of $G_0$ defined by the 1-cocycle} $c$.
 \end{subsec}

\begin{subsec}\label{ss:pure}
Let $G$ be a linear algebraic group over $k$, and let $Y$ be a
  \emph{quasi-projective} $G$-$k$-variety. Let $G_\dmd$ be a
  $k_0$-model of $G$, and assume that $Y$ admits a
  $G_\dmd$-equivariant $k_0$-model $Y_\dmd$.
Let
\[\ctil\colon \sG\to G_\dmd(k)\]
be a 1-cocycle \emph{with values in $G_\dmd$}, that is, $\ctil\in Z^1(k_0,G_\dmd)$.
With the notation of Subsection \ref{ss:intro-twist},
consider  $i\circ \ctil\in Z^1(k_0,\Gbar_\dmd)$;
by abuse of notation, we write $_\ctil G_\dmd$ for  $_{i\circ\ctil} G_\dmd$.
We say that $_\ctil G_\dmd$ is a \emph{pure inner form} of $G_\dmd$.

It follows from the cohomology exact sequence \eqref{e:coh-e-s} below
that, for a cocycle $c\in Z^1(k_0, \Gbar_\dmd)$,
the twisted form $_c G_\dmd$ is a pure inner form of $G_\dmd$ if and only if $\delta[c]=1$.
\end{subsec}

\begin{lemma}\label{l:pure-inner}
  Let $G$, $Y$, $G_\dmd$, and $Y_\dmd$ be as in \ref{ss:pure}.
  Let $\ctil\in Z^1(k_0,G_\dmd)$ be a 1-cocycle in $G_\dmd$.
  Consider the pure inner
  form~${}_\ctil G_\dmd$. Then $Y$ admits a
  ${}_\ctil G_\dmd$-equivariant $k_0$-model.
\end{lemma}

\begin{proof}
The $k_0$-models $G_\dmd$ and $Y_\dmd$ define
semilinear actions
\[ \sigma\colon \sG\to\SAut(G)\quad\text{and}\quad\mu\colon\sG\to\SAut(Y)\]
such that for any $\gamma\in \sG$ the semi-automorphism
$\mu_\gamma$ is $\sigma_\gamma$-equivariant, that is,
\[\mu(g\cdot y)=\sigma_\gamma(g)\cdot\mu_\gamma(y)
       \quad\text{for all }g\in G(k),\, y\in Y(k)\text{.}\]
Let $\ctil\colon \sG\to G(k)$ be a 1-cocycle, that is,
$\ctil\in Z^1(k_0,G_\dmd)$. Consider the pure inner form
$G_0= {}_\ctil G_\dmd$; then
\[\sigma_\gamma^0(g)=\ctil_\gamma\cdot \sigma_\gamma(g)\cdot \ctil_\gamma^{\hs-1}
             \text{ for all }\gamma\in\sG,\, g\in G(k)\text{,}\]
where $\sigma^0$ is the semilinear action defined by $G_0$.
Now we define the twisted form ${}_\ctil Y_\dmd$ as follows.
We set
\[\mu_\gamma^0(y)=\ctil_\gamma\cdot\mu_\gamma(y)\text{ for } y\in Y(k);\]
then $\mu_\gamma^0$ is a $\gamma$-semi-automorphism of $Y$.
Since $\ctil$ is a 1-cocycle, we have
\[\mu^0_{\gamma \beta  }=\mu^0_\gamma\circ\mu^0_{\beta  }
       \quad\text{for all }\gamma,\beta  \in\sG;\]
hence, $\mu^0$ is a semilinear action on $Y$.
An easy calculation shows that
\[\mu^0(g\cdot y)=\sigma^0_\gamma(g)\cdot\mu^0_\gamma(y)
       \quad\text{for all }g\in G(k),\, y\in Y(k);\]
hence, the semilinear action $\mu^0$ is $\sigma^0$-equivariant.

Let $k_1/k_0$ be a finite Galois extension in $k$ such that
the restriction of $\ctil$ to $\Gal(k/k_1)$ is identically $1$.
Set $G_1=G_\dmd\times_{k_0} k_1$ and $Y_1=Y_\dmd\times_{k_0} k_1$.
Then $Y_1$ is a $G_1$-equivariant $k_1$-model of $Y$,
and $Y_1$ induces a semilinear action $\mu^1\colon \Gal(k/k_1)\to\SAut(Y)$,
which is the restriction of $\mu$ and of $\mu^0$ to $\Gal(k/k_1)$.
Thus $\mu^0$ is algebraic.

Since $Y$ is quasi-projective, by Lemma \ref{p:lemma-5-4}
the algebraic $\sigma^0$-equivariant semilinear action $\mu^0\colon\sG\to\SAut(Y)$ on $Y$
defines a $G_0$-equivariant $k_0$-model $Y_0=\hs_\ctil Y_\dmd$.
\end{proof}

\begin{subsec}\label{ss:BRD}
Let $G$ be a reductive group over $k$.
Let
\[\BRD(G)=\BRD(G,T,B)=(X,X^\vee,R,R^\vee,\Sm,\Sm^\vee)\]
denote the {\em based root datum of $G$}.
Here $X=\X^*(T)$ is the character group of $T$,
$X^\vee$ is the cocharacter group of $T$,
$R=R(G,T)\subset X$ is the root system,
$R^\vee\subset X^\vee$ is the coroot system,
$\Sm=\Sm(G,T,B)\subset R$ is the basis of $R$ defined by $B$,
$\Sm^\vee\subset R^\vee$ is the corresponding basis of $R^\vee$.
See Springer \cite[Sections 1 and 2]{Springer-RG} for details.

We explain how the Borel subgroup $B$ determines $\Sm$.
Write $\gg=\Lie(G)$. Then we have  root decompositions
\[
\gg=\Lie(T)\oplus\bigoplus_{\beta\in R}\gg_\beta\qquad\text{and}
\qquad\Lie(B)=\Lie(T)\oplus\bigoplus_{\beta\in R^+}\gg_\beta\hs,
\]
where the {\em set of positive roots} $R^+=R^+(G,T,B)$ is defined
by the root decomposition of $\Lie(B)$.
A {\em simple root} is a positive root $\alpha \in R^+$
that is not a sum of two or more positive roots.
The basis $\Sm=\Sm(G,T,B)$ of $R$ is the set of all simple roots in $R^+$.
\end{subsec}

\begin{construction}\label{con:*-action}
Let $G_0$ be a $k_0$-model of $G$.
Let $\gamma\in \Gal(k/k_0)$; it defines  a $\gamma$-semi-automorphism $\sigma_\gamma$ of $G$.
We fix a Borel pair $(T,B)$, that is, a maximal torus $T$
and a Borel subgroup $B$ such that $T\subset B\subset G$.
Consider
\[ \sigma_\gamma(T)\subset\sigma_\gamma(B) \subset G.\]
Then $(\sigma_\gamma(T),\sigma_\gamma(B))$ is again a Borel pair, and by
Theorems 6.2.7 and 6.4.1 in Springer's book \cite{Springer-book},
there exists $g_\gamma\in G(k)$ such that
\[ g_\gamma\cdot\sigma_\gamma(T)\cdot g_\gamma^{-1}=T,
     \quad  g_\gamma\cdot\sigma_\gamma(B)\cdot g_\gamma^{-1}=B.\]
Set $\tau={\rm inn}(g_\gamma)\circ\sigma_\gamma\colon\  G\to G$; then $\tau$
is a $\gamma$-semi-automorphism of $G$, and
\begin{equation}\label{e:sig-T-B}
 \tau(B)=B,\quad \tau(T)=T.
\end{equation}

By \eqref{e:sig-T-B},  $\tau$ naturally acts on $\BRD(G,T,B)$;
we denote the corresponding automorphism of $\BRD(G,T,B)$ by $\veg$.
By definition,
\begin{equation*}
\veg(\chi)(b)=\upgam(\chi(\hs \tau^{-1}(b)\hs )\hs )
              \quad\text{for }\chi\in\X^*(B),\ b\in B(k),
\end{equation*}
and the same holds for the characters of $T$ (recall that $\X^*(B)=\X^*(T)$\hs ).
Since  $\tau(T)=T$ and $\tau(B)=B$, we see that $\veg$, when acting on $\X^*(T)$,
preserves $R$ and  $\Sm$.
Similarly, $\veg$ acts on $X^\vee$, $R^\vee$, and $\Sm^\vee$.
It is well known (see for example \cite[Section 3.2 and Proposition 3.1(a)]{BKLR})
that the automorphism $\veg$ does not depend on the choice of $g_\gamma$ and that the map
\[\ve\colon \Gal(k/k_0)\to \Aut\,\BRD(G,T,B),\quad \gamma\mapsto \veg \]
is a homomorphism.
The action of $\Gal(k/k_0)$ on $\BRD(G,T,B)$ via $\ve$ is called the {\em  $*$-action};
see Tits \cite[Section 2.3]{Tits}, or Conrad \cite[Remark 7.1.2]{Conrad},
or \cite[Section 7.1]{BG}.
\end{construction}

\begin{subsec}\label{ss:Luna-Vust}
Let $G$ be a (connected) reductive $k$-group
and $Y=G/H$ be a spherical homogeneous space.
According to the Luna-Vust theory of spherical varieties
(Luna and Vust \cite{lv83}, Knop \cite{kno91}),
with $Y=G/H$  one can associate a certain triple $(\sX, \Vm, \Dm)$ as follows.

  We fix a Borel subgroup $B$ of $G$ and a maximal torus $T\subset B$.
  We denote by $\X^*(B)$ the character group of $B$.
  Let $k(Y)$ denote the field of rational functions of $Y$.
For $\chi\in\X^*(B)$, let
 $k(Y)^{(B)}_\chi$ denote the space of $\chi$-eigenfunctions in $k(Y)$, that is,
 the $k$-space of rational functions $f\in k(X)$ such that
\[  f(b^{-1}\hm\cdot\hm y)=\chi(b)\hs f(y)\quad \text{for all } b\in B(k),\ y\in Y(k).\]
Since $B$ has an open dense orbit in $Y$,
the $k$-dimension of $k(Y)^{(B)}_\chi$ is  at most 1.
Let $\sX=\sX(Y)\subset\X^*(B)$ denote the set
of characters $\chi$ of $B$ such that $k(Y)^{(B)}_\chi\neq \{0\}$.
Then $\sX$ is a subgroup of $\X^*(B)$ called the {\em weight lattice of $Y$.}
We set
\[ V=V(Y)=\Hom_\Z(\sX,\Q).\]

Let $\Val(k(Y))$ denote the {\em set of $\Q$-valued valuations}
of the field $k(Y)$ that are trivial on $k$.
The group  $G(k)$ naturally acts on $k(Y)$ and on $\Val(k(Y))$.
We shall consider
the set $\Val^B(k(Y))$ of $B(k)$-invariant valuations and
the set $\Val^G(k(Y))$ of $G(k)$-invariant valuations.
We have a canonical map
\[\rho\colon \Val^B(k(Y))\to V,\quad v\mapsto (\chi\mapsto v(f_\chi)\hs ),\]
where $v\in\Val^B(k(Y)),\ \chi\in\sX,\ f_\chi\in k(Y)^{(B)}_\chi$\hm, $f_\chi\neq 0$.
We denote by
\[\sV=\sV(Y)\coloneqq \rho(\Val^G(k(Y)))\subset V\]
the image of $\Val^G(k(Y))$ in $V$.
It is a cone in $V$ called the {\em valuation cone of $Y$}.

Let $\sD=\sD(Y)$ denote the {\em set of colors of $Y$}, that is,
the set of $B$-invariant prime divisors in $Y$.
Each color $D\in\sD$ defines a $B$-invariant valuation of $k(Y)$ that we denote by $\val(D)$.
Thus we obtain a map
\[\val\colon \sD\to\Val^B(k(Y)).\]
By abuse of notation we denote $\rho(\val(D))\in V$ by $\rho(D)$.
Thus we obtain a map
\[\rho\colon \sD\to V.\]

For $D\in\sD$, let $\Stab_G(D)$ denote the stabilizer of $D\subset Y$ in $G$.
Clearly $\Stab_G(D)\supset B$, hence $\Stab_G(D)$ is a parabolic subgroup of $G$.
Let $\Sm=\Sm(G,T,B)$ denote the set of simple roots of $G$.
For $\alpha\in \Sm$, let $P_\alpha\supset B$
denote the corresponding minimal parabolic subgroup of $G$ containing $B$.
Let $\vs(D)$ denote the set of $\alpha\in \Sm$ for which
$P_\alpha$ is {\em not contained} in $\Stab_G(D)$.
We obtain a map
\[\vs\colon \sD\to\sP(\Sm),\]
where $\sP(\Sm)$ denotes the set of all subsets of $\Sm$.
See \cite[Section 6.2]{BG} for details.
\end{subsec}

\begin{subsec}\label{ss:Om1-Om2}
  From $(\Dm,\rho,\vs)$ we obtain two finite sets $\Omone$ and $\Omt$.
  Namely, we denote by $\Omega$ the image of the map
  \[\rho\times\vs\colon\, \Dm\,\longrightarrow \,  \Hom(\sX,\Q)\times\Pm(\Sm)\text{.}\]
  Every element of $\Omega$ has either one or two preimages in $\Dm$
  under $\rho\times\vs$; see \cite[Corollary~6.4]{BG}. Let
  $\Omone$ (resp.\ $\Omt$) denote the subset of  $\Omega$ consisting of
  elements with exactly one preimage (resp.\ exactly two preimages) in $\Dm$.
  Thus, from the set of colors $\Dm$ endowed with the maps $\rho$ and
  $\vs$, we obtained two subsets
  \[\Omone,\Omt\subset \Hom(\sX,\Q)\times\Pm(\Sm)\text{.}\]
  We say that $(\sX, \Vm, \Omone, \Omt)$ are the
  combinatorial invariants of $G/H$. By Losev's Uniqueness Theorem
  \cite[Theorem~1]{los09a} (see also \cite[Proposition 6.7 and Corollary 6.10]{BG}\hs)
  the invariants $(\sX,\Vm,\Omone,\Omt)$
  uniquely determine a spherical homogeneous space $G/H$
  of the given reductive $k$-group $G$ up to a $G$-equivariant isomorphism.

  A $k_0$-model $G_0$ of $G$ defines the  $*$-action $\ve$ of $\sG$ on $X=\X^*(B)$ and
  on $\Sm\subset \X^*(B)$, so for every $\gamma \in \gal$ we obtain a
  new set of invariants
  $({}^\gamma\sX, {}^\gamma\Vm, {}^\gamma\Omone,{}^\gamma\Omt)$.
\end{subsec}

 \begin{proposition}[{Huruguen \cite[Section 2.2]{Huruguen}, \cite[Proposition 8.2]{BG}}]
 \label{p:BG-Hur}
 If $G/H$ admits a $G_0$-equivariant $k_0$-model, then the $*$-action
 of $\sG$ on $\X^*(B)$ and $\Sm$ defined by $G_0$
  preserves the combinatorial invariants  $(\sX, \Vm, \Omone,\Omt)$, that is,
  \[(\hs^\gamma\hm\sX, {}^\gamma\Vm, {}^\gamma\Omone,{}^\gamma\Omt)=(\sX, \Vm, \Omone,\Omt)\]
  for all $\gamma\in\Gm$.
 \end{proposition}

\section{Model for an inner twist of the group}
\label{s:existence}

\begin{subsec}\label{ss:horo-1}
Let $k_0$, $k$, and $\sG$ be as in Subsection \ref{ss:mod-Y}.
Let $G$ be a linear algebraic group over $k$.
Let $Y$ be a $G$-$k$-variety.
Let $Z$ denote the center of $G$.
We consider the algebraic group $\Gbar\coloneqq G/Z$.
The algebraic group $\Gbar$ naturally acts on $G$:
\[g  Z(k)\colon x\mapsto g   x   g^{-1}\quad
       \text{for all }g  Z(k)\in \Gbar(k),\, x\in G(k).\]

Let $G_\dmd$ be a $k_0$-model of $G$.
We write $\Gbar_\dmd=G_\dmd/Z_\dmd$, where $Z_\dmd$ is the center of $G_\dmd$.
The $k_0$-model $\Gbar_\dmd$ of $\Gbar$ defines  a semilinear action:
\[\sigma\colon\sG\to \SAut(\Gbar);\]
see~\eqref{e:s-action}.
We have an action of $\sG$ on the group of $k$-points $\Gbar_\dmd(k)$:
\[(\gamma,g)\mapsto \upgam\hm g=\sigma_\gamma(g) \quad
        \text{for } \gamma\in\sG,\, g\in \Gbar(k)=\Gbar_\dmd(k).\]

Let $c\in Z^1(k_0,\Gbar_\dmd)$ be a 1-cocycle.
We denote by $G_0={}_c G_\dmd$ the corresponding inner twisted form of $G_\dmd$;
see Subsection \ref{ss:tw}.
This means that $G_0(k)=G_\dmd(k)$, but the Galois action is twisted by $c$:
\[\sigma^0_\gamma=c_\gamma\circ\sigma_\gamma\quad\text{for all }\gamma\in \sG,\]
where we regard $\Gbar_\dmd(k)$ as a subgroup of $\Aut(G)$.

In this section we assume that there exists
a $G_\dmd${\em-equivariant} $k_0$-model $Y_\dmd$ of $Y$.
We give a criterion for the existence of
a $G_0${\em -equivariant} $k_0$-model $Y_0$ of $Y$,
where $G_0={}_c G_\dmd$.
\end{subsec}

\begin{subsec}\label{ss:horo-2}
We write $[c]\in H^1(k_0,\Gbar_\dmd)$ for the cohomology class of $c$.
We consider the short exact sequence
\[ 1\to Z_\dmd\to G_\dmd\to \Gbar_\dmd\to 1\]
and the corresponding connecting map
\[\delta\colon H^1(k_0,\Gbar_\dmd)\to H^2(k_0,Z_\dmd)\]
from the cohomology exact sequence
\begin{equation}\label{e:coh-e-s}
H^1(k_0,Z_\dmd)\to H^1(k_0,G_\dmd)\to H^1(k_0,\Gbar_\dmd)\labelto{\delta}H^2(k_0,Z_\dmd);
\end{equation}
see Serre \cite[I.5.7, Proposition 43]{ser97}.
We obtain $\delta[c]\in H^2(k_0,Z_\dmd)$.

The $G_\dmd$-equivariant $k_0$-model $Y_\dmd$ of $Y$  defines
an action of $\sG$ on $\sA\coloneqq \Aut^G(Y)$ by
\begin{equation}\label{e:ggg-action}
(\upgam a)(y)=\upgam(a(\hs^{\gamma^{-1}}\!y))
     \quad \text{for all }\gamma\in \sG,\, a\in\sA,\, y\in Y(k).
\end{equation}
We denote by $\sA_\dmd$ the corresponding $\sG$-group.
We obtain a homomorphism
\[\mu\colon \sG\to\SAut(Y), \quad \gamma\mapsto \mu_\gamma,
        \quad\text{where }\mu_\gamma(y)=\upgam y\ \
        \text{for all }\gamma\in\sG,\, y\in Y(k)=Y_\dmd(k),\]
and a homomorphism
\begin{equation}\label{e:tau-tau}
\tau\colon \sG\to\Aut(\sA),\quad \gamma\mapsto\tau_\gamma\hs ,
    \quad\text{where }\tau_\gamma(a)=\upgam a\ \
\text{for all }\gamma\in\sG,\, a\in \sA.
\end{equation}

The center $Z_\dmd\subset G_\dmd$ acts on $Y_\dmd$,
and this action clearly commutes with the action of~$G_\dmd$.
Thus we obtain a canonical $\sG$-equivariant homomorphism
\begin{equation}\label{e:vk-def}
\vk\colon Z_\dmd(k)\to \sA_\dmd.
\end{equation}
\end{subsec}

\begin{subsec}\label{ss:nonab-H2-0}
We need  the nonabelian cohomology set $H^2(\sG,\sA_\dmd)$;
see Springer \cite[Section 1.14]{spr66}.
Set $\Out(\sA)=\Aut(\sA)/\Inn(\sA)$, the group of outer automorphisms of $\sA$.
The homomorphism $\tau\colon \sG\to\Aut(\sA)$ induces
a homomorphism $\kappa\colon\sG\to\Out(\sA)$.
We consider the set of 2-cocycles $Z^2(\sG,\sA,\kappa)$.
By definition, a 2-{\em cocycle} $(f,g)\in Z^2(\sG,\sA,\kappa)$ is a pair of  maps
\[ f\colon\sG\to \Aut(\sA),\quad g\colon \sG\times \sG\to \sA,\]
 such that the pair $(f,g)$
satisfies conditions (5) of  Springer \cite[Section 1.14]{spr66},
which we do not rewrite here.
Here  the map $g$ must be locally constant, and also $f$
must satisfy a certain continuity condition,
which we do not need in this article;
compare Flicker, Scheiderer, and Sujatha \cite[Definition (1.10)]{FSS}.
There is a natural equivalence relation on $Z^2(\sG,\sA,\kappa)$,
see formula (6) in \cite[Section 1.14]{spr66}.
By definition, $H^2(\sG,\sA_\dmd)\coloneqq H^2(\sG,\sA,\kappa)$
is the quotient of  $Z^2(\sG,\sA,\kappa)$ by this equivalence relation.

Recall that an (abelian) 2-cocycle $z\in Z^2(k_0, Z_\dmd)$ is a locally constant map
\[z\colon \sG\times  \sG\to Z_\dmd(k),\quad (\alpha,\beta)\mapsto z_{\alpha,\beta}\]
such that
\[^\alpha\! z_{\beta,\gamma}\cdot z_{\alpha,\beta\gamma}
       =z_{\alpha,\beta}\cdot z_{\alpha\beta,\gamma}\quad
         \text{for all }\alpha,\beta,\gamma\in \sG.\]
Then $\vk_*([z])\in H^2(\sG,\sA_\dmd)$ is by definition
the class of the 2-cocycle $(\tau,\hs\vk\circ z)$, where
$\tau\colon\sG\to\Aut(\sA)$ is as in \eqref{e:tau-tau}.
This class is called \emph{neutral} if there exists
a locally constant map $a\colon \sG \to \sA$
such that
\begin{equation}\label{e:neut}
 a_\gamma\cdot \upgam a_{\beta  }\cdot\vk(z_{\gamma,\beta  })\cdot a_{\gamma\beta  }^{-1}=1
      \quad\text{for all }\gamma,\beta  \in\sG.
\end{equation}
\end{subsec}

\begin{theorem}\label{t:twist}
Let $k$, $G$, $Y$, $k_0$, $G_\dmd$, $Y_\dmd$, $A_\dmd$, and $\delta$
be as in Subsections~\ref{ss:horo-1} and \ref{ss:horo-2}.
In particular, we assume that the $G$-$k$-variety $Y$
admits a $G_\dmd$-equivariant $k_0$-model $Y_\dmd$.
We also assume that $Y$ is quasi-projective.
Let $c\in Z^1(k_0,\Gbar_\dmd)$ be a 1-cocycle,
and consider its class $[c]\in H^1(k_0,\Gbar_\dmd)$.
Set $G_0= {}_c G_\dmd$ (the inner twisted form of $G_\dmd$ defined by the 1-cocycle~$c$).
Then the $G$-variety $Y$ admits a $G_0$-equivariant $k_0$-model
if and only if the cohomology class
\[\vk_*(\delta[c])\in H^2(\sG, \sA_\dmd)\]
is neutral.
\end{theorem}

\begin{proof}
The $k_0$-model $G_\dmd$ of $G$ defines a semilinear action
\[\sigma\colon\sG\to \SAut(G),\quad \gamma \mapsto \sigma_\gamma \hs. \]
The $G_\dmd$-equivariant $k_0$-model $Y_\dmd$ of $Y$
defines a semilinear action
\[\mu\colon \sG\to\SAut(Y),\quad \gamma\mapsto\mu_\gamma\]
such that each $\mu_\gamma$  is $\sigma_\gamma$-equivariant, that is,
\begin{equation}\label{e:mu-g-y}
\mu_\gamma(g\cdot y)=\sigma_\gamma(g)\cdot \mu_\gamma(y)\quad\text{for all }g\in G(k),\, y\in Y(k).
\end{equation}
Since the map $\gamma\mapsto\mu_\gamma$ is a homomorphism, we have
\begin{equation}\label{e:action-hom}
\mu_{\gamma \beta  }=\mu_\gamma\circ\mu_{\beta  }\quad\text{for all }\gamma,\beta  \in\sG.
\end{equation}

Since $c$ is a 1-cocycle, it satisfies the cocycle condition
$c_{\gamma\beta  }=c_\gamma\cdot\upgam c_{\beta  }$,
whence it follows immediately that $c_1=1_{\Gbar}$.
We lift the 1-cocycle $c\colon\sG\to \Gbar(k)$
to a locally constant map
\[\ctil\colon \sG\to G(k)\]
such that $\ctil_1=1_G$.
Note that the map $\ctil$  might not be a 1-cocycle.
Let $\sigma^0\colon \sG\to\SAut(G)$ denote the homomorphism corresponding
to the twisted form $G_0= {}_cG_\dmd$;
then by definition
\[\sigma^0_\gamma(g)=\ctil_\gamma\cdot\sigma_\gamma(g)\cdot\ctil_\gamma^{\hs -1} \quad
\text{for all }\gamma\in\gal,\, g\in G(k)\text{.}\]

For $g\in G(k)$, we write $l(g)$ for the automorphism $y\mapsto g\cdot y$ of $Y$.
We have
\begin{equation}\label{e:g-a}
l(g)\circ a=a\circ l(g)\quad\text{for all } g\in G(k),\, a\in\sA_\dmd\hs ,
\end{equation}
because $a$ is a $G$-equivariant automorphism of $Y$.
By \eqref{e:mu-g-y} we have
\begin{equation}\label{e:mu-g}
\mu_\gamma\circ l(g)=l(\sigma_\gamma(g))\circ\mu_\gamma
    \quad\text{for all }\gamma\in \sG,\, g\in G(k)\text{.}
\end{equation}
Similarly, $\tau_\gamma(a)(\mu_\gamma(y))=\mu_\gamma(a(y))$, hence,
\begin{equation}\label{e:mu-a}
\mu_\gamma\circ a=\tau_\gamma(a)\circ\mu_\gamma
     \quad\text{for all }\gamma\in \sG,\, a\in \sA_\dmd\text{.}
\end{equation}

By definition (Serre \cite[Section I.5.6]{ser97}),
the cohomology class  $\delta[c]\in H^2(k_0, Z_\dmd)$
is the class of the 2-cocycle given by
\[
(\gamma,\beta  )\mapsto \ctil_\gamma\cdot\upgam \ctil_{\beta  }\cdot
\ctil_{\gamma \beta  }^{\hs-1} \in Z_\dmd(k) \quad
  (\gamma,\beta  \in\sG)\text{.}
\]
Then $\vk_*(\delta[c])$ is the class of the 2-cocycle
with first component $\tau$ and second component
\[(\gamma,\beta  )\mapsto \vk( \ctil_\gamma\cdot\upgam \ctil_{\beta  }\cdot
\ctil_{\gamma \beta  }^{\hs-1})\in \sA_\dmd.\]

Let
\[ a\colon \sG\to\sA_\dmd\hs,\quad \gamma\mapsto a_\gamma\]
be any locally constant map.
We define
\begin{equation}\label{e:mu0}
\mu^0_\gamma=a_\gamma\circ l(\ctil_\gamma)\circ\mu_\gamma=
      l(\ctil_\gamma)\circ a_\gamma\circ\mu_\gamma\hs;
\end{equation}
this means that
\[\mu^0_\gamma(y)= a_\gamma(\ctil_\gamma\cdot\mu_\gamma(y)\hs)=
             \ctil_\gamma\cdot a_\gamma(\mu_\gamma(y))\quad\text{for }y\in Y(k).\]
Then $\mu^0_\gamma$ is a $\gamma$-semi-automorphism of $Y$.
\begin{lemma}\label{l:equi}
For any $\gamma\in \sG$, the $\gamma$-semi-automorphism
$\mu^0_\gamma$ is $\sigma^0_\gamma$-equivariant.
\end{lemma}

\begin{proof}
 Using  \eqref{e:mu-g-y} and \eqref{e:g-a}, we compute:
\begin{align*}
\mu^0_\gamma(g\cdot y)
&=(a_\gamma\circ l(\ctil_\gamma))(\mu_\gamma(g\cdot y))\\
&= a_\gamma\big( \ctil_\gamma\cdot\sigma_\gamma(g)\cdot\mu_\gamma(y)\hs\big)
=a_\gamma\big(\ctil_\gamma\hs\sigma_\gamma(g)\hs\ctil_\gamma^{\hs-1}\cdot
\ctil_\gamma\hs\mu_\gamma(y)\hs\big)\\
&=\ctil_\gamma\hs  \sigma_\gamma(g)\hs\ctil_\gamma^{\hs-1}\cdot
a_\gamma(\hs  \ctil_\gamma\cdot \mu_\gamma(y)\hs )
=\sigma^0_\gamma(g)\cdot\mu^0_\gamma(y).\qedhere
\end{align*}
\end{proof}

\begin{lemma}\label{l:hom}
The map $\gamma\mapsto\mu^0_\gamma$  of \eqref{e:mu0} is a homomorphism
if and only if
\begin{equation}\label{e:neutral}
a_\gamma\cdot \upgam a_{\beta  }\cdot\vk( \ctil_\gamma
           \cdot  \upgam \ctil_{\beta  }  \cdot
            \ctil_{\gamma \beta  }^{\hs-1})\cdot a_{\gamma \beta  }^{-1}=1
\end{equation}
for all $\gamma,\beta  \in\sG$.
\end{lemma}

\begin{proof}
Write $b_\gamma=a_\gamma\circ l(\ctil_\gamma)\in\Aut(Y)\subset\SAut(Y)$;
then $\mu^0_\gamma=b_\gamma\circ\mu_\gamma$.
The map $\gamma\mapsto\mu^0_\gamma$ is a homomorphism if and only if
\begin{equation*}
b_\gamma\circ\mu_\gamma\circ b_{\beta  }\circ \mu_{\beta  }
\circ(b_{\gamma\beta  }\circ\mu_{\gamma\beta  })^{-1}=1\quad
\text{for all }\gamma,\beta\in\sG.
\end{equation*}
Since $\mu_{\gamma\beta  }=\mu_\gamma\circ\mu_{\beta  }$, this is equivalent to
\[(b_\gamma\circ\mu_\gamma\circ b_{\beta  }\circ \mu_{\beta  })
\circ(\mu_\beta^{-1}\circ\mu_\gamma^{-1}\circ b_{\gamma\beta}^{-1})=1\]
and to
\begin{equation}\label{e:cocycle}
 b_\gamma\circ(\mu_\gamma\circ b_{\beta } \circ\mu_\gamma^{-1})\circ b_{\gamma\beta}^{-1}=1.
\end{equation}
Since $b_\gamma=a_\gamma\circ l(\ctil_\gamma)$, this is equivalent to
\[
a_\gamma\circ l(\ctil_\gamma)\circ\big(\mu_\gamma\circ a_{\beta  }
\circ\mu_\gamma^{-1}\big)\circ\big(\mu_\gamma\circ l(\ctil_{\beta  })\circ\mu_\gamma^{-1}\big)
\circ\big(a_{\gamma\beta  }\circ l(\ctil_{\gamma\beta  })\hs\big )^{-1}=1\text{.}
\]
Taking into account \eqref{e:mu-a} and  \eqref{e:mu-g}, this is equivalent to
\[a_\gamma\circ l(\ctil_\gamma)\circ\tau_\gamma(a_{\beta  })
   \circ l(\sigma_\gamma(\ctil_{\beta  }))
   \circ \big(a_{\gamma\beta  }\circ l(\ctil_{\gamma\beta  })\hs \big)^{-1}=1.\]
Writing $\upgam a_{\beta  }=\tau_\gamma(a_{\beta })$ and
$\upgam\ctil_{\beta  }=\sigma_\gamma(\ctil_{\beta })$,
and taking into account that  by \eqref{e:g-a} we have
$l(\ctil_\gamma)\circ \upgam a_{\beta  }=\upgam a_{\beta  }\circ l(\ctil_\gamma)$,
this is equivalent to
\[a_\gamma\cdot \upgam a_{\beta  }\cdot l( \ctil_\gamma\upgam \ctil_{\beta  }\,
            \ctil_{\gamma \beta  }^{\hs-1})\cdot a_{\gamma \beta  }^{-1}=1,\]
which is \eqref{e:neutral} since $\ctil_\gamma\upgam \ctil_{\beta  }\,
            \ctil_{\gamma \beta  }^{\hs-1} \in Z_\dmd(k)$.
Thus the map $\gamma\mapsto \mu^0_\gamma$ is a homomorphism if and only if
\eqref{e:neutral} holds for all $\gamma,\beta  \in \sG$,
which completes the proof of Lemma~\ref{l:hom}.
\end{proof}

We resume proving Theorem~\ref{t:twist}.
Assume that $\vk_*(\delta[c])\in H^2(\sG,\sA_\dmd)$ is neutral.
This means that there exists a locally constant map
\[a\colon \sG\to \sA_\dmd\]
such that \eqref{e:neutral} holds.
Then by Lemma~\ref{l:hom} the map
\[\mu^0\colon \sG\to\SAut(Y),\quad \gamma\mapsto\mu^0_\gamma\]
is a homomorphism.
We see that $\mu^0$ is a semilinear action of $\sG$ on $Y$.
By Lemma~\ref{l:equi}, the semilinear action  $\mu^0$ is $\sigma^0$-equivariant.

\begin{lemma}\label{l:comes-from-k1}
With the above assumptions and notation, there exists a finite Galois extension $k_1/k_0$ in $k$
such that the restriction of the map $\mu^0\colon \sG\to \SAut(Y)$ to $\Gal(k/k_1)$
comes from some $G_1$-equivariant $k_1$-model $Y_1$ of $Y$, where $G_1=G_0\times _{k_0} k_1$.
\end{lemma}

\begin{proof}
Since the map $\ctil$ is locally constant and $\ctil_{1}=1_G$,
there exists a normal open subgroup $\sU\subset\sG$ such that $\ctil|_\sU=1_{G}$.
Then for every $\gamma,\beta  \in \sU$ we have
$\ctil_\gamma\hs \upgam\ctil_{\beta  }\hs \ctil_{\gamma\beta  }^{\hs-1}=1$,
hence $\vk(\ctil_\gamma\hs \upgam\ctil_{\beta  }\hs \ctil_{\gamma\beta  }^{\hs-1})=1$,
and from \eqref{e:neutral}
we obtain that
\[ a_\gamma\hs \upgam\hm a_{\beta  }\hs\hs  a_{\gamma\beta  }^{-1}=1.\]
This means that the restriction of $a$ to $\sU$ is a 1-cocycle, and hence, $a_{1}=1_\sA$.
Since the map $\gamma\mapsto  a_\gamma$ is locally constant,
there exists an open subgroup $\sU_1\subset \sU$, normal in $\sG$,
such that $\ctil|_{\sU_1}=1_G$ and $a|_{\sU_1}=1_\sA$.
Then by formula \eqref{e:mu0}, for all $\gamma\in  \sU_1$ we have $\mu_\gamma^0=\mu_\gamma$.

Write  $k_1=k^{\sU_1}$; then $k_1/k_0$ is a finite Galois extension and  $\sU_1=\Gal(k/k_1)$.
We see that  for $\gamma\in  \sU_1=\Gal(k/k_1)$, the semi-automorphism $\mu^0_\gamma$
comes from the $G_1$-equivariant
$k_1$-model $Y_1\coloneqq Y_\dmd\times _{k_0} k_1$ of $Y$,
where  $G_1=G_\dmd\times _{k_0} k_1=G_0\times _{k_0} k_1$.
This completes the proof of the lemma.
\end{proof}

We resume proving Theorem~\ref{t:twist}.
By Lemma~\ref{l:comes-from-k1}, the $\sigma^0$-equivariant
semilinear action $\mu^0$ is algebraic.
Since by assumption  $Y$ is quasi-projective,
by Lemma \ref{p:lemma-5-4} the variety $Y$
admits a $G_0$-equivariant $k_0$-model $Y_0$
inducing the semilinear action $\mu^0$, as required.

Conversely, assume that there exists a $G_0$-equivariant $k_0$-model $Y_0$ of $Y$.
Since $\ctil_{1}=1_G$, there exists a finite Galois extension $k_1/k_0$ in $k$
such that $\ctil_\gamma=1_G$ for all $\gamma\in  \sU_1\coloneqq \Gal(k/k_1)$.
Set $G_1=(G_\dmd)_{k_1}$; then $(G_0)_{k_1}=G_1$.
Set $Y_1=(Y_\dmd)_{k_1}$; then $(Y_0)_{k_1}$ and $Y_1$
are two $G_1$-varieties over $k_1$, and they become $G$-isomorphic over $k$.
By \cite[Lemma 5.6(ii)]{BG}
there exist a finite Galois extension $k_2/k_1$ in $k$
and a $(G_1)_{k_2}$-equivariant isomorphism
\[\varphi_2\colon (Y_0)_{k_2}\isoto (Y_\dmd)_{k_2}.\]
Then $\varphi_2$ induces a $G$-equivariant $k$-isomorphism
\[\varphi\colon  (Y_0)_{k}\isoto (Y_\dmd)_{k}=Y.\]
The $k_0$-model $Y_0$ of $(Y_0)_k$ defines a homomorphism
\[\mu^0\colon \sG\lra\SAut\,(Y_0)_k\labelto{\varphi_*}\SAut(Y)\]
such that $\mu^0_\gamma=\mu_\gamma$ for all $\gamma\in  \sU_2\coloneqq \Gal(k/k_2)$
(because $\varphi_2$ is a $k_2$-isomorphism).

For every $\gamma\in\sG$, set
\[b_\gamma=\mu^0_\gamma\circ(\mu_\gamma)^{-1}\in\Aut_k(Y)\subset\SAut(Y).\]
Since  $\mu^0$ is a homomorphism, by \eqref{e:cocycle} we have
\begin{equation}\label{e:b-cocycle}
b_{\gamma\beta }=b_\gamma\circ (\mu_\gamma\circ b_{\beta }\circ\mu_\gamma^{-1})
\quad\text{for all }\gamma,\beta\in\sG.
\end{equation}
For $\beta  \in \sU_2$ we have $\mu^0_{\beta  }=\mu_{\beta  }$, hence $b_{\beta  }=\id_Y$.
From the cocycle formula \eqref{e:b-cocycle} we see
that the map $\gamma\mapsto b_\gamma$ is locally constant.
Set
\[a_\gamma=b_\gamma\hs  \circ l(\ctil_\gamma)^{-1}\in\Aut_k(Y)\text{;}\]
then the map $a\colon \gamma\mapsto a_\gamma$ is locally constant,
because both maps $b\colon \gamma\mapsto b_\gamma$ and $\ctil\colon \gamma\mapsto \ctil_\gamma$
are locally constant.
We have
\[\mu^0_\gamma = b_\gamma\circ\mu_\gamma = a_\gamma\circ l(\ctil_\gamma)\circ\mu_\gamma\hs .\]
By Lemma~\ref{l:equi}, the $\gamma$-semi-automorphism
$l(\ctil_\gamma)\circ\mu_\gamma$ is $\sigma^0_\gamma$-equivariant.
Since the $\gamma$-semi-automorphism $\mu^0_\gamma$ is $\sigma^0_\gamma$-equivariant as well,
we see that $a_\gamma$ is a $G$-equivariant $k$-automorphism of $Y$,
that is, $a_\gamma\in \Aut^G(Y)=\sA$.
Since the map $\gamma\mapsto \mu^0_\gamma$ is a homomorphism,
by Lemma~\ref{l:hom} the equality \eqref{e:neutral} holds,
and hence $\vk_*(\delta[c])$ is neutral in $H^2(\sG, \sA_\dmd)$.
This completes the proof of Theorem \ref{t:twist}.
\end{proof}

Note that while proving Theorem \ref{t:twist}, we actually proved the
following result:

\begin{proposition}\label{p:twist-0}
  Let $k$, $G$, $Y$, $k_0$, $G_\dmd$, $A_\dmd$, and $\delta$ be as in
  Subsections~\ref{ss:horo-1} and \ref{ss:horo-2}, but instead of
  assuming that the $G$-$k$-variety $Y$ admits a $G_\dmd$-equivariant
  $k_0$-model $Y_\dmd$, we assume only that $Y$ admits an algebraic
  $\sigma$-equivariant semilinear action
  \[\mu\colon \sG\to\SAut(Y),\]
  where $\sigma\colon\sG\to\SAut(G)$ is the semilinear action on $G$
  defined by the $k_0$-model $G_\dmd$ of $G$. We do not assume that
  $Y$ is quasi-projective. Let $c\in Z^1(k_0,\Gbar_\dmd)$ be a
  1-cocycle, and consider its class $[c]\in H^1(k_0,\Gbar_\dmd)$. Set
  $G_0= {}_c G_\dmd$, and let $\sigma^0\colon\sG\to\SAut(G)$ denote
  the corresponding semilinear action on $G$. Then the $G$-variety
  $Y$ admits an algebraic $\sigma^0$-equivariant semilinear action
  $\mu^0\colon \sG\to\SAut(Y)$ if and only if the cohomology class
  \[\vk_*(\delta[c])\in H^2(\sG, \sA_\dmd)\]
  is neutral.
\end{proposition}

\section{Model of a homogeneous space of a reductive group}
\label{s:tits}

Let $k_0$, $k$, and $\sG$ be as in Subsection \ref{ss:mod-Y}.
In this section $G$ is a (connected) reductive group over $k$.
We need the following result:

\begin{proposition}\label{p:inn-qs}
Let $G$ be a  reductive group over $k$, and
let $G_0$ be any $k_0$-model of~$G$.
Then there exist a quasi-split
$k_0$-model $G_\qs$ of $G$
 and a cocycle $d\in Z^1(k_0,\Inn(G_0))$ such that
$G_\qs\simeq {}_d G_0$ (we say that $G_\qs$ is a quasi-split inner $k_0$-form of $G_0$).
Moreover, if $G_\qs$ and $G'_\qs$ are two
quasi-split inner $k_0$-forms of $G_0$, then they are isomorphic,
and if $\smash{d,d'\in Z^1(k_0,\Inn(G_0))}$ are two such cocycles,
then they are cohomologous.
\end{proposition}

\begin{proof}
Let $\Gtil_0$ denote the universal cover of the commutator subgroup $[G_0,G_0]$ of $G_0$.
Then $\Gtil_0$ is a simply connected semisimple $k_0$-group, and $\Inn(\Gtil_0)=\Inn(G_0)$.
By \enquote{The Book of Involutions} \cite[Propositions (31.5) and (31.6)]{KMRT},
there exists a unique (up to isomorphism) quasi-split $k_0$-model $\Gtil_\qs$ of $\Gtil$
and a unique (up to equivalence) 1-cocycle $d\in Z^1(k_0,\Inn(\Gtil_0))=Z^1(k_0,\Inn(G_0))$
such that $\Gtil_\qs\simeq {}_d \Gtil_0$.
Then $G_\qs\coloneqq {}_d G_0$ is the required unique quasi-split inner form of $G_0$.
\end{proof}

\begin{subsec}\label{ss:Tits}
Let $G$ be a  reductive group over $k$, and let $G_0$ be a $k_0$-model of $G$.
Write $\Gbar=G/Z(G)$ for the corresponding adjoint group, and $\Gtil$ for the universal cover of
the connected semisimple group $[G,G]$.
We fix $d\in Z^1(k_0, G_0/Z(G_0))$ as in Proposition~\ref{p:inn-qs};
then   $G_\qs\coloneqq {}_d G_0$ is a quasi-split $k_0$-model of $G$.
We write $\Gbar_0=G_0/Z(G_0)$ and $\Gbar_\qs=G_\qs/Z(G_\qs)$.

We write $\Ztil_\qs$ for the center $Z(\Gtil_\qs)$ of the universal cover
$\Gtil_\qs$ of the connected semisimple group $[G_\qs,G_\qs]$.
Similarly, we write $\Ztil_0$ for the center $Z(\Gtil_0)$ of the universal cover
$\Gtil_0$ of the connected semisimple group $[G_0,G_0]$.
The short exact sequence
\[1\to \Ztil_0\to \Gtil_0 \to \Gbar_0\to 1\]
induces a cohomology exact sequence
\[H^1(k_0,\Ztil_0)\to  H^1(k_0,\Gtil_0)\to H^1(k_0,\Gbar_0)\labelto{\deltatil_0} H^2(k_0,\Ztil_0).\]
Similarly, the short exact sequence
\[1\to \Ztil_\qs\to \Gtil_\qs \to \Gbar_\qs\to 1\]
induces a cohomology exact sequence
\[H^1(k_0,\Ztil_\qs)\to  H^1(k_0,\Gtil_\qs)\to H^1(k_0,\Gbar_\qs)
    \labelto{\deltatil_\qs} H^2(k_0,\Ztil_\qs).\]
By definition, the \emph{Tits class} $t(\Gtil_0)\in H^2(k_0,\Ztil_0)$ is defined by
\[t(\Gtil_0)=(\deltatil_0[d])^{-1},\]
the inverse of the image of the cohomology class $[d]\in H^1(k_0,\Gbar_0)$
under the connecting map $\deltatil_0\colon H^1(k_0,\Gbar_0)\to H^2(k_0,\Ztil_0)$.
See \cite[Section~31, before (31.7)]{KMRT},
where the Tits class $t(\Gtil_0)$ is denoted by $t_{\Gtil_0}$.

Set $c_\gamma=d_\gamma^{-1}\in\Gbar(k)$ for all $\gamma\in \sG$.
Let
\[\mu^0\colon \sG\to \SAut(G)\quad\text{and}\quad\mu^\qs\colon \sG\to \SAut(G)\]
be the semilinear actions corresponding to the $k_0$-models
$G_0$ and $G_\qs$\hs, respectively. Then
\[ \mu^\qs_\gamma=\inn(d_\gamma)\circ\mu^0_\gamma\hs ,\ \ \text{and hence,}\ \
\mu^0_\gamma=\inn(c_\gamma)\circ\mu^\qs_\gamma\ \ \text{for all } \gamma\in \sG,\]
where we write  $\inn(g)$ for the inner automorphism of $G$
defined by an element $g\in \Gbar(k)$.
It follows that $c$ is a 1-cocycle, that is,  $c\in Z^1(k_0,G_\qs)$\hs .
\end{subsec}

\begin{lemma}\label{l:Tits-qs}
With the above assumptions and notation we have
\[\deltatil_\qs[c]=t(\Gtil_0)\in H^2(k_0,\Ztil_\qs),\]
where we identify $\Ztil_\qs$ with $\Ztil_0$.
\end{lemma}

\begin{proof} A straightforward calculation.\end{proof}

Let $H\subset G$ be an algebraic subgroup (not necessarily spherical).
We consider the homogeneous $G$-variety $Y=G/H$. Consider the abstract
group $\sA=\Aut^G(G/H)$ and the algebraic group $A=\sN_G(H)/H$; then
there is a canonical isomorphism $\smash{A(k)\isoto\sA}$ (see, for instance,
\cite[Corollary~4.3]{BG}). Let $G_\qs$ be a quasi-split $k_0$-model of $G$
and assume that there exists a $G_\qs$-equivariant model of $G/H$ of
the form $G_q/H_q$, where $H_q\subset G_q$ is a $k_0$-subgroup. We set
$A_q=\sN_{G_q}(H_q)/H_q$, which is a $k_0$-model of $A$.
We have a canonical homomorphism
\begin{equation}\label{e:vktil}
\vktil\colon\, \Ztil_\qs\labelto{\pi} Z_\qs\labelto{\vk_\qs} A_\qs\hs ,
\end{equation}
where the homomorphism $\pi\colon \Ztil_\qs\to Z_\qs$
is induced by the canonical homomorphism $\Gtil_\qs\to G_\qs$,
and $\vk_\qs\colon Z_\qs\to A_\qs$ is the homomorphism of $k_0$-groups
inducing the homomorphism  of groups of $k$-points
$\vk\colon Z_q(k)\to\Am=A_\qs(k)$  of \eqref{e:vk-def}.
By abuse of notation we shall write $\vk$ for $\vk_\qs$\hs.

\begin{proposition}\label{p:tits}
Let $k_0$, $k$, and $\sG$ be as in Subsection \ref{ss:mod-Y}.
  Let $G$ be a (connected) reductive group over $k$. Let $H\subset G$ be an
  algebraic subgroup. Let $G_0$ be a $k_0$-model of $G$.
  Write $G_0= {}_c G_\qs$, where $G_\qs$ is a quasi-split inner form
  of $G_0$ and where $c\in Z^1(k_0,\Gbar_\qs)$. Assume that $G/H$
  admits a $G_\qs$-equivariant $k_0$-model of the form $G_q/H_q$,
  where $H_q\subset G_q$ is a $k_0$-subgroup.
  Set  $\smash{A_q=\sN_{G_q}(H_q)/H_q}$.
  Then $G/H$ admits  a $G_0$-equivariant $k_0$-model if and only if the image in
  $\smash{H^2(k_0,A_\qs)}$ of the Tits class
  \[t(\Gtil_0)\in H^2(k_0,Z(\Gtil_0))=H^2(k_0,Z(\Gtil_\qs))\] under the map
\[\vktil_*\colon\, H^2(k_0,Z(\Gtil_\qs))\labelto{\pi_*}
    H^2(k_0,Z(G_\qs))\labelto{\vk_*}  H^2(k_0,A_\qs)\]
 is neutral.
\end{proposition}

This partial result  will be used in Subsection \ref{ss:proof-main}
in the proof of Theorem~\ref{t:sphfull'}.

\begin{proof}
By Theorem \ref{t:twist}, the homogeneous space $G/H$
admits a $G_0$-equivariant $k_0$-model
if and only if the image
\[\vk_{*}(\delta_\qs[c])\in H^2(k_0,A_\qs)\]
is neutral.
We write  $Z_\qs$ for $Z(G_\qs)$ and  $\Ztil_\qs$ for $Z(\Gtil_\qs)$.
From the commutative diagram with exact rows
\[
\xymatrix{
1 \ar[r] & \Ztil_\qs \ar[r]\ar[d]_\pi & \Gtil_\qs \ar[r]\ar[d] & \Gbar_\qs \ar[r]\ar[d]^{\id} & 1 \\
1 \ar[r] & Z_\qs \ar[r]           &  G_\qs \ar[r]          & \Gbar_\qs \ar[r]             & 1
}
\]
we obtain a commutative diagram
\[
\xymatrix{
H^1(k_0,\Gbar_\qs)\ar[r]^{\deltatil_\qs} \ar[d]_{\id}  & H^2(k_0, \Ztil_\qs)\ar[d]^{\pi_*} \\
H^1(k_0,\Gbar_\qs)\ar[r]^{\delta_\qs}                                   & H^2(k_0, Z_\qs)\mathrlap{,}
}
\]
which shows that
\[ \delta_\qs[c]=\pi_*(\deltatil_\qs[c]).\]
By Lemma~\ref{l:Tits-qs} we have
\[t(\Gtil_0)=\deltatil_\qs[c]\in  H^2(k_0,\Ztil_\qs).\]
Thus
\[\vk_{*}(\delta_\qs[c])=\vk_*(\pi_*(\deltatil_\qs[c]))=\vktil_*(\deltatil_\qs[c])
    =\vktil_*(t(\Gtil_0))\in H^2(k_0,A_\qs)\text{.}\]
We conclude that the homogeneous space $G/H$ admits
a $G_0$-equivariant $k_0$-model if and only if
$\vktil_*(t(\Gtil_0))\in H^2(k_0,A_\qs)$ is neutral, as required.
\end{proof}

\section{Models of affine spherical varieties}
\label{s:sphaff}
Let $k_0$, $k$, and $\sG$ be as in Subsection \ref{ss:mod-Y}.
Let $G$ be a (connected) reductive group over $k$.
We fix a Borel subgroup $B \subset G$ and a
maximal torus $T \subset B$. Let $U$ denote the unipotent radical of $B$;
then $B=U\rtimes T$.

\begin{subsec}
For any affine $G$-variety $X$, its \emph{weight monoid}
$\Gamma \subset \X^*(B)$ is defined to be the monoid consisting of
those $\lambda \in \X^*(B)$ for which there exists a nonzero
$B$-$\lambda$-semi-invariant regular function in the coordinate ring $k[X]$.

Now let $\Gamma \subset \X^*(B)$ be a submonoid of dominant weights.
For each $\lambda\in\Gamma$ we denote by $V_\lambda$ an irreducible
$G$-module of highest weight $\lambda \in \X^*(B)$. For any affine
\emph{spherical} $G$-variety $X$ with weight monoid $\Gamma$ we have
\begin{align*}
  k[X] \cong \bigoplus_{\lambda \in \Gamma} V_\lambda \quad\text{ and }\quad
  k[X]^U \cong \bigoplus_{\lambda \in \Gamma} V_\lambda^U
\end{align*}
because the coordinate ring $k[X]$ of an affine spherical $G$-variety
$X$ is a \emph{multiplicity-free} rational representation of the
reductive group $G$ (see, for example, Perrin
\cite[Theorem~2.1.2]{per14}). In this case, the monoid $\Gamma$ is
finitely generated and moreover saturated, that is, $\Gamma$ is the
intersection of the rational cone spanned by $\Gamma$ with the lattice
generated by $\Gamma$; see, for instance, Avdeev and Cupit-Foutou \cite[Subsection~2.2]{acf18}.
\end{subsec}

\begin{subsec}
Let $X$ be an affine spherical $G$-variety.
We write $\smash{k[X]^{(B)}}$ for the monoid of nonzero
$B$-semi-invariant regular functions on $X$.
We write $k[X]^U$ for the ring of $U$-invariants in $k[X]$; then $T$ acts on $k[X]^U$.
We write $\smash{(k[X]^U)^{(T)}}$ for the monoid
of nonzero $T$-semi-invariants in $k[X]^U$;
then clearly $\smash{(k[X]^U)^{(T)}=k[X]^{(B)}}$.
\end{subsec}

\begin{subsec}
  Up to $G$-equivariant isomorphism, there may be several different
  affine spherical $G$-varieties $X$ with given weight monoid
  $\Gamma$; in other words there may be several multiplicative
  structures on the $G$-vector space
  $\bigoplus_{\lambda\in\Gamma}V_\lambda$ (compatible with the
  $G$-action).

  On the other hand, the ring of $U$-invariants $k[X]^U$ is always
  (non-canonically) $T$-equivariantly isomorphic to the semigroup
  algebra $k[\Gamma]$ as a $k$-algebra. To construct an isomorphism of
  $k$-algebras $\smash{k[\Gamma] \to k[X]^U}$, choose a point $x$ in
  the open $B$-orbit in $X$ and for every $\lambda \in \Gamma$ define
  $\smash{f_\lambda \in V_\lambda^{(B)} = V_\lambda^U \subset k[X]^U}$
  to be the unique $B$-$\lambda$-semi-invariant function such that
  $f_\lambda(x) = 1$. We obtain a homomorphism of monoids
  \[\Gamma\longrightarrow k[X]^{(B)}=(k[X]^U)^{(T)},
        \qquad \lambda\mapsto f_\lambda\hs ,\]
  which uniquely extends by linearity to a homomorphism of
  $k$-algebras $k[\Gamma]\isoto k[X]^U$, and this homomorphism is an
  isomorphism because every $V_\lambda^U$ is $1$-dimensional.
  Note that for different choices of $x$, we may get
  different isomorphisms of $k$-algebras $k[\Gamma] \to k[X]^U$, but
  the induced isomorphism of monoids
  \[\Gamma \to (k[X]^U)^{(T)}/k^*, \qquad \lambda\mapsto [f_\lambda] \]
  will be the same.
\end{subsec}

\begin{subsec}
  Given $\Gamma$, Alexeev and Brion \cite{ab05}, see also Brion \cite{bri13},
  defined an affine moduli scheme $M_\Gamma$ which parameterizes
  varieties $X$ as above considered with an additional structure. We explain it
  as in Avdeev and Cupit-Foutou \cite[Section~2]{acf18}. The $k$-points of $M_\Gamma$
  correspond to equivalence classes $[X, \tau]$ of pairs $(X, \tau)$
  where $X$ is an affine spherical $G$-variety with weight monoid
  $\Gamma$ and the map $\tau: k[X]^U \to k[\Gamma]$ is a
  $T$-equivariant isomorphism of $k$-algebras (there exists at least
  one such map by the above discussion). Two pairs $(X_1, \tau_1)$ and
  $(X_2, \tau_2)$ are equivalent if there is a $G$-equivariant
  isomorphism of $k$-algebras $\varphi\colon k[X_1] \isoto k[X_2]$
  such that the following diagram commutes:
  \begin{align*}
  \xymatrix{
    k[X_1]^U \ar[rd]_{\tau_1} \ar[rr]^{\varphi_*} & & k[X_2]^U \ar[ld]^{\tau_2} \\
    & k[\Gamma]\text{\rlap{.}}
  }
  \end{align*}
  Since the $k$-algebra $k[X_1]$ is generated by $G(k)\cdot k[X_1]^U$
  (even as a vector space), every $G$-equivariant isomorphism of
  $k$-algebras $k[X_1] \isoto k[X_2]$ is uniquely determined by its
  restriction to the subalgebra of $U$-invariants
  $k[X_1]^U \isoto k[X_2]^U$. It follows that the additional structure
  $\tau$ guarantees that for any two equivalent pairs $(X_1, \tau_1)$ and
  $(X_2, \tau_2)$, there is a \emph{unique} isomorphism of
  $G$-$k$-varieties $X_1 \isoto X_2$ that respects this additional
  structure (while if $X_1$ has non-trivial $G$-equivariant
  automorphisms, then there is more than one isomorphism $X_1 \isoto X_2$
  not necessarily respecting $\tau$).
\end{subsec}

\begin{subsec}\label{ss:G-B-X}
  Let $X$ be an affine spherical $G$-variety, and let
  $\Gamma=\Gamma(X) \subset \X^*(B)$ denote
    its weight monoid.
  Let $G_0$ be a quasi-split $k_0$-model $G_0$ of $G$,
  and let $B_0\subset G_0$ be a Borel subgroup defined over $k_0$.
  Let $T_0$ be a maximal torus in $B_0$,
  and let $U_0$ be the unipotent radical of $B_0$. We choose
  $B$, $T$, and $U$ to be the base changes from $k_0$ to $k$
  of $B_0$, $T_0$, and $U_0$,  respectively.

  The $k_0$-model $B_0$ of $B$ defines a $\gal$-action on $\X^*(B)$.
  This $\gal$-action coincides with the $\gal$-action on $\X^*(B)$
  defined by the $k_0$-model $G_0$ of $G$ in \cite[Section 7.1]{BG}.
 \end{subsec}

 \begin{subsec}\label{ss:gamma*-G-B}
Let $G$, $B$, and  $X$ be as in  \ref{ss:G-B-X}, and let $\gamma\in\gal$.
Consider the $k$-variety $\upgam X$; see \cite[Section 1.1]{BG},
where $\upgam X$ is denoted by $\gamma_*X$.
We may assume that our affine $k$-variety $X$ is embedded into the affine space $\A^n_k$
for some natural number $n$.
Then
\[\upgam X=\{\upgam x\hs\mid\hs x\in X(k)\subset A^n(k)=k^n\},\]
where by abuse of notation we identify $\upgam X$ with $(\upgam X)(k)$;
see~\cite[Section~1.15]{BG}.

Similarly, we define  the algebraic $k$-groups $\upgam G$ and $\upgam B$
naturally acting on $\upgam X$; see~\cite[Construction~2.6]{BG}.
The compatible $k_0$-models $G_0$ of $G$ and $B_0$ of $B$
define a $\gamma$-semilinear automorphism $\sigma_\gamma\colon G\to G$ preserving $B$.
By the definition of a $\gamma$-semilinear automorphism,
$\sigma_\gamma$ is the same as a $k$-isomorphism of algebraic $k$-groups
$\upgam G\to G$, which restricts to a $k$-isomorphism $\upgam B\to B$;
see \cite[Section 1.9, formula (9) and Construction 2.6]{BG}.
Thus we obtain canonical structures of a $G$-variety and a $B$-variety
on the affine $k$-variety $\upgam X$.
\end{subsec}

\begin{lemma}\label{l:Gamma(upgam-X)}
With the notation and assumptions of \ref{ss:G-B-X} and \ref{ss:gamma*-G-B},
let $\Gamma(X)\subset\X^*(B)$ and  $\Gamma(\upgam X)\subset\X^*(B)$
denote the weight monoids of the affine $G$-varieties $X$ and $\upgam X$,
respectively. Then $\Gamma(\upgam X)=\upgam \Gamma(X)$.
\end{lemma}

\begin{proof}
We write $X'=\upgam X$, $\Gamma=\Gamma(X)$, $\Gamma'=\Gamma(X')$.
Let $\lambda\in\Gamma$, and let $f\in k[X]$ be a nonzero regular function such that
\[f(b^{-1}\cdot x)=\lambda(b)\cdot f(x)\quad\text{for all }\,b\in B(k),\ x\in X(k).\]
The $\gamma$-semi-isomorphism $X\to X'=\upgam X$ (see \cite[Example 1.10]{BG})
sends our $f\in k[X]$ to the nonzero regular function $f'\in k[X']$ satisfying
\[f'(x')=\gamma(f(\upgamM\hm x'))\quad \text{for all }\,x'\in X'(k)\subset \A^n_k(k)=k^n \]
(where we assume that $X$ and $X'$ are embedded into $\A^n_k$).
We compute:
\begin{align*}
f'(b^{-1}\cdot x')&=\gamma(f(\upgamM b^{-1}\cdot\upgamM\hm x'))
=\gamma(\lambda(\upgamM b)\cdot f(\upgamM\hm x'))\\
&=\gamma(\lambda(\upgamM b))\cdot\gamma( f(\upgamM\hm x'))
=(\upgam\lambda)(b)\cdot f'(x')\quad\text{for all }\,b\in B(k),\ x'\in X'(k).
\end{align*}
Thus $\upgam\lambda\in\Gamma'$ and $\upgam\Gamma\subset\Gamma'$.
Applying the same argument to the $\gamma^{-1}$-equivariant
semi-isomorphism $X'\to X$, we obtain that $\upgamM\hs\Gamma'\subset\Gamma$.
Thus $\Gamma'=\upgam\Gamma$, as required.
\end{proof}

\begin{corollary}\label{c:gal-preserves}
For $G$, $G_0$, $B$, $B_0$, and $X$ as in  \ref{ss:G-B-X} and \ref{ss:gamma*-G-B},
assume that $X$ admits a $G_0$-equivariant $k_0$-model.
Then the $\gal$-action on $\X^*(B)$ preserves $\Gamma(X)$.
\end{corollary}

\begin{proof}
Since $X$ admits a $G_0$-equivariant $k_0$-model,
for any $\gamma\in\gal$ there exists a $G$-equivariant  $k$-isomorphism
$\upgam X\isoto X$.  It follows that $\Gamma(\upgam X)=\Gamma(X)$.
Since by Lemma~\ref{l:Gamma(upgam-X)} we have $\Gamma(\upgam X)=\upgam\Gamma(X)$,
we obtain that $\upgam\Gamma(X)=\Gamma(X)$, as required.
\end{proof}

  Conversely, assume that
  the $\gal$-action preserves $\Gamma=\Gamma(X)$. Then we
  obtain a natural semilinear $\gal$-action on $k[\Gamma]$.
  For $\gamma\in\gal$ we denote by  $\gamma_*\colon k[\Gamma]\to k[\Gamma]$
  the corresponding $\gamma$-semilinear  map.

  \begin{construction}\label{p:action-action}
    Assume that the $\sG$-action preserves $\Gamma$. We construct an
    action of $\sG$ on the set of $k$-points of $M_\Gamma$. Let
    $\gamma\in\sG$ and let $[X',\tau']\in M_\Gamma(k)$. We construct
    \[[X'', \tau''] \coloneqq {}^\gamma[X', \tau']\] as follows:
   First, the $k$-variety $(X'',p'') \coloneqq \gamma_*(X',p')$ is obtained from
    $(X',p')$ by the base change
    \begin{align*}
      \xymatrix{
      X'' \ar[r]^b\ar[d]_{p''} & X' \ar[d]^{p'} \\
      \Spec k\ar[r]^{\gamma^*} & \Spec k\text{\rlap{,}}
                                 }
    \end{align*}
    where we may take $X''=X'$, $b=\id_{X'}$, and
    $p''=(\gamma^*)^{-1}\circ p'$; see \cite[Lemma 1.2]{BG}. Moreover,
    the action of $G$ on $X''$ is obtained from the action of $G$ on
    $X'$ as follows:
   \begin{align*}
      \xymatrix{
      G \times  X'' \ar[rr]^{\sigma_\gamma^{-1} \times  b} && G \times
      X' \ar[rr] && X' \ar[rr]^{b^{-1}} && X''\text{\rlap{.}}                                                                                                        }
    \end{align*}
    Finally, the map $\tau''\colon k[X'']^U \to k[\Gamma]$ is obtained from the
    following commutative diagram:
    \begin{align*}
      \xymatrix{
      k[X']^U \ar[r]^-{\tau'} \ar[d]_-{b^*} & k[\Gamma] \ar[d]^-{\gamma_*} \\
      k[X'']^U \ar[r]^-{\tau''} & k[\Gamma]\text{\rlap{.}}
                                  }
    \end{align*}
  \end{construction}

  \begin{prop}
    \label{prop:actmg}
    Assume that the $\sG$-action preserves $\Gamma$. Then there exists
    a semilinear action of $\gal$ on $M_\Gamma$ for which the action
    on $k$-points
  is given in Construction~\ref{p:action-action}. This semilinear
  action defines a $k_0$-model $\MGz$ of $M_\Gamma$.
  \end{prop}

  \begin{proof}
    We prove the proposition in four steps.

    \emph{Step 1.} Consider the $G \times B$-action on $k[G]$, where
    $G$ acts from the left and $B$ acts from the right, that is, for
    $(g, b) \in G \times B$ and $f \in k[G]$ we have
    $(g, b)\cdot f(x) = f(g^{-1}x b)$. Moreover, consider the
    $G \times B$-action on $k[\Gamma]$ where $G$ acts trivially and
    $B$ acts via the isomorphism $B/U \cong T$,
    where $t\in T(k)$ acts on the subspace $k\cdot\lambda\subset k[\Gamma]$
    by multiplication by $\lambda(t)$.
        The $G$-module
    \[ V_{k[\Gamma]} = \Ind_B^G k[\Gamma] \coloneqq (k[G] \otimes_k
      k[\Gamma])^B\] is called the induced representation; see, for
    instance, Timashev \cite[Section 2.2]{tim11}. Then
  \[k[\Gamma] = (V_{k[\Gamma]})^U\subset V_{k[\Gamma]}\text{,}\] and
  hence for each $\lambda\in\Gamma$ we have
\[\lambda\in k[\Gamma]\subset  V_{k[\Gamma]}.\]
Changing the notation, we now denote by $V_\lambda$ the $k$-span of
the set $G(k)\cdot\lambda$ in $V_{k[\Gamma]}$, that is, the
subrepresentation of $G$ in $V_{k[\Gamma]}$ generated by $\lambda$.
Then we have
\[ V_{k[\Gamma]} = \bigoplus_{\lambda \in \Gamma} V_{\lambda}.\] As
before, $V_\lambda$ is an irreducible $G$-module of highest weight
$\lambda$, but now $V_\lambda$ has a distinguished $B$-eigenvector
$v_\lambda=\lambda\in V_\lambda^U\subset V_\lambda$. Note that from
the $k_0$-model $B_0$ of $B$ and the semilinear $\sG$-action on
$k[\Gamma]$, we obtain a $G_0$-equivariant $k_0$-model of
$V_{k[\Gamma]}$.

  We now repeat some exposition from Avdeev and Cupit-Foutou \cite[Section 2.5]{acf18}.
  For every $[X, \tau] \in M_\Gamma(k)$, the isomorphism
  $\tau\colon k[X]^U \to k[\Gamma]$ induces an isomorphism of
  $G$-representations $k[X]\isoto V_{k[\Gamma]}$ and thus a
  $k$-$G$-algebra (that is, multiplicative) structure on
  $V_{k[\Gamma]}$. More generally, for any $k$-scheme $S$, we denote
  by $\Am_S$ the set of $\Om_{S}$-$G$-algebra structures on the sheaf
  $\Om_{S} \otimes_k V_{k[\Gamma]}$ that extend the multiplication on
  $k[\Gamma]$. An element $m \in \Am_S$ is a morphism of sheaves of
  $\Om_S$-modules
    \begin{equation*}
    m\colon (\Om_S \otimes_k V_{k[\Gamma]}) \otimes_{\Om_S} (\Om_S \otimes_k V_{k[\Gamma]}) \to \Om_S \otimes_k V_{k[\Gamma]}
  \end{equation*}
  satisfying associativity, commutativity, and compatibility with the
  $G$-action that extends the multiplication map on $k[\Gamma]$. The
$k$-scheme $M_\Gamma$ then represents the contravariant functor
$\Mfunc\colon \textit{Sch}/k \to \textit{Sets}$, $S \mapsto \Am_S$.

\emph{Step 2.}
  Our first goal is to define a semilinear $\sG$-action on
  $M_{\Gamma}$. Let $\gamma \in \gal$. A~$\gamma$-se\-mi\-linear morphism
  $M_\Gamma \to M_\Gamma$ is the same as a $k$-morphism
  $M_\Gamma \to \gamma^{-1}_*M_\Gamma$. Since we have
  $\Hom(S, \gamma^{-1}_*M_\Gamma) = \Hom (\gamma_*S, M_\Gamma)$, the
  $k$-scheme $\gamma^{-1}_*M_\Gamma$ represents the contravariant
  functor $\gamma_*^{-1}\Mfunc\colon \textit{Sch}/k \to \textit{Sets}$,
  $S \mapsto \Am_{\gamma_*S}$. We can therefore define the action map
  $\mu_\gamma\colon M_\Gamma \to M_\Gamma$ by defining a natural
  transformation of functors $\Mfunc \to \gamma^{-1}_*\Mfunc$.

   The ($\gamma^{-1}$-semilinear) base change morphism
    $\gamma_*S \to S$ allows us to consider any sheaf on $S$ as a
    sheaf on $\gamma_*S$ and vice versa. In particular, the base
    change morphism yields a ($\gamma$-semilinear) morphism of sheaves
    $\Om_S \to \Om_{\gamma_*S}$. Together with the
    $\sigma_\gamma$-equivariant $\gamma$-semilinear action map
    $\mug\colon V_{k[\Gamma]} \to V_{k[\Gamma]}$, we obtain a
    $\gamma$-semilinear morphism of sheaves
    \[w_\gamma\colon \Om_S \otimes_k V_{k[\Gamma]} \to \Om_{\gamma_*S}
      \otimes_k V_{k[\Gamma]}\text{.}\]
  Every $m \in \Am_S$ induces an element
    $\gamma^*m \in \Am_{\gamma_*S}$ via the commutative diagram
\begin{align*}
  \xymatrix{
  (\Om_S \otimes_k V_{k[\Gamma]}) \otimes_{\Om_S} (\Om_S \otimes_k V_{k[\Gamma]})
  \ar[r]^-{m} \ar[d]_-{w_\gamma \otimes w_\gamma} &
\Om_S \otimes_k V_{k[\Gamma]}  \ar[d]^-{w_\gamma} \\
  (\Om_{\gamma_*S} \otimes_k V_{k[\Gamma]}) \otimes_{\Om_{\gamma_*S}}
  (\Om_{\gamma_*S} \otimes_k V_{k[\Gamma]})
  \ar[r]^-{\gamma^*m} & \Om_{\gamma_*S} \otimes_k V_{k[\Gamma]} \text{\rlap{.}}
  }
  \end{align*}
  Hence for any $k$-scheme $S$ we may define the map
  $\Am_S \to \Am_{\gamma_*S}$, $m \mapsto \gamma^*m$. The resulting
  family of maps forms a natural transformation
  $\Mfunc \to \gamma^{-1}_*\Mfunc$, which defines a $k$-morphism
    $M_\Gamma \to \gamma^{-1}_*M_\Gamma$, which in turn defines a
    $\gamma$-semilinear morphism
    $\mug\colon M_\Gamma \to M_\Gamma$.
  A calculation shows that for every $\alpha, \beta \in \gal$, we have
  $\mu_{\alpha\beta } = \mu_\alpha \circ \mu_{\beta }$, which means
  that the maps $\mu_\gamma$ define a semilinear $\sG$-action on
  $M_\Gamma$.

\emph{Step 3.}
We now show that for $S = \Spec k$, we recover the action on
$k$-points from Construction~\ref{p:action-action}. There is a
  unique isomorphism of $k$-schemes
  $\Spec k \isoto \gamma_*(\Spec k)$. Applying the functor $\Mfunc$,
  we obtain a canonical identification
  $\Am_{\gamma_*(\Spec k)} \cong \Am_{\Spec k}$. Using this
  identification, the commutative diagram above specializes to
  \begin{align}\label{e:diag-V-Gamma}
   \begin{aligned}
\xymatrix{
    V_{k[\Gamma]} \otimes_{k} V_{k[\Gamma]} \ar[r]^-{m}
       \ar[d]_-{\mu_\gamma\otimes \mu_\gamma} & V_{k[\Gamma]} \ar[d]^-{\mu_\gamma}
    \\
   V_{k[\Gamma]} \otimes_{k} V_{k[\Gamma]}
   \ar[r]^-{\gamma^*m}
   & V_{k[\Gamma]} \text{\rlap{.}}
   }
    \end{aligned}
   \end{align}
   We write $(V_{k[\Gamma]}, m)$ for the $k$-algebra consisting of the
   $k$-vector space $V_{k[\Gamma]}$ together with the multiplication
   map $m$. Let $X' \coloneqq \Spec (V_{k[\Gamma]}, m)$ and
   \smash{$X'' \coloneqq \Spec (V_{k[\Gamma]}, \gamma^*m)$}. We have
   canonical isomorphisms
   $\tau' \colon (V_{k[\Gamma]}, m)^U \isoto k[\Gamma]$ and
   $\tau'' \colon (V_{k[\Gamma]}, \gamma^*m)^U \isoto
   k[\Gamma]$. Now $X''$ is the base change of $X'$ as in
   Construction~\ref{p:action-action} and the maps $\tau'$ and
   $\tau''$ fit into the commutative diagram at the bottom of
   Construction~\ref{p:action-action}.

   \emph{Step 4.}
   Finally, in order to show that this semilinear action defines a
   $k_0$-model of $M_\Gamma$, it suffices to exhibit an intermediate
   field $k_0 \subset k_1 \subset k$ such that $k_1/k_0$ is finite and
   the semilinear action on $M_\Gamma$ restricted to the open subgroup
   $\Gal(k/k_1) \subset \gal$ comes from a $k_1$-model of $M_\Gamma$.

   Let $k_1/k_0$ be a finite extension in $k$ splitting $T_0$. Then
   $\sG_1\coloneqq\Gal(k/k_1)$ acts trivially on $\X^*(B)=\X^*(T)$ and
   hence on $\Gamma$. The map $m$ decomposes into components
     $m^\nu_{\lambda,\mu} \in \Hom^G(V_\lambda \otimes_k V_\mu,
     V_\nu)$ for $\lambda,\mu,\nu \in \Gamma$; see Brion \cite[Section 4.3]{bri13}.
     Since $\gamma\in  \sG_1$ and $\sG_1$ acts trivially on $\Gamma$, we
     obtain the following commutative diagram from the last one:
  \begin{align*}
\xymatrix{
  V_{\lambda} \otimes_{k} V_{\mu} \ar[rr]^-{m^\nu_{\lambda,\mu}} \ar[d]_-{} && V_{\nu} \ar[d]^-{} \\
    V_{\lambda} \otimes_{k} V_{\mu}
   \ar[rr]^-{(\gamma^*m)^\nu_{\lambda,\mu}}
   && V_{\nu} \text{\rlap{.}}
   }
  \end{align*}

Now recall that for any $\lambda\in\Gamma$ the irreducible
$G$-module $V_\lambda$ admits a $k_1$-model in which the distinguished
vector $v_\lambda$ is defined over $k_1$, and such a model is unique
up to a unique isomorphism (because the $k$-span of
$G(k)\cdot v_\lambda$ is all of $V_\lambda$).
From these models, we
obtain $k_1$-models of $A^\nu_{\lambda,\mu} \coloneqq \Hom^G(V_\lambda \otimes_k V_\mu, V_\nu)$ and
hence also a semilinear $\sG_1$-action on the infinite-dimensional affine space
  \begin{align*}
    A \coloneqq \prod_{\lambda,\mu,\nu \in \Gamma} A^\nu_{\lambda,\mu}\text{,}
  \end{align*}
  which is the spectrum of its coordinate ring
  \begin{align*}
    k[A] = \bigotimes_{\lambda,\mu,\nu \in \Gamma} k[A^\nu_{\lambda,\mu}]\text{,}
  \end{align*}
  where every element is a finite sum of pure tensors and at most
  finitely many factors of every pure tensor are different from $1$.

  As explained in Brion \cite[Section 4.3]{bri13}, the moduli scheme $M_\Gamma$
  is a closed subscheme of $A$. The $k$-points of $M_\Gamma$ are those
  multiplication maps $m \in A(k)$, which satisfy commutativity,
  associativity, and compatibility with multiplication on $k[\Gamma]$;
  these conditions can be expressed as polynomial relations.

  Since the semilinear $\sG_1$-action on $A$ sends $m$ to
  $\gamma^*m$, the closed embedding
  $M_\Gamma\hookrightarrow A$ is $\sG_1$-equivariant. The coordinate
  ring $k[A]$ is generated by functions in
  $k[A^\nu_{\lambda,\mu}] \subset k[A]$ for
  $\lambda,\mu,\nu \in \Gamma$ (infinitely many functions in total).
  The stabilizer in $\sG_1$ of every function in
  $k[A^\nu_{\lambda,\mu}]$ is open because the $\sG_1$-action defines
  a $k_1$-model of $A^\nu_{\lambda,\mu}$. It follows that the
  stabilizer in $\sG_1$ of any function in $k[A]$ is open, and hence,
  that the same is true for any function in $k[M_\Gamma]$.
  By  Lemma  \ref{e:fixed-points} below,
 the $\sG_1$-action defines a $k_1$-model of the affine $k$-scheme of finite type
 $M_\Gamma$, which completes the proof of the proposition.
\end{proof}

\begin{lemma}\label{e:fixed-points}
  Let $k/k_0$ be a Galois extension (not necessarily finite) with
  Galois group~$\sG$. Let $R$ be a finitely generated commutative
  $k$-algebra with unit endowed with a $k$-semi\-linear action
  $\sG \times R \to R$ such that the stabilizer in $\sG$ of any
  element $f\in R$ is open. Let $R_0=R^\sG$ denote the $k_0$-algebra
  with a unit consisting of the fixed points of $\sG$ in $R$. Then the
  natural map $R_0\otimes_{k_0} k\to R$ is an isomorphism of
  $k$-algebras, and $R_0$ is a finitely generated $k_0$-algebra.
\end{lemma}

\begin{proof}
  The second assertion follows from the first; see EGA \cite[Lemma
  2.7.1.1]{EGA-IV-2}. We prove the first assertion.

  Let $x_1,\dots,x_n$ be a set of generators of $R$ over $k$. Let
  $\sU_i$ denote the stabilizer of $x_i$ in~$\sG$, and set
  $\sU = \sU_1 \cap \ldots \cap \sU_n$. Shrinking $\sU$ if necessary,
  we may assume that $\sU$ is a normal open subgroup of $\sG$ and that
  $\sU$ fixes all generators $x_1,\dots,x_n$.

  Consider the set of multi-indices
  $\mathcal{I}=(\mathbb{Z}_{\ge 0})^n$. For
  $I=(i_1,\dots,i_n)\in\mathcal{I}$, write
  $ x^I=x_1^{i_1}\cdots x_n^{i_n}$. Then the set
  $\{x^I\}_{I\in\mathcal{I}}$ generates $R$ as a vector space over
  $k$. It follows that there exists a subset
  $\mathcal{J}\subset\mathcal{I}$ such that the family
\begin{equation}\label{e:J}
\{x^I\}_{I\in\mathcal{J}}
\end{equation}
is a basis of $R$ as a vector space over $k$;
see, for instance, Lang \cite[Theorem III.5.1]{Lang}.
Clearly, each $x^I$ is $\sU$-stable.
We see that an element
\[\sum_{I\in\mathcal{J}} a_I x^I\in R \quad (a_I\in k)\]
is $\sU$-stable if and only if $a_I\in k^\sU$ for all
$I\in \mathcal{J}$, because the family \eqref{e:J} is a $k$-basis of~$R$.
Thus the family \eqref{e:J} is a basis of $R^\sU$ over $k^\sU$.
It follows that the canonical map
\[R^\sU\otimes_{k^\sU} k\to R\] is an isomorphism of vector spaces
over $k$, and hence, an isomorphism of $k$-algebras with units.

Since $(R^\sU)^{\sG/\sU} = R^\sG$, this reduces the lemma to the case
of the \emph{finite} Galois extension $k^{\sU}/k_0$ with finite Galois
group $\sG/\sU$, which is classical; see, for instance, Jahnel
\cite[Proposition~2.3]{Jahnel}.
\end{proof}

\begin{subsec}
  Given any $k_0$-point of $\MGz$ of the form $[X, \tau]$, it will be
  straightforward to obtain a $G_0$-equivariant $k_0$-model of $X$
  (this will be done in the proof of Theorem~\ref{th:qsaff}). In order
  to prove that such a $k_0$-point exists, our plan is to consider a
  certain subscheme $C_X$ of $M_\Gamma$, which will be a smooth toric
  variety defined over $k_0$, and then to apply Voskresenski\u{\i} and Klyachko
  \cite[Proposition~4]{vk85}, which is a result on the existence of
  $k_0$-points on smooth toric varieties.

The $k_0$-torus acting on this toric variety will be
$\overline{T}_0 \coloneqq T_0/Z(G_0)$. Over $k$, we consider the
natural action of $\overline{T} \coloneqq T/Z(G)$ on $M_\Gamma$ as
defined in Alexeev and Brion \cite[Section~2]{ab05}
and Avdeev and Cupit-Foutou \cite[Subsection~2.7]{acf18}.

Before we can describe this $\overline{T}$-action on $M_\Gamma$, we
recall that
  \begin{align*}
    \Xi\hs =\hs \big\{\lambda + \mu - \nu \ \mid\ \lambda, \mu, \nu \in
      \Gamma, \ \langle G(k) \cdot f_\nu\rangle_k \hs\cap\hs \big\langle\hs (G(k)
      \cdot f_\lambda) \cdot (G(k) \cdot f_\mu)\hs \big\rangle_k \ne 0\big\}\hs
    \subset\hs \Gamma
  \end{align*}
  is called the \emph{root monoid} of $X$. In other words,
  if $[X, \tau] \in M_\Gamma(k)$ for any $\tau$, then $\Xi$
  is determined by the corresponding
  multiplication map
  \begin{align*}
    m\colon V_{k[\Gamma]} \otimes_{k} V_{k[\Gamma]}\to V_{k[\Gamma]}\text{.}
  \end{align*}
  Namely, for every $\xi\in\Gamma$ we have
  $\xi \in \Xi$ if and only if there exist $\lambda, \mu, \nu \in \Gamma$
  such that $m^\nu_{\lambda,\mu} \ne 0$ and $\lambda + \mu - \nu = \xi$.

\begin{remark}
  Let $z \in Z(G)$ and $\xi \in \Xi$. Then there exists a component
  \begin{align*}
    m^\nu_{\lambda,\mu}\colon V_\lambda \otimes_k V_\mu \to V_\nu
  \end{align*}
  of $m$ as well as $v_{\lambda,\mu} \in V_\lambda \otimes_k V_\mu$ and $v_\nu \in V_\nu$
  such that $\xi = \lambda + \mu - \nu$ and $m^\nu_{\lambda,\mu}(v_{\lambda,\mu}) = v_\nu \ne 0$.
  Hence we have, by Schur's lemma,
  \begin{align*}
    \nu(z)\cdot v_{\nu} = z\cdot v_{\nu} = m^\nu_{\lambda,\mu}(z\cdot v_{\lambda,\mu}) =
    m^\nu_{\lambda,\mu}((\lambda+\mu)(z)\cdot v_{\lambda,\mu}) = (\lambda+\mu)(z)\cdot v_\nu\text{,}
  \end{align*}
  and we obtain $\xi(z) = 1$.
\end{remark}

  We can now describe the $\overline{T}$-action on $M_\Gamma(k)$.
  Let $\overline{t} \in \overline{T}(k)$, $[X', \tau'] \in
  M_\Gamma(k)$, and let $m' \colon V_{k[\Gamma]} \otimes_{k}
  V_{k[\Gamma]} \to V_{k[\Gamma]}$ be the multiplication map
  corresponding to $[X', \tau']$. Then the multiplication map
  $\overline{t} \cdot m'$ corresponding to $\overline{t}\cdot [X',
    \tau']$ is given by
  \begin{align*}
      (\overline{t}\cdot m')^\nu_{\lambda,\mu} =
      (\lambda+\mu-\nu)(\overline{t}) \cdot (m')^\nu_{\lambda,\mu}
  \end{align*}
  on each component; see Alexeev and Brion \cite[Proposition~2.11]{ab05}.

  Equivalently, we can set $X'' \coloneqq X'$, lift $\overline{t}$ to an element
  $t \in T(k)$, and define $\tau''$ as the composition
  \begin{align*}
    \xymatrix{
    k[X]^U \ar[rr]^-{\tau'} && k[\Gamma] \ar[rr]^-{f\mapsto t^{-1}\cdot f}
    && k[\Gamma]\text{\rlap{.}}
      }
  \end{align*}
  The $\overline{T}$-action is then given by $\overline{t}\cdot [X',
    \tau'] \coloneqq [X'', \tau'']$, which is independent of the lift
  $t \in T(k)$ of $\overline{t}$.

  We now show that this action
  is defined over $k_0$.
\end{subsec}

\begin{prop}
  \label{prop:comp}
  Assume that the $\gal$-action preserves $\Gamma$. Then for every
  $\gamma \in \gal$, $\overline{t} \in \overline{T}(k)$, and
  $[X', \tau'] \in M_\Gamma(k)$, we have
  \begin{align*}
    {}^\gamma(\overline{t} \cdot [X', \tau']) =
    {}^\gamma \overline{t} \cdot {}^\gamma([X', \tau'])\text{.}
  \end{align*}
  In particular, we obtain a $\overline{T}_0$-action on $\MGz$.
\end{prop}
\begin{proof}
  The diagram
         \begin{align*}
           \xymatrix{
           k[\Gamma] \ar[rr]^-{f \mapsto t^{-1}\cdot f} \ar[d]_{\gamma_*} &&
           k[\Gamma] \ar[d]^{\gamma_*} \\
           k[\Gamma] \ar[rr]^{f \mapsto {}^\gamma t^{-1} \cdot f} &&
           k[\Gamma]}
         \end{align*}
         commutes for every $t \in T(k)$, which gives the
         required property.
\end{proof}

\begin{remark}
\label{rem:c}
According to  Alexeev and Brion \cite[Theorem~1.12 and Lemma~2.2]{ab05}, see also
Avdeev and Cupit-Foutou \cite[Corollary~2.21]{acf18}, two points
$[X', \tau'], [X'', \tau''] \in M_\Gamma(k)$ are in the same
$\overline{T}$-orbit if and only if $X'$ and $X''$ are $G$-equivariantly
isomorphic (ignoring the maps $\tau'$ and $\tau''$). In particular,
the $\overline{T}$-orbit $C_{X'}^\circ$ of $[X', \tau']$ and its closure
$C_{X'}$, considered with the reduced subscheme structure, do not
depend on the map $\tau'$.
\end{remark}

\begin{prop}[{\cite[Corollary~4.14]{acf18}}]
  \label{prop:smoothtoric}
  The $\overline{T}$-scheme $C_X$ is equivariantly isomorphic to a
  smooth toric $\overline{T}$-variety (or, more specifically, an affine
  space).
\end{prop}

\begin{prop}
  \label{prop:gammavm}
  Assume that the $\gal$-action preserves $\Gamma$, choose any
  $\tau$ such that
  $[X, \tau] \in M_\Gamma$, let $\gamma \in \sG$, and let
  $[X', \tau'] = {}^\gamma[X, \tau]$. Then the valuation cone of $X'$
  is ${}^\gamma\Vm$.
\end{prop}
\begin{proof}
  We have
  \begin{align*}
    \Vm = \big\{v \in \Hom_{\Zd}(\sX, \Qd)\, \mid\, \langle v, \xi \rangle
    \le 0 \text{ for all $\xi \in \Xi$}\big\}
  \end{align*}
  by Knop \cite[Theorem~1.3]{kno96}.
  The proposition now follows from the diagram \eqref{e:diag-V-Gamma}.
  \end{proof}

\begin{prop}
  \label{prop:restr}
  Assume that the $\gal$-action preserves $\Gamma$ and $\Vm$. Then
  the semilinear $\gal$-action on $M_\Gamma$ preserves  $C_X$. In
  particular, we obtain a $k_0$-model $\CXz$ of $C_X$.
\end{prop}
\begin{proof}
  According to Losev \cite[Theorem~1.2]{los09b} (see also Avdeev and
  Cupit-Foutou \cite[Corollary~4.16]{acf18} for a proof using the
  moduli scheme $M_\Gamma$), any affine spherical $G$-variety $X$ is
  determined up to $G$-equivariant isomorphism by $\Gamma$ and $\Vm$.
  Therefore, it follows from Proposition~\ref{prop:gammavm} that
  $\gal$ preserves $C_X^\circ$, hence also its closure $C_X$.
\end{proof}

We can now show that the open $\overline{T}_0$-orbit in $\CXz$ admits a
$k_0$-point.

\begin{prop}
  \label{prop:fixed}
  Assume that the $\gal$-action preserves $\Gamma$ and $\Vm$. Then
  the open $\overline{T}_0$-orbit in $\CXz$ admits a $k_0$-point.
\end{prop}
\begin{proof}
  According to Propositions~\ref{prop:restr} and \ref{prop:comp}, we
  have a $\overline{T}_0$-equivariant (toric) $k_0$-model $\CXz$ of
  $C_X$. According to \cite[Theorem~2.7]{ab05}, see also
  \cite[Theorem~2.22]{acf18}, the toric variety $C_X$ contains exactly
  one closed $\overline{T}$-orbit, which is one point. Clearly, this
  point is a $k_0$-point in $\CXz$. Since the {\em smooth} toric variety
  $\CXz$ contains a $k_0$-point, by Voskresenski\u{\i} and Klyachko \cite[Proposition~4]{vk85}
  the open $\overline{T}_0$-orbit in $\CXz$ also contains a $k_0$-point.
\end{proof}

We state the main result of this section in a self-contained way.

\begin{theorem}
  \label{th:qsaff}
  Let $k_0$, $k$, and $\sG$ be as in Subsection \ref{ss:mod-Y}.
  Let $X$ be an affine spherical $G$-variety over $k$ with weight monoid $\Gamma$
  and valuation cone $\Vm$.  Let $G_0$ be a
  \emphb{quasi-split} $k_0$-model of $G$.
  Then  $X$ admits a $G_0$-equivariant $k_0$-model if and only if
  the $\sG$-action defined by $G_0$ preserves $\Gamma$ and $\Vm$.
\end{theorem}

\begin{proof}
  If  $X$ admits a $G_0$-equivariant $k_0$-model,
  then by Corollary \ref{c:gal-preserves} the $\gal$-action preserves $\Gamma$.
  Moreover, then the unique  open $G$-orbit $X^\circ$ in $X$ admits a
   $G_0$-equivariant $k_0$-model, and by Huruguen \cite[Section 8.2]{Huruguen},
   see also \cite[Proposition 2.2]{BG}, the $\gal$-action preserves $\sV(X^\circ)=\sV(X)$.

  Conversely, assume that the $\sG$-action defined by $G_0$ preserves $\Gamma$ and
  $\Vm$. We first construct a semilinear $\gal$-action on $X$ such
  that
  \begin{align*}
    {}^\gamma(g \cdot x) = {}^\gamma g \cdot {}^\gamma x
  \end{align*}
  for every $\gamma \in \gal$, $g \in G(k)$, and $x \in X(k)$.
  According to Proposition~\ref{prop:fixed}, there exists a
  $\sG$-fixed point $[X'', \tau''] \in C_X^\circ(k)$. By
    Remark~\ref{rem:c}, there exists $\tau$ such that $[X'', \tau''] =
    [X, \tau]$. Let $\gamma \in \gal$ and $[X', \tau'] = {}^\gamma
  [X, \tau]$ with $X'$ and $\tau'$ as in Construction~\ref{p:action-action}.
  Since we have $[X, \tau] = [X', \tau']$, there exists a uniquely
  determined $G$-equivariant isomorphism $X \to X'$ such that the
  following diagram commutes:
  \begin{align*}
    \xymatrix{
    k[X]^U \ar[rd]_{\tau} \ar[rr]^-\cong & & k[X']^U \ar[ld]^-{\tau'} \\
                                        & k[\Gamma]\text{\rlap{.}}
                                          }
  \end{align*}
  The isomorphism fits into the following diagram (in which the left square
  is as in Construction~\ref{p:action-action}):
  \begin{align*}
       \xymatrix{
    X \ar[r]^{b^{-1}}\ar[d] & X' \ar[r]^{\smash{\cong}} \ar[d] & X\ar[d] \\
    \Spec k\ar[r]^{(\gamma^*)^{-1}} & \Spec k \ar[r]^{\operatorname{id}} & \Spec k\text{\rlap{.}}}
  \end{align*}
  Now we use the top row to define the required action map
  $\mu_\gamma\colon X \to X$.

  In order to show that this semilinear action defines a $k_0$-model
  of $X$, it remains to exhibit a finite extension $k_1/k_0$ in $k$
  such that the semilinear action on $X$ restricted to the open
  subgroup $\Um \coloneqq \Gal(k/k_1) \subset \gal$ comes from a
  $k_1$-model of $X$. Let $k_0 \subset k_1 \subset k$ be such that
  $G_1 \coloneqq G_0 \times _{k_0} k_1$ is split and $k_1/k_0$ is
  finite. For every $\lambda \in \Gamma$ the action of $\sG_1$
    on $k[X]$ restricts to $V_\lambda$ and fixes the distinguished
    $B$-eigenvector $v_\lambda$. Moreover, there is only one such
    $\sG_1$-action, namely the $\sG_1$-action coming from the unique
    $k_1$-model of $V_\lambda$ such that $v_\lambda$ is defined over
    $k_1$. In particular, the stabilizer in $\sG_1$ of every function
    in $k[X]$ is open.  Hence by Lemma~\ref{e:fixed-points}, the
    $\sG_1$-action defines a $k_1$-model of $X$.
\end{proof}

\begin{theorem}
  \label{th:affp}
  Let $X_0$ be a $G_0$-equivariant $k_0$-model of
  $X$ constructed as in Theorem~\ref{th:qsaff}.
  Then the set of $k_0$-points in $X_0$ is dense.
\end{theorem}

\begin{proof}
   We have $X_0/U_0 \cong \Spec(k[\Gamma])$ as a $k_0$-variety.
   The map $k[\Gamma] \to k$ sending $\lambda$ to $1$ for every
    $\lambda \in \Gamma$ is defined over $k_0$ and hence defines a
    $k_0$-point. This point lies in the open $T$-orbit of
    $\Spec(k[\Gamma])$; see, for instance, Cox, Little, and Schenck
    \cite[Proposition~1.1.14]{cls11}. By Borel \cite[Corollary~18.3]{bor91},
    the set of $k_0$-points in $\Spec(k[\Gamma])$ is dense. Since
  $X_0 \to X_0/U_0 \cong \Spec(k[\Gamma])$ is a quotient by a
  unipotent $k_0$-group, by Proposition \ref{p:unip-dense} in
  Appendix \ref{s:G/U} the set of $k_0$-points in $X_0$ is also dense.
\end{proof}

\section{Models of spherical homogeneous spaces of quasi-split groups}
\label{s:spharb}

\begin{subsec}
Let $k_0$, $k$, and $\sG$ be as in Subsection \ref{ss:mod-Y}.
Let $G$ be a (connected) reductive group over $k$.
We fix a Borel subgroup $B \subset G$ and
a maximal torus $T \subset B$.
Let $G/H$ be a spherical homogeneous space with combinatorial
invariants $(\sX, \Vm, \Dm)$. Recall  that $\Dm$
is a finite set equipped with two maps
$\rho\colon \Dm \to \Hom(\sX,\Q)$ and
$\varsigma\colon \Dm \to \Pm(\Sm)$. We denote by
$\Sigma \subset \sX$  the set of {\em spherical roots} of $G/H$, that is,
the uniquely determined linearly independent
set of primitive elements  of $\sX$ such that
\begin{align*}
  \Vm = \bigcap_{\gamma \in \Sigma}\big\{v \in \Hom_\Zd(\sX,\Q)
  \mid \langle v, \gamma \rangle \le 0\big\}\text{.}
\end{align*}
Whenever $\sX$ is given, we may specify $\Sigma$ instead of $\Vm$ and
vice versa. We recall from Subsection~\ref{ss:sph-qs} the
alternative presentation of $(\sX, \Vm, \Dm)$ as
$(\sX, \Vm, \Omone, \Omt)$. As before, we define
$\Omega = \Omone \cup \Omt$. For all $\alpha \in S$ we write
$\Dm(\alpha) = \{D \in \Dm \mid \alpha \in \varsigma(D)\}$.
\end{subsec}

\begin{prop}[Corollary of Luna {\cite[Proposition~6.4]{lun01}}]
  \label{prop:aug}
  Let $\sX' \subset \X^*(B)$ be a sublattice containing $\sX$ such
  that all elements of $\Sigma$ are primitive in $\sX'$, and let
  $\rho' \colon \Dm \to \Hom(\sX', \Zd)$ be a map such that
  $\rho'(D)|_{\sX} = \rho(D)$ for every $D \in \Dm$. We write $\Dm'$
  for the set $\Dm$ equipped with the maps $\rho'$ and
  $\varsigma' \coloneqq \varsigma$. Assume that for every
  $\alpha \in \Sm$ the following conditions are satisfied:
   \begin{enumerate}
   \item if $\Dm(\alpha)$ contains two different elements $D^{+}$ and $D^-$, then
     $\rho'(D^{+}) + \rho'(D^-) = \alpha^\vee|_{\sX'}$,
   \item if $2\alpha \in \Sigma$ with $\Dm(\alpha) = \{D\}$, then
     $\rho'(D) = \frac{1}{2}\alpha^\vee|_{\sX'}$,
   \item if $2\alpha \notin \Sigma$ and $\Dm(\alpha) = \{D\}$, then
     $\rho'(D) = \alpha^\vee|_{\sX'}$,
   \item if $\Dm(\alpha) = \emptyset$, then $\alpha^\vee|_{\sX'} = 0$.
   \end{enumerate}
   Then $(\sX', \Sigma, \Dm')$ are the invariants of a
  spherical subgroup  $H'\subset G$.
 \end{prop}

 \begin{prop}[Luna {\cite[proof of Proposition~6.3]{lun01} and Losev \cite[Theorem 1]{los09a},
 see also Knop \cite[Section~4]{kno91} and Hofscheier \cite[Theorem~1.4]{hofcont}}]
   \label{prop:sub}
   In  Proposition~\ref{prop:aug}, the spherical
   subgroup $H' \subset G$ with combinatorial invariants
   $(\sX', \Sigma, \Dm')$ can be chosen such that $H' \subset H$.
 \end{prop}

 \begin{prop}[{Timashev \cite[Corollary~15.6]{tim11}}\hs]
   \label{prop:qaf}
   The spherical homogeneous space $G/H$ is quasi-affine if and only
   if the set $\{\rho(D) \mid D \in \Dm \}$ does not contain
   $0$ and spans a strictly convex cone in $\Hom(\sX,\Q)$.
\end{prop}

\begin{subsec} We recall a construction based on Section 4.1 of Brion \cite{bri97}, see
  also Gagliardi \cite[Proposition~3.1]{gag14}, which permits us to write $G/H$ as a
  quotient of a {\em quasi-affine} spherical homogeneous space by a torus.

  Let $\langle\Dm\rangle$ denote the free abelian group with basis
  $\{e_D \mid D\in\Dm\}$, and let $C$ denote the $k$-torus with character
  group $\langle\Dm\rangle$; then $C(k)$ is the group of maps
  $\Dm\to k^\times$. We define $G' = G \times C$ and
  $B' = B \times C$, so that $\X^*(B') = \X^*(B) \oplus \X^*(C)$.
  Then $\sX$ can be regarded as a sublattice of $\X^*(B')$, and in
  this case $(\sX, \Sigma, \Dm)$ are the combinatorial invariants of
  the spherical subgroup $H \times C \subset G'$.

 We denote by $\omega_\alpha\in \X^*(B)\otimes_\Z \Q$ the fundamental dominant weight
   associated with a simple root $\alpha$, and we choose an integer
   $q > 0$ such that $q\omega_\alpha \in \X^*(B)$ for all
   $\alpha \in \Sm$. For each $D \in \Dm$ we set
 \begin{align*}
   \omega_D = q\cdot r_D\cdot\sum_{\alpha \in \varsigma(D)} \omega_\alpha\ \in\X^*(B),
 \end{align*}
 where
 \begin{align*}
   r_D = \begin{cases}
     2 &\text{if $\varsigma(D) \subset \frac{1}{2}\Sigma$,} \\
     1 &\text{otherwise.}
     \end{cases}
   \end{align*}
   Furthermore, we set
 \[\lambda_D = (\omega_D, e_D) \in\X^*(B)\oplus\X^*(C)= \X^*(B').\]
 \end{subsec}

 \begin{subsec}
  We define new invariants. First, we set
 \begin{align*}
   \sX' = \sX \, +\,  \langle \lambda_D : D \in \Dm \rangle_\Zd \subset \X^*(B')\text{,}
 \end{align*}
 where the sum is clearly direct. Moreover, for each $D \in \Dm$ we
 define
 \[\rho'(D) \in \Hom(\sX',\Z)\subset\Hom(\sX',\Q)\]
 by the equalities
 \begin{align*}
   \langle \rho'(D), \lambda \rangle =
   \begin{cases}\langle \rho(D), \lambda \rangle & \text{for $\lambda \in \sX$,} \\
     q & \text{for $\lambda = \lambda_D$,} \\
     0 & \text{for $\lambda = \lambda_{D'}$ with $D' \ne D$.}\end{cases}
 \end{align*}
 Finally, we set
 \[\varsigma'(D) = \varsigma(D) \subset \Sm.\]
 Thus we obtain a new set of colors $\Dm'$, which is the same set
 $\Dm$, but with the associated maps $\rho'$ and $\varsigma'$. Note
 that the map $\rho'\times\vs'$ is injective, because the map $\rho'$
 is clearly injective. If we set $\Omega'=\im(\rho'\times\vs')$, then
 with the natural notation we have $\Omega^{\prime\ste{2}}=\emptyset$
 and $\Omega^{\prime\ste{1}}=\Omega'\cong \Dm'=\Dm$.
\end{subsec}

\begin{prop}
   \label{prop:invqaf}
   The invariants $(\sX', \Sigma, \Dm')$ come from a
   quasi-affine spherical homogeneous space
   $G'/H' = (G\times C)/H'$ such that $H' \subset H \times C$.
   \end{prop}

   \begin{proof}
     We begin by verifying that  conditions  (1--4)  of
       Proposition~\ref{prop:aug} are satisfied.

       Let $\alpha \in S$ with $\Dm(\alpha) = \{D^{+}, D^-\}$.
       Then, for $\lambda \in \sX$, we have
       \begin{align*}
         \langle \rho'(D^{+}) + \rho'(D^-), \lambda \rangle
         = \langle \rho(D^{+}) + \rho(D^-), \lambda \rangle
         = \langle \alpha^\vee, \lambda \rangle\text{.}
       \end{align*}
       Moreover, for every $D' \in \Dm$ we have
       \begin{align*}
         \langle \rho'(D^{+}) + \rho'(D^-), \lambda_{D'} \rangle
         = \begin{dlrcases}q & \text{if $\alpha \in \varsigma(D')$}\\ 0 & \text{otherwise}\end{dlrcases}
         = \langle \alpha^\vee, \omega_{D'} \rangle = \langle \alpha^\vee, \lambda_{D'} \rangle\text{.}
       \end{align*}
      {This gives $(1)$.}

     Let $2\alpha \in \Sigma$ with $\Dm(\alpha) = \{D\}$. Then, for $\lambda \in \sX$, we have
     \begin{align*}
       \langle \rho'(D), \lambda \rangle = \langle \rho(D), \lambda \rangle
        = \langle \tfrac{1}{2}\alpha^\vee, \lambda \rangle\text{.}
     \end{align*}
     Moreover, for every $D' \in \Dm$, we have
     \begin{align*}
       \langle \rho'(D), \lambda_{D'} \rangle = \begin{dlrcases}q &
       \text{if $\alpha \in \varsigma(D')$}\\
       0 & \text{otherwise}\end{dlrcases} = \tfrac{1}{2}\langle \alpha^\vee , \omega_{D'}\rangle=
        \langle \tfrac{1}{2}\alpha^\vee, \lambda_{D'} \rangle\text{.}
     \end{align*}
     {This gives $(2)$.}

     Let $\alpha \in S$, $2\alpha \notin \Sigma$, and
     $\Dm(\alpha) = \{D\}$. Then, for $\lambda \in \sX$, we have
     \begin{align*}
       \langle \rho'(D), \lambda \rangle = \langle \rho(D), \lambda \rangle =
        \langle \alpha^\vee, \lambda \rangle\text{.}
     \end{align*}
     Moreover, for every $D' \in \Dm$, we have
     \begin{align*}
       \langle \rho'(D), \lambda_{D'} \rangle = \begin{dlrcases}q & \text{if $\alpha \in \varsigma(D')$}\\
       0 & \text{otherwise}\end{dlrcases} = \langle \alpha^\vee , \omega_{D'}\rangle=
        \langle \alpha^\vee, \lambda_{D'} \rangle\text{.}
     \end{align*}
     {This gives $(3)$.}

     Let $\alpha \in S$ such that $\Dm(\alpha) = \emptyset$. Then, we
     have $\alpha^\vee|_{\sX} = 0$ and
     $\langle \alpha^\vee, \lambda_{D'}\rangle = 0$ for every
     $D' \in \Dm$. Therefore, we have $\alpha^\vee|_{\sX'} = 0$.
     {This gives $(4)$.}

     We have shown that the conditions of Proposition~\ref{prop:aug}
     are satisfied.
     Hence there exists a spherical subgroup
     $H' \subset G'$ with invariants
     $(\sX', \Sigma, \Dm')$.
     By Proposition~\ref{prop:sub}, we can choose $H'$ to be a subgroup of $H \times C$.

     The set $\{\rho'(D) \mid D \in \Dm' \}$ does not contain $0$ and
     is linearly independent. In particular, it spans a strictly convex
     cone in $\Hom(\sX,\Q)$. It follows from
     Proposition~\ref{prop:qaf} that the spherical homogeneous space
     $G'/H'$ is quasi-affine.
   \end{proof}

   \begin{lemma}
     In Proposition~\ref{prop:invqaf} we have
     $H'\cdot C=H\times C\subset G\times C$.
   \end{lemma}

   \begin{proof}
     Since we have $H' \cdot C \subset H \times C$, in order to show
     $H'\cdot C = H \times C$, it is sufficient to show that the
     spherical homogeneous spaces $(G\times C)/(H'\cdot C)$ and
     $(G\times C)/(H \times C)$ of the same group $G\times C$
     have the same combinatorial invariants.

     The weight lattice of $(G\times C)/(H'\cdot C)$ consists of the
     $C$-invariants
     \[(\sX')^C = \sX' \cap \X^*(B)\subset \X^*(B)\oplus \X^*(C)\]
     of the weight lattice
     $\sX'$ of the spherical homogeneous space $(G\times C)/H'$.
     It follows from the definition of $\sX'$
     that we have $\sX' \cap \X^*(B) = \sX$, which is the weight lattice
     of $(G\times C)/(H \times C)$.

     By construction, the surjective map $(G\times C)/H' \to (G\times C)/(H\times C)$
     induces a bijection of colors $\Dm\isoto \Dm'$, where $\varsigma'=\vs$
     and $\rho$ is the composite map
     \[\Dm\labelto{\rho'}\Hom(\Xm',\Q)\to\Hom(\Xm,\Q),\]
     where the right-hand map is induced by the embedding $\Xm\into\Xm'$.
     Moreover, this surjective  map $(G\times C)/H' \to (G\times C)/(H\times C)$
     factors as
     \begin{align*}
       (G\times C)/H' \to (G\times C)/(H'\cdot C) \to (G\times C)/(H\times C)\text{.}
     \end{align*}
     We see that the map
     $(G\times C)/(H'\cdot C) \to (G\times C)/(H\times C)$
     induces a bijection of colors as well (now preserving both associated
     maps, because the weight lattices are identical).

     Since $H'\cdot C\subset H\times C$ and  the weight lattices are identical, it follows from
     Knop \cite[Section~4]{kno96} that the valuation cones of
     $(G\times C)/(H'\cdot C)$ and $(G\times C)/(H \times C)$ are the
     same.

     We see that the spherical homogeneous spaces
 $(G\times C)/(H'\cdot C)$ and  $(G\times C)/(H \times C)$ of the  group $G\times C$
 have the same combinatorial invariants.
 By Losev's uniqueness theorem \cite[Theorem 1]{los09a},
 the subgroups $H'\cdot C$ and $H\times C$ are conjugate in $G\times C$.
 Since $H'\cdot C\subset H\times C$, we conclude that $H'\cdot C= H\times C$.
    \end{proof}

\begin{corollary}\label{c:quot}
 Any fiber of  the quotient map
   \[(G \times C)/H' \to (G\times C)/(H'\cdot C)=(G\times C)/(H\times C)= G/H\]
is an orbit of the torus $C$  under  the natural action of $C\subset G\times C$ on $(G\times C)/H'$.
\end{corollary}

 \begin{subsec}
   \label{rem:gpp}
   Let $G_0$ be a quasi-split $k_0$-model of $G$, and assume that
   the induced $\gal$-action preserves the quadruple
   $(\sX, \Sigma, \Omone, \Omt)$.
   Consider the $\gal$-action
   \[\alpha_\Omega^\ste{2}\colon \sG\times\Omt\to\Omt\text{,}\]
   which is clearly continuous (the stabilizer of any point of $\Omt$ is open in $\gal$).
   Let $\Dm^\ste{2}$ denote the preimage of $\Omt$ in $\Dm$.
   Choose a continuous action
   \[\alpha_\Dm^\ste{2}\colon \gal\times \Dm^\ste{2}\to\Dm^\ste{2}\]
   lifting $\alpha_\Omega^\ste{2}$.
   In this way we obtain a continuous action
   \[\alpha_\Dm\colon\gal\times\Dm\to \Dm\]
   lifting the $\gal$-action on $\Omega$, that is, such that the map
        \[\rho\times\vs\colon \ \Dm\longrightarrow \  \Omega\subset \Hom(\sX,\Q)\times\Pm(\Sm)\]
     is $\sG$-equivariant.

     \begin{remark}
       At least one such lift always exists. Let $\sU$ denote the
       kernel of the action of $\sG$ in $\Omega$; then, clearly, $\sU$
       is an open subgroup of $\sG$. By \cite[Lemma 9.3]{BG} the
       $\sG$-action on the finite set $\Omega$,
        preserving the subsets $\Omone$ and $\Omt$, can be lifted (in
       general non-uniquely) to a $\sG$-action on $\Dm$ such that the
       kernel of this new action is again $\sU$.
   \end{remark}

   Now $\sG$ acts on $\Dm$ and hence on the basis $(e_D : D \in \Dm)$
   of $\X^*(C)$. We consider the corresponding action of $\sG$ on
   $\X^*(C)$ and the corresponding quasi-trivial $k_0$-torus $C_0$
   (which depends on $\alpha_\Dm$).

   \begin{prop}
     \label{prop:psemact}
     Set $G_0'=G_0\times_{k_0} C_0$ and  $G'=(G'_0)_k= G \times C$.
     Then the $\gal$-action on
     $\X^*(B)$ and $\Sm$
     preserves the combinatorial invariants
     $(\sX', \Sigma, \Omega'{}^{\ste{1}}, \Omega'{}^{\ste{2}})$.
   \end{prop}
   \begin{proof}
     For every $\gamma \in \gal$ and $D \in \Dm$ we have
     $r_{\hs^\gamma\hm D} = r_D$ and
     \begin{align*}
       \sum_{\alpha \in\hs \varsigma(D)} {}^\gamma\hm\omega_\alpha
       = \sum_{\alpha\in \hs \varsigma(D)} \omega_{\hsss^\gamma\hm\alpha}
       = \sum_{\alpha \in \hs^\gamma\hm\varsigma(D)} \omega_\alpha
       = \sum_{\alpha \in\hs \varsigma({}^\gamma\hm D)} \omega_\alpha\text{,}
     \end{align*}
     hence ${}^\gamma\hm \lambda_D = ({}^\gamma\omega_D, {}^\gamma\hm e_D) =
     (\omega_{\hsss^\gamma\hmm D}, e_{\hs^\gamma\hmm D}) = \lambda_{\hs^\gamma\hmm D}$.
    Taking into account that $\Gm$ preserves $\Xm$, we conclude
    that $\gal$ preserves $\sX'$.

Since $\Omega^{\prime\ste{2}}=\emptyset$, it remains to prove that $\Gm$ preserves   the subset
$\Omega^{\prime\ste{1}}=\Omega'$.
We have $\Omega'=\im(\rho'\times\vs')$, and so it suffices to show that the maps $\rho'$ and $\vs'$
are $\Gm$-equivariant.
Since $\vs'=\vs$, it remains to show that the map
$\rho'\colon \Dm'\to\Hom_\Z(\Xm',\Q)$ is $\Gm$-equivariant.

     Let  $D\in\Dm'=\Dm$. Then
     \begin{align*}
   \langle \rho'(D), \lambda \rangle =
   \begin{cases}\langle \rho(D), \lambda \rangle & \text{for $\lambda \in \sX$,} \\
     q & \text{for $\lambda = \lambda_D$,} \\
     0 & \text{for $\lambda = \lambda_{D'}$ with $D' \ne D$.}\end{cases}
     \end{align*}
   Let  $\gamma \in \gal$.
   For every $\lambda \in \sX$ we have
     \begin{align*}
       \langle \hs^\gamma\hmm\rho'(D), \lambda \rangle
       = \langle \rho'(D), {}^{\gamma^{-1}}\hm\hm\lambda \rangle
       = \langle \rho(D), {}^{\gamma^{-1}}\hm\hm\lambda \rangle
       = \langle \hs^\gamma\hm \rho(D), \lambda \rangle
       = \langle \rho(\hs^\gamma\hmm D), \lambda \rangle
       =\langle \rho'(\hs^\gamma\hmm D), \lambda \rangle
       \text{,}
     \end{align*}
     and for every $D' \in \Dm$ we have
     \begin{align*}
       \langle \hs^\gamma\hmm\rho'(D), \lambda_{\hs^\gamma\hm D'} \rangle
       = \langle \rho'(D), {}^{\gamma^{-1}}\hm\hm\lambda_{\hs^\gamma\hm D'} \rangle
       = \langle \rho'(D), \lambda_{D'} \rangle = \begin{cases}
         q & \text{for ${}^\gamma D' = {}^\gamma D$,}\\
         0 & \text{for ${}^\gamma D' \ne {}^\gamma D$,}\end{cases}
     \end{align*}
     which shows that
     \[\langle \hs^\gamma\hmm\rho'(D), \lambda \rangle
     =\langle \rho'(\hs^\gamma\hmm D), \lambda \rangle\]
     for all $\lambda\in\Xm'$.
     Thus $\hs^\gamma\hmm\rho'(D)= \rho'(\hs^\gamma\hmm D)$, and hence,
     the map $\rho'$ is $\Gm$-equivariant, as required.
   \end{proof}

\end{subsec}

\begin{theorem}  \label{t:sphqs}
  Let $k_0$, $k$, and $\sG$ be as in Subsection \ref{ss:mod-Y}.
  Let $G/H$ be a spherical homogeneous space over $k$, and let
  $(\sX, \Vm, \Omone,\Omt)$ be its combinatorial invariants.
   Let $G_0$ be a \emphb{quasi-split} $k_0$-model of $G$.
 Assume that  the $\sG$-action
  defined by $G_0$  preserves the combinatorial invariants $(\sX, \Vm, \Omone,\Omt)$.
  Then for any continuous lift $\alpha_\Dm\colon \sG\times \Dm\to\Dm$
  of the $\sG$-action on $\Omega$,
  there exists a $G_0$-equivariant $k_0$-model $Y_0$ of $G/H$
  inducing this action $\alpha_\Dm$ on $\Dm$.
\end{theorem}

\begin{proof}
  Let $G'/H'$ be the quasi-affine spherical homogeneous space from
  Proposition~\ref{prop:invqaf} with combinatorial invariants
  $(\sX', \Sigma, \Omega'{}^{\ste{1}}, \Omega'{}^{\ste{2}})$, and let
  $X = \Spec(\Gamma(G'/H', \Om_{G'/H'}))$ be its affine closure.
  According to Grosshans \cite[Theorem~4.2]{gro97}, the codimension of
  $X\smallsetminus G'/H'$ in $X$ is at least~$2$, which means that
  there are no $G'$-invariant prime divisors in~$X$. In other words,
  the set of colors $\Dm'$ is the set of all $B'$-invariant prime
  divisors in~$X$. A~rational $B'$-semi-invariant function
  $f \in k(X)^{(B')}$ is regular, that is, $f \in k[X]^{(B')}$, if and
  only if $\nu_D(f) \ge 0$ for every $D \in \Dm'$ (that is, if and
  only if $f$ has no poles). Since by definition
  $\langle \rho'(D), \lambda \rangle = \nu_D(f)$ where $\lambda$ is
  the $B'$-weight of $f$, the weight monoid $\Gamma'$ of $X$ is given
  by
  \begin{align}\label{a:Gamma'}
    \Gamma' = \{\lambda \in \sX' \hs\mid\hs \text{$\langle \rho'(D), \lambda \rangle
     \ge 0$ for every $D \in \Dm'$} \}\text{.}
  \end{align}
  Set $G'_0 = G_0 \times C_0$ with $C_0$ as in Subsection~\ref{rem:gpp}. In
    particular, the $\sG$-action on $C$ defined by $C_0$ comes from
     $\alpha_\Dm$. Then the $\sG$-action defined by $G'_0$
  preserves
  $(\sX', \Sigma, \Omega'{}^{\ste{1}}, \Omega'{}^{\ste{2}})$, hence it
  also preserves~$\Gamma'$, which can be described in terms of the
  tuple $(\sX', \Sigma, \Omega'{}^{\ste{1}}, \Omega'{}^{\ste{2}})$;
  see \eqref{a:Gamma'}. By Theorem~\ref{th:qsaff}, there exists a
  $G'_0$-equivariant $k_0$-model $X_0$ of $X$, which induces a
  $G'_0$-equivariant $k_0$-model $X_0'$ of the open $G'$-orbit
  $G'/H' = (G\times C)/H'$ in~$X$. Its quotient by $C_0$ is a
  $G_0$-equivariant $k_0$-model $Y_0$ of $G/H$.
  By construction, the action of $\sG$ on
  $\Omega'{}^{\ste{1}} \cong \Dm$ is $\alpha_\Dm$,  as required.
\end{proof}

  \begin{theorem}\label{t:k0-point}
   The $k_0$-model of $G/H$ constructed in the proof of
      Theorem~\ref{t:sphqs} contains a $k_0$-point.
  \end{theorem}

  \begin{proof}
    We use the notation of the proof of Theorem~\ref{t:sphqs}.
    By Theorem~\ref{th:affp}, the set of $k_0$-points in $X_0$ is dense.
    Therefore, the induced $k_0$-model $X_0'$ of the open orbit
    $G'/H'\subset X$ contains a $k_0$-point $x_0'$, the image of which in $Y_0=X'_0/C_0$
    is a $k_0$-point of the  $k_0$-model $Y_0$ of $G/H$.
  \end{proof}

\begin{subsec}\label{ss:proof-main}
{\em Proof of  Theorem~\ref{t:sphfull'}.}
  Assume that the $\sG$-action on $X^*(B)$ and $\Sm$ defined by $G_0$
  preserves $(\sX, \Vm, \Omone,\Omt)$. Then the $\sG$-action defined
  by $G_q$ preserves $(\sX, \Vm, \Omone,\Omt)$ as well, because it is
  the same action. By virtue of Theorem~\ref{t:sphqs}, there exists a
  $G_\qs$-equivariant $k_0$-model of $G/H$ of the form $G_q/H_q$,
  where $H_q\subset G_q$ is a $k_0$-subgroup. By
  Proposition~\ref{p:tits}, the homogeneous space $G/H$ admits a
  $G_0$-equivariant $k_0$-model if and only if
  $\vktil_*(t(G_0))=1\in H^2(k_0,A_\qs)$.
  \qed
\end{subsec}

\section{Models of spherical embeddings}
\label{s:embeddings}

\begin{subsec}\label{ss:emb-1}
Let $k_0$, $k$, and $\sG$ be as in Subsection \ref{ss:mod-Y}.
Let $G/H$ be a spherical homogeneous
space of a (connected) reductive group $G$ defined over $k$. Let
$G/H\into Y^e$ be a spherical embedding, that is, a $G$-equivariant
open embedding of $G/H$ into a normal irreducible $G$ variety $Y^e$.

Let $\Delta$ denote the set of {\em $B$-invariant} prime divisors in $Y^e$. If we
identify the colors with their closures in $Y^e$, we have
$\Dm \subset \Delta$, and $\Delta \smallsetminus \Dm$ is the set of
{\em $G$-invariant} prime divisors in~$Y^e$.
We write $V= \Hom_\Z(\sX, \Qd)$.
Every prime divisor $D \in \Delta$ defines an element $\rho(D) \in V$.

We briefly recall how a spherical embedding $G/H\into Y^e$ defines a colored fan $\CF(Y^e)$;
we refer to Knop \cite{kno91} for details.
For every $G$-orbit $Z$ in $Y^e$ consider the set $I_Z \subset \Delta$ consisting of those
divisors that contain $Z$.
The finite set $I_Z$ defines a colored cone
\begin{align*}
  (\Cm_Z, \Fm_Z) = \big(\cone(\rho(I_Z)),\hs I_Z\cap \Dm\hs\big)\text{.}
\end{align*}
associated with the $G$-orbit $Z$, where we write $\cone(\rho(I_Z))$ for the cone in $V$
generated by the finite set $\rho(I_Z)$.    Then the set of colored cones
\begin{align*}
  \CF(Y^e) =\big\{ (\Cm_Z, \Fm_Z) \,\mid\, \text{$Z \subset Y^e$ is a $G$-orbit}\big\}
\end{align*}
is the colored fan associated with the spherical embedding
$G/H \hookrightarrow Y^e$.
\end{subsec}

\begin{subsec}\label{ss:G-stable}
Let $G_0$ be a $k_0$-model of $G$.
Assume that there exists a $G_0$-equivariant $k_0$-model of $G/H$.
Then $\sG$ naturally acts on $V$ by a continuous homomorphism
$\varphi\colon \sG\to\GL(V)$, and
for any $\gamma\in\sG$ and any cone $\Cm\subset V$
we obtain a new cone $\hs^\gamma\Cm=\varphi_\gamma(\Cm)$.
Furthermore, $\sG$ naturally acts on $\Omega\subset V\times \Pm(\Sm)$.

Now let us {\em fix} a $G_0$-equivariant $k_0$-model $Y_0$ of $G/H$.
It defines a continuous action $\alpha^\Dm$ of $\gal$ on $\Dm$
lifting the $\sG$-action on $\Omega$.
If $\gamma\in \sG$, then for any colored cone $(\Cm,\Fm)$ in $V$
we obtain a new colored cone $(\upgam\hs \Cm,\upgam \Fm)\coloneqq
(\hs\varphi_\gamma(\Cm),\hs\alpha^\Dm_\gamma(\Fm)\hs)$.
Following Huruguen \cite{Huruguen},
we say that the colored fan $\CF(Y^e)$ is {\em $\gal$-stable
with respect to $(G_0\hs,\alpha^\Dm)$}
if for any $\gamma\in\gal$ and for any colored cone $(\Cm,\Fm)$ in $\CF(Y^e)$,
the colored cone $(\varphi_\gamma(\Cm),\alpha^\Dm_\gamma(\Fm))$
is contained in $\CF(Y^e)$ as well.
\end{subsec}

\begin{construction}\label{con:Y^e}
With the  notation and assumptions of Subsection~\ref{ss:G-stable},
consider the group of $G$-equivariant  automorphisms
$\sA=\Aut^G(G/H)$. Note that $\Am$ is the group
of $k$-points of the algebraic group  $A=\sN_G(H)/H$.
The group $\Am$ naturally acts on the finite set $\Dm$.
Consider the homomorphism $\Am\to\Aut(\Dm)$. It is clear that the kernel
$\Am^{\kk}\coloneqq\ker[\Am\to\Aut(\Dm)]$ is the group of $k$-points
of some algebraic $k$-subgroup $A^{\kk}$ of $A$.
Then
$$A^\kk=\overline{H}/H\, \subset\, \NGH/H=A,$$
where $\overline{H}$ denotes the spherical closure of $H$
(this follows from the definition of spherical closure).

We have a canonical surjective map
\[\pi_\Dm=\rho\times\vs\colon\, \Dm\to\Omega.\]
Let $\Aut_\Omega(\Dm)$ denote the group of permutations $s$ of the finite set $\Dm$
such that $\pi_\Dm(s(D))=\pi_\Dm(D)$ for all $D\in\Dm$.
Our homomorphism $\Am\to\Aut(\Dm)$ is actually  a homomorphism
\begin{equation}\label{e:Losev-sur}
A\to\Aut_\Omega(\Dm).
\end{equation}
By Losev's theorem, see \cite[Theorem B.5]{BG}, the homomorphism \eqref{e:Losev-sur}
is surjective. Thus we have a short exact sequence
\begin{equation}\label{e:Ac}
 1\to A^{\kk}\to A \to \Aut_\Omega(\Dm)\to 1.
\end{equation}
\end{construction}

\begin{construction}\label{con:Y^e-2}
Let $G_0$ be a $k_0$-model of $G$, inducing a semilinear $\sG$-action
on $G$ and a continuous $\sG$-action on $\Omega$.
Assume that there exists a lifting of the $\sG$ action on $\Omega$
to a continuous $\sG$-action $\alpha^\Dm$ on $\Dm$
such that the colored fan $\CF(Y^e)$ is $\sG$-stable
with respect to $(G_0\hs,\alpha^\Dm)$.
Write $G_0=\hs_c G_\qs$, where $G_\qs$ is a quasi-split $k_0$-group
and $c\in Z^1(k, G_\qs/Z(G_\qs))$ is a cocycle.

By Theorem \ref{t:sphqs}, there exists a $G_\qs$-equivariant
$k_0$-model $Y_\qs$ of $G/H$ inducing this $\sG$-action
$\alpha^\Dm$ on $\Dm$ and a $G_\qs$-semilinear action $\mu^\qs$ on $G/H$.
This model $Y_\qs$ induces compatible semilinear $\sG$-actions on the algebraic groups
in the short exact sequence \eqref{e:Ac}, and so we obtain
a short exact sequence of commutative algebraic $k_0$-groups
\begin{equation}\label{e:Ac-qs}
 1\to A^{\kk}_\qs\labelto{i} A_\qs \labelto{\pi} \Aut_\Omega(\Dm)\to 1,
\end{equation}
(which actually does not depend on the choice of $\alpha^\Dm$;  see Appendix \ref{s:A}).
Consider the induced cohomology exact sequence
\begin{equation}\label{e:Ac-qs-coho}
 H^1(k_0, A_q)\labelto{\pi_*}H^1(k_0,\Aut_\Omega(\Dm)\hs)\labelto{\delta}H^2(k_0,A^{\kk}_\qs)
       \labelto{i_*}H^2(k_0,A_\qs).
\end{equation}
\end{construction}

\begin{proposition}\label{l:inj}
In the exact sequence of abelian groups \eqref{e:Ac-qs-coho},
the homomorphism $i_*\colon H^2(k_0,A^{\kk}_\qs)\to H^2(k_0,A_\qs)$ is injective.
\end{proposition}

\begin{proof}
We fix a $G_\qs$-equivariant $k_0$-model $Y_\qs$ of $G/H$. Let $\alpha_\qs^\sD$
denote the induced action of $\sG$ on $\sD$.
The group $H^1(k_0,\Aut_\Omega(\Dm)\hs)$ classifies $\sG$-actions $\beta^\Dm$ on $\Dm$
compatible with the given $\sG$-action on $\Omega$,
and the neutral element $1\in H^1(k_0,\Aut_\Omega(\Dm)\hs)$
corresponds to $\beta^\sD=\alpha_\qs^\sD$.
The group $ H^1(k_0, A_q)$ classifies $G_\qs$-equivariant $k_0$-models of $Y=G/H$,
and the neutral element $1\in H^1(k_0, A_q)$ corresponds to $Y_\qs$.
By Theorem \ref{t:sphqs}, for any such $\beta^\Dm$ there exists
a $G_\qs$-equivariant $k_0$-model $Y_\beta$ of $Y=G/H$
that induces this action $\beta^\Dm$ on $\Dm$.
Thus  the homomorphism $\pi_*\colon H^1(k_0, A_q)\to H^1(k_0,\Aut_\Omega(\Dm)\hs)$ is surjective.
It follows that the homomorphism $\delta$ is identically 1 and hence,
the homomorphism $i_*$ is injective.
\end{proof}

See Appendix \ref{s:A} for a proof of Proposition \ref{l:inj} not using Theorem \ref{t:sphqs}.

\begin{construction}\label{con:Y^e-3}
Since $Z(G)\subset B$, the composite homomorphism
\[Z(G)\to A\to \Aut_\Omega(\Dm)\]
of $Z(G)$ on $\Dm$ is trivial (that is,
it maps all elements of $Z(G)(k)$ to the identity).
It follows that the homomorphism of $k_0$-groups $\vk\colon Z(G_\qs)\to A_\qs$
comes from a homomorphism $\vk^{\kk}\colon Z(G_\qs)\to A^{\kk}_\qs$.
We obtain a natural homomorphism
\[\vktil^{\kk}\colon Z(\Gtil_0)=Z(\Gtil_\qs)
\labelto{} Z(G_\qs)\labelto{\vk^\kk} A^{\kk}_\qs,\]
and a commutative diagram
\begin{equation}\label{e:d:inj}
\begin{aligned}
\xymatrix{
H^2(k_0, Z(G_0)\ar[d]_{\vktil^\kk_*}\ar[rd]^{\vktil_*})\\
H^2(k_0,A^{\kk}_\qs)\ar@{^(->}[r]^{i_*} & H^2(k_0,A_\qs),
}
\end{aligned}
\end{equation}
in which the horizontal arrow is injective by Proposition \ref{l:inj}.
\end{construction}

\begin{proposition}\label{p:main2}
With the notation of Constructions \ref{con:Y^e} and \ref{con:Y^e-2},
let $Y=G/H$ be a spherical homogeneous space.
Let $G_0$ be a $k_0$-model of $G$ and write $G_0=\hs_c G_\qs$,
where $G_\qs$ is a quasi-split inner form of $G_0$.
Assume that the $\Gm$-action defined by $G_0$
preserves the combinatorial invariants of $G/H$
and that the following two equivalent conditions are satisfied:
\begin{enumerate}
\item[{\rm (a)}]  $\vktil_*(t(G_0))=1\in H^2(k_0,A_q)$, or, which is the same,
\item[{\rm (b)}] $\vktil^\kk_*(t(G_0))=1\in H^2(k_0,A^\kk_q)$.
\end{enumerate}
Then for any  continuous $\sG$-action $\alpha^\Dm$ on $\Dm$
lifting the $\sG$-action on $\Omega$ defined by $G_0$,
there exists a $G_0$-equivariant $k_0$-model $Y_0$ of $Y=G/H$
inducing this action $\alpha^\Dm$ on $\Dm$.
\end{proposition}

\begin{proof}
We see from the diagram \eqref{e:d:inj} that (a) and (b) are equivalent.
Let $\ctil\colon\sG\to G_\qs(k)$ be a locally constant lift
of the inner cocycle $c\colon \sG\to \Gbar(k)$
such that $\ctil_1=1_G$.
Since (b) holds, we have $\vk^{\kk}_*(\delta[c])=1$,
where $\vk^{\kk}\colon Z(G_\qs)\to A_\qs^{\kk}$ is the canonical homomorphism.
This  means that there exists a locally constant map
$a\colon\sG\to A^{\kk}_\qs(k)$ such that
\begin{equation}\label{e:neutral-bis}
a_\gamma\cdot \upgam a_{\beta  }\cdot\vk( \ctil_\gamma
\cdot  \upgam \ctil_{\beta  }  \cdot
            \ctil_{\gamma \beta  }^{\hs-1})\cdot a_{\gamma \beta  }^{-1}=1
\end{equation}
for all $\gamma,\beta  \in\sG$.

Let $\alpha^\Dm$ be as above.
Since $G_\qs$ is quasi-split,
by Theorem \ref{t:sphqs} there exists
a $G_\qs$-equivariant $k_0$-model $Y_\qs$ of $Y$ inducing  $\alpha^\Dm$.
Let $\mu\colon \sG\to\SAut(Y)$ denote the corresponding
$G_\qs$-equivariant semilinear action.
As in the proof of Theorem \ref{t:twist}, we set
\begin{equation}\label{e:mu0-bis}
\mu^0_\gamma=l(\ctil_\gamma)\circ a_\gamma\circ\mu_\gamma\hs;
\end{equation}
then $\mu^0_\gamma$ is a $G_0$-equivariant $\gamma$-semi-automorphism of $Y$.
Since \eqref{e:neutral-bis} holds, the proof of Theorem \ref{t:twist} shows
that $\mu^0$ is an algebraic $G_0$-equivariant semilinear $\sG$-action of $Y$,
which by Lemma \ref{p:lemma-5-4} defines a $G_0$-equivariant $k_0$-model $Y_0$ of $Y$.
Since  for all $\gamma\in\sG$ we have $a_\gamma\in A^{\kk}_\qs(k)$,
we see that $a_\gamma$ fixes all colors $D\in\Dm$, and hence, $\mu^0$
induces the same action $\alpha^\Dm$ on $\Dm$ as $\mu$ does, as required.
\end{proof}

\begin{theorem}\label{t:main2}
With the notation of Constructions \ref{con:Y^e},  \ref{con:Y^e-2}, and \ref{con:Y^e-3},
let $G/H\into Y^e$ be a spherical embedding.
Let $G_0$ be a $k_0$-model of $G$ and write $G_0=\hs_c G_\qs$,
where $G_\qs$ is a quasi-split inner form of $G_0$.
Assume that the $\Gm$-action defined by $G_0$
preserves the combinatorial invariants of $G/H$.
If $Y^e$ admits an algebraic $G_0$-equivariant semilinear $\sG$-action, then
the following assertions hold:
\begin{enumerate}
\item[{\rm (i)}]The $\sG$-action on $\Omega$ can be lifted
to a continuous $\sG$-action $\alpha^\Dm$ on $\Dm$ such that
the colored fan $\CF(Y^e)$ is $\sG$-stable with respect to $(G_0\hs,\alpha^\Dm)$;
\item [{\rm (ii)}]   $\vktil_*(t(G_0))=1\in H^2(k_0,A_q)$,
or, equivalently, $\vktil^\kk_*(t(G_0))=1\in H^2(k_0,A^\kk_q)$.
\end{enumerate}
Conversely, if conditions (i) and (ii) are satisfied,
then $Y^e$ admits an algebraic $G_0$-equivariant
semilinear $\sG$-action inducing the $\sG$-action $\smash{\alpha^\Dm}$ on $\Dm$ from (i).
\end{theorem}

\begin{proof}
If $Y^e$ admits an algebraic $G_0$-equivariant semilinear $\sG$-action $\mu^{0,e}$,
then it induces a continuous $\sG$-action $\alpha^\Dm$ on $\Dm$
lifting the $\sG$-action on $\Omega$, and then $\sG$ preserves $\CF(Y^e)$,
which gives (i); see Huruguen \cite[Theorem 2.23]{Huruguen}.
Moreover, clearly $\mu^{0,e}$ preserves the open $G$-orbit $Y=G/H$, and hence
induces an algebraic $G_0$-equivariant semilinear $\sG$-action $\mu^e$ on $Y$.
Since the variety $Y$ is quasi-projective,
by Lemma \ref{p:lemma-5-4} the algebraic $\sG$-action $\mu^e$ induces
a $G_0$-equivariant $k_0$-model $Y_0$ of $Y$.
By Theorem \ref{t:sphfull'} (proved in Subsection \ref{ss:proof-main}), (ii) holds.

Conversely, assume that conditions (i) and (ii) are satisfied.
Let $\alpha^\Dm$ be as in (i).
Since (ii) holds, by Proposition \ref{p:main2} there exists
a $G_0$-equivariant $k_0$-model $Y_0$ of $Y=G/H$  with semilinear action $\mu^0$ on $G/H$
inducing this action $\alpha^\Dm$ on $\Dm$.
Since by (i) the pair $(G_0,\alpha^\Dm)$ preserves $\CF(Y^e)$,
by Huruguen's theorem  \cite[Theorem 2.23]{Huruguen}
the $G_0$-equivariant  semilinear $\sG$-action $\mu^0$ on $G/H$
extends to a $G_0$-equivariant semilinear $\sG$-action $\mu^{0,e}$ on $Y^e$.

Note that Huruguen \cite{Huruguen}
does not mention the algebraicity property,
though he implicitly uses it when constructing
a $k_0$-model of $Y^e$ from $\mu^{0,e}$ via Galois descent
in the proof of his Proposition 2.21.
We show here that the semilinear $\sG$-action $\mu^{0,e}$ on $Y^e$ is algebraic.
By \cite[Corollary 5.7]{BG}
the variety $Y^e$ and the action of $G$ on  $Y^e$ are
``defined'' over some finite Galois extension $k_1/k_0$ in $k$.
This means that there exists a $G_1$-equivariant $k_1$-model $Y_1^e$ of $Y^e$,
where $G_1=G\times_{k_0}k_1$\hs.
This $k_1$-model  defines a $G_1$-equivariant $\sG_1$-action
\[\mu^{1,e}\colon\sG_1\to\SAut_{k/k_1}(Y^e),\]
where $\sG_1=\Gal(k/k_1)$.
The $\sG_1$-action $\mu^{1,e}$ preserves the open $G$-orbit $Y\subset
Y^e$ and thus defines a $G_1$-equivariant $k_1$-model $Y_1\subset Y_1^e$ of $Y$
and a $\sG_1$-action $\mu^1$ on $Y$.

Now we have two $k_1$-models $Y_0\times_{k_0}k_1$ and $Y_1$ of $Y$.
By \cite[Proposition 5.6(ii)]{BG} the composite isomorphism
\[(Y_0\times_{k_0}k_1)\times _{k_1} k\,\labelto{\sim}\,
            Y\,\labelto{\sim}\, Y_1\times_{k_1}k \]
is ``defined'' over some finite Galois extension $k_2/k_1$ in $k$.
It follows that the restrictions  of the semilinear actions $\mu^0$ and $\mu^1$  to
$\sG_2\coloneqq\Gal(k/k_2)$ coincide.
Since $Y$ is Zariski-dense in $Y^e$, we obtain that
that the restrictions to $\sG_2$ of the semilinear actions $\mu^{0,e}$
and $\mu^{1,e}$  coincide.
Since the restriction  of $\mu^{1,e}$ to $\sG_2$
comes from the $k_2$-model $Y_1^e\times_{k_1}\! k_2$ of $Y^e$,
we conclude that the $\sG$-action $\mu^{0,e}$ on $Y^e$ is algebraic.
\end{proof}

\begin{corollary}
\label{c:main2}
If the spherical embedding $Y^e$ in Theorem \ref{t:main2} is quasi-projective,
then $Y^e$ admits a $G_0$-equivariant $k_0$-model if and only if (i) and (ii) are satisfied.
\end{corollary}

\begin{proof}
The corollary follows from Theorem \ref{t:main2} and Lemma \ref{p:lemma-5-4}.
\end{proof}

We now consider the case of a not necessarily quasi-projective embedding $Y^e$.

\begin{definition}
  Let $(\sC, \sF), (\sC', \sF') \in \CF(Y^e)$. Then  the colored cone
  $(\sC', \sF')$ is said to be a \emph{face} of the colored cone $(\sC, \sF)$ if
  the cone $\sC'$ is a face of the cone $\sC$.
\end{definition}

\begin{remark}
  Here we consider only colored cones inside the fixed colored fan
  $\CF(Y^e)$. For the general definition of colored cones and their
  faces, we refer to Knop \cite[Section~3]{kno91}.
\end{remark}

\begin{definition}
  A subset $\Sigma \subset \CF(Y^e)$ is called a colored subfan of
  $\CF(Y^e)$ if for every $(\sC, \sF) \in \Sigma$ all its faces
  $(\sC', \sF')$ in $\CF(Y^e)$ belong to $\Sigma$.
\end{definition}

\begin{prop}
  \label{p:subsph}
  The map
  \begin{align*}
    \{\text{$G$-stable open subvarieties of $Y^e$}\}\, \longrightarrow\, \{\text{colored subfans of $\CF(Y^e)$}\},\ \
    Y' \mapsto \CF(Y')
  \end{align*}
  is a bijection.
\end{prop}
\begin{proof}
  This follows from  Knop \cite[Lemma~3.2]{kno91}.
\end{proof}

We call a colored fan \emph{quasi-projective} if the corresponding
spherical embedding is quasi-projective. A criterion for a colored fan
to be quasi-projective is given in Brion \cite[Corollary~3.3]{bri89}.
Following Huruguen \cite{Huruguen}, we consider coverings of $Y^e$ by
$G$-stable quasi-projective open subvarieties.

\begin{prop}  \label{p:nonqp}
  The spherical embedding $Y^e$ in Theorem \ref{t:main2} admits a
  $G_0$-equivariant $k_0$-model if and only if in addition to (i) and
  (ii), the following condition is satisfied:
  \begin{enumerate}
  \item[{\rm (iii)}] The colored fan $\CF(Y^e)$ is a union of
    quasi-projective colored subfans of $\CF(Y^e)$ that are $\sG$-stable with
    respect to $(G_0\hs,\alpha^\Dm)$.
  \end{enumerate}
\end{prop}

\begin{proof}
  Assume that $Y^e$ satisfies (i) and (ii). We fix the continuous
  $\sG$-action $\alpha^\Dm$ on $\Dm$ and a $G_0$-equivariant
  $k_0$-model $Y_0$ of $Y=G/H$ inducing $\alpha^\Dm$; see Proposition \ref{p:main2}.
  Then by the proof of Theorem~\ref{t:main2}, $Y^e$ admits
  an algebraic semilinear $\sG$-action $\mu^e$
  extending the $\sG$-action $\mu$ on $Y=G/H$ corresponding to the $k_0$-model $Y_0$.
  If, in addition, $Y^e$ satisfies (iii),
  then there exist $\sG$-stable quasi-projective
  colored subfans $\Sigma^{(1)}, \dots, \Sigma^{(n)} \subset
  \CF(Y^e)$ such that $\CF(Y^e) = \bigcup_{i=1}^n \Sigma^{(i)}$. For
  every $i \in \{1, \dots, n\}$, let $Y^{(i)} \subset Y^e$ be the
  $G$-stable open subvariety from Proposition~\ref{p:subsph} such that
  $\CF(Y^{(i)}) = \Sigma^{(i)}$. Then $Y^{(i)}$ is quasi-projective
  and $Y^e = \bigcup_{i=1}^n Y^{(i)}$. By the proof of Theorem \ref{t:main2},
  since $\Sigma^{(i)}$ is $\sG$-stable,  the $\sG$-action $\mu$ on $Y$
  extends to an algebraic semilinear $\sG$-action $\mu^{(i)}$ on $Y^{(i)}$.

  \begin{lemma}\label{l:res}
     The $\sG$-action  $\mu^{(i)}$ on $Y^{(i)}$ is the restriction to $Y^{(i)}$
     of the $\sG$-action $\mu^e$ on~$Y^e$.
  \end{lemma}

  \begin{proof}
    Let $\gamma\in\sG$. Consider the $k$-morphism
    \[ \nu\coloneqq (\mu^e_\gamma)^{-1}\circ\mu^{(i)}_\gamma\colon\, Y^{(i)}\to Y^{(i)}\to Y^e\]
    and the closed subvariety
    \[Z^{(i)}=\{y\in Y^{(i)}\mid\nu(y)=y\}.\]
    Since the restrictions of the $\sG$-actions $\mu^{(i)}$ and  $\mu^e$ to $Y$ coincide,
    we see that  $Z^{(i)}\supset Y$.
    Since $Y$ is dense in $Y^{(i)}$, we conclude that $Z^{(i)}=Y^{(i)}$.
    Therefore $\smash{\mu^{(i)}_\gamma(y)=\mu^e_\gamma(y)}$ for every $y\in Y^{(i)}$, as required.
  \end{proof}

  We resume proving Proposition \ref{p:nonqp}.
  Since $\smash{Y^{(i)}_0}$ is quasi-projective,
  by Corollary~\ref{c:main2} the $\sG$-action  $\mu^{(i)}$
  on $Y^{(i)}$ defines a $G_0$-equivariant $k_0$-model  $\smash{Y^{(i)}_0}$.
  Since the $k_0$-models $\smash{Y^{(1)}_0},
  \dots, \smash{Y^{(n)}_0}$ are defined by the restriction of the same
  $\sG$-action $\mu^e$ on $Y^e$, they glue together into a $G_0$-equivariant
  $k_0$-model $Y^e_0$ of $Y^e$.

  Conversely, assume that $Y^e$ admits a $G_0$-equivariant $k_0$-model
  $Y^e_0$. Let $Y_0$ denote the induced $k_0$-model of $Y=G/H$. Then
  $Y^e_0$ defines an algebraic $G_0$-equivariant semilinear
  $\sG$-action on $Y^e$, and by Theorem~\ref{t:main2} the assertions
  (i) and (ii) hold. Since $Y^e_0$ is a $k_0$-variety, there
  exist affine open $k_0$-subvarieties $\smash{U^{(1)}_0, \dots,
    U^{(n)}_0 \subset Y^e_0}$ such that $Y^e_0 = \bigcup_{i=1}^n
  \smash{U^{(i)}_0}$. We set $\smash{U^{(i)}=U^{(i)}_0\times_{k_0}k}$
  in $Y^e$; then each $U^{(i)}$ is a $\sG$-stable affine open
  subvariety of $Y^e$, and we have $Y^e = \bigcup_{i=1}^n U^{(i)}$.
  For every $i \in \{1, \dots, n\}$, by Sumihiro \cite[Theorem~3.8]{sum75} the
  open subvariety $Y^{(i)} \coloneqq G\cdot U^{(i)}$ of $Y^e$ is
  quasi-projective. It is $G$-stable, so that we may consider the
  quasi-projective colored fan $\Sigma^{(i)} \coloneqq \CF(Y^{(i)})$,
  which by Proposition~\ref{p:subsph} is a colored subfan of
  $\CF(Y^{e})$. Since $Y^e = \bigcup_{i=1}^n U^{(i)}$, we have $Y^e
  =\bigcup_{i=1}^n Y^{(i)}$, and hence $\CF(Y^e) =\bigcup_{i=1}^n
  \Sigma^{(i)}$. Since the $G$-action and the $\sG$-action on $Y^e$
  commute, each open subvariety $Y^{(i)}$ is $\sG$-stable, and
  therefore the $\sG$-action on $Y=G/H$ corresponding to the
  $k_0$-model $Y_0$ of $G/H$ extends to $Y^{(i)}$. By Huruguen
  \cite[Theorem~2.23]{Huruguen}, the quasi-projective colored fan
  $\Sigma^{(i)}$ of $Y^{(i)}$ is $\sG$-stable, which completes the
  proof.
\end{proof}

\begin{remark}
  If the spherical embedding $Y^e$ does not admit a $G_0$-equivariant
  $k_0$-model in the category of algebraic $k_0$-schemes because
  condition (iii) in Proposition~\ref{p:nonqp} is not satisfied, we
  still obtain from Theorem~\ref{t:main2} a $G_0$-equivariant
  $k_0$-model {\em in the category of algebraic $k_0$-spaces}; see Wedhorn
  \cite[Proposition~8.1]{wed18}.
\end{remark}

\section{Models of horospherical homogeneous spaces}
\label{s:horo}

This section is inspired by Moser-Jauslin and Terpereau \cite{MJTB},
where the case $k_0=\R$ was considered.

Let $k_0$, $k$, and $\sG$ be as in Subsection \ref{ss:mod-Y}.
Let $G$ be a (connected) reductive group over $k$.
We fix a Borel subgroup $B \subset G$ and
a maximal torus $T \subset B$.
Let $U$ denote the unipotent radical of  $B$.
The Bruhat decomposition shows that $U$ is a spherical subgroup of $G$.

\begin{definitions}
An algebraic subgroup $H\subset G$ is called {\em horospherical}
if $H$ contains $gUg^{-1}$ for some $g\in G(k)$.
A {\em horospherical homogeneous space of $G$} is a homogeneous space  of the form $Y=G/H$,
where $H$ is a horospherical subgroup of $G$.
\end{definitions}

Clearly, any horospherical subgroup $H$ is spherical,
because it contains a spherical subgroup of the form $gUg^{-1}$.

We describe the conjugacy classes of horospherical subgroups
following Pasquier \cite{Pasquier}.

\begin{subsec}\label{ss:BRD-BRD}
Let
\[\BRD(G)=\BRD(G,T,B)=(X,X^\vee,R,R^\vee,\Sm,\Sm^\vee)\]
denote the based root datum of $G$; see Subsection \ref{ss:BRD}.
If $\beta\in R$, we write $U_\beta$ for
the one-dimensional unipotent subgroup of $G$ with Lie algebra $\gg_\beta$.

Let $I$ be a subset of $\Sm$. We set
\[R^+_I=R^+\cap\sum_{\alpha\in I}\Z\alpha,\]
where $R^+\subset R$ is the set of positive roots with respect to $B$.
We denote by $P_I$ the parabolic subgroup of $G$
generated by the Borel subgroup $B$  and the one-dimensional
unipotent subgroups $U_{-\beta}$ for $\beta\in R^+_I$.
We say that $P_I$ is the
{\em standard parabolic subgroup of $G$ corresponding to $I$}.
It is well known that
\begin{equation}\label{e:P-I-char}
\X^*(P_I)=\big\{\chi\in X\, \mid\, \langle\chi,\alpha^\vee\rangle=0\ \forall\alpha\in I\big\},
\end{equation}
where $\alpha^\vee$ denotes the simple coroot corresponding to a simple root $\alpha$.
\end{subsec}

\begin{definition}
A {\em horospherical datum for} $G$ is a pair $(I,M)$,
where $I$ is a subset of $\Sm$ and $M$ is a subgroup of $X=\X^*(T)$ satisfying
\begin{equation}\label{e:s-datum}
\langle\chi,\alpha^\vee\rangle=0\quad\text{for all }\chi\in M,\ \alpha\in I.
\end{equation}
\end{definition}

\begin{construction}\label{con:horo-subgp}
For given $(T,B)$ and for a horospherical datum $(I,M)$,
we construct a horospherical subgroup $H_{I,M}\subset G$ as follows.
Let $P_I$ denote the standard parabolic subgroup of $G$ corresponding to $I$.
By \eqref{e:P-I-char}, we may regard elements $\chi\in M$ as characters of $P_I$.
We set
\[H_{I,M}=\bigcap_{\chi\in M}\ker[\chi\colon P_I\to\G_m].\]
Since $P_I\supset B\supset U$, and for any
$\chi\in M\subset\X^*(P_I)$ the restriction of $\chi$ to $U$ is trivial,
we see that $H_{I,M}\supset U$.
Thus  $H_{I,M}$ is a horospherical subgroup of $G$.
We denote by $[H_{I,M}]$ the conjugacy class of the subgroup $H_{I,M}$ in $G$.
\end{construction}

\begin{proposition}[{Pasquier \cite[Proposition 2.4]{Pasquier}}]
\label{p:Pasquier}
The construction
\[(I,M)\mapsto H_{I,M}\mapsto [H_{I,M}]\]
induces  bijections between the set of horospherical data for $G$,
the set of horospherical subgroups of $G$ containing $U$,
and the set of conjugacy classes of horospherical subgroups of $G$.
\end{proposition}

\begin{subsec}\label{ss:Y-A}
If $H\subset G$ is a horospherical subgroup, then by Proposition
\ref{p:Pasquier} the subgroup $H$ is conjugate
to a unique subgroup of the form $H_{I,M}$.
We say then that $I$ and $M$ are the {\em horospherical invariants of $H$}.
We say also that $H_{I,M}$ is the
{\em standard horospherical subgroup with horospherical invariants $(I,M)$}.

We consider $Y=G/H$.
We may and shall assume that $H=H_{I,M}\subset P_I\subset G$
with the notation of Construction \ref{con:horo-subgp}.
Write
\[ A=\sN_G(H)/H=\sN_G(H_{I,M})/H_{I,M}=P_I/H_{I,M}\hs;\]
then $A(k)=\Aut^G(G/H)$.
We see that $A$ is a $k$-torus with character group $\X^*(A)=M$.
\end{subsec}

\begin{subsec}\label{ss:H1-H2}
Let  $H_1\subset G$ be a horospherical subgroup with horospherical invariants $(I,M)$.
Let $Y_1=G/H_1$ be the corresponding horospherical homogeneous space.
Let $G_0$ be a $k_0$-model  of $G$, and let $\gamma\in \Gm$.
The $k_0$-model $G_0$ defines  a $\gamma$-semi-automorphism $\sigma_\gamma$ of $G$.
Set $H_2=\sigma_\gamma(H_1)\subset G$, and denote by $Y_2\coloneqq G/H_2$
the corresponding horospherical homogeneous space.

We wish to know whether the horospherical homogeneous spaces
$Y_1$ and $Y_2$ are isomorphic as $G$-varieties,
that is, whether $H_1$ and $H_2$ are conjugate in $G$.
For this end we compare the horospherical invariants of $H_1$ and $H_2$.
We write $\upgam\hmm I=\veg(I)$ and  $\upgam\hmm M=\veg(M)$,
where $\veg$ is the automorphism of the based root datum  $\BRD(G)$
defined in Construction \ref{con:*-action}.
\end{subsec}

\begin{proposition}\label{p:*-horo}
Let $k_0$, $k$, and $\Gm$  be as in Subsection \ref{ss:mod-Y}.
Let $G$ be a reductive group over $k$, and let $G_0$ be a $k_0$-model of $G$.
Let $H_1\subset G$ be a horospherical subgroup with horospherical invariants $(I,M)$,
and let $\gamma\in \Gm$.
Then the subgroup  $H_2\coloneqq \upgam\hm H_1$ is a horospherical subgroup
with horospherical invariants $(\upgam\hmm I,\upgam\hmm M)$.
\end{proposition}

\begin{proof}
Let $g_\gamma$, $\tau$, and $\veg$ be as in Construction \ref{con:*-action}.
We set $H'_2=\tau(H_1)\subset G$.
We have $H'_2=g_\gamma\cdot H_2\cdot g_\gamma^{-1}$, so  $H_2$ and $H_2'$ are conjugate, and
hence they have the same horospherical invariants.

We may and shall assume that $H_1=H_{I,M}$.
The automorphism $\veg$ of $\BRD(G,T,B)$
takes $R^+_I$ to $R^+_{\veg(I)}$\hs.
The semi-automorphism $\tau$ of $G$ acts also on $\Lie(G)$,
and it is clear that $\tau(\gg_{-\beta})=\gg_{-\veg(\beta)}$
for any $\beta\in R^+_I$.
 It follows easily that $\tau(P_I)=P_{\veg(I)}$.
We have
\begin{equation}\label{e:H-I-M}
\begin{aligned}
\tau(H_{I,M})\, &=\, \bigcap_{\chi\in M}\ker\left[\hs\veg(\chi)\colon P_{\veg(I)}\to\G_{m,k}\hs\right]\\
&=\, \bigcap_{\chi\in \veg(M)}\ker\left[\hs\chi\colon
   P_{\veg(I)}\to\G_{m,k}\hs\right]\, =\, H_{\veg(I),\hs\veg(M)}\hs.
\end{aligned}
\end{equation}
Thus $H_2'$ and  hence $H_2$ are horospherical subgroups with invariants
$\big(\veg(I),\veg(M)\big)=(\upgam\hmm I,\upgam\hmm M)$, as required.
\end{proof}

\begin{corollary}\label{c:*-horo}
If in Proposition \ref{p:*-horo} we set $Y=G/H_1$\hs,
then the $G$-variety $\upgam Y$ is isomorphic to $Y$
if and only if\/ $\upgam\hmm I=I$ and $\upgam\hmm M=M$.
\end{corollary}

\begin{proof}
We may assume that $H_1=H_{I,M}$. The $G$-variety $\upgam Y$  is isomorphic to $Y$ if and only if
the subgroup $\upgam\hmm H_{I,M}$ is conjugate to $H_{I,M}$; see,  for instance, \cite[Corollary 4.7]{BG}.
By Proposition \ref{p:*-horo}, the subgroup $\upgam\hmm H_{I,M}$ is a horospherical subgroup
with  horospherical invariants $(\upgam\hmm I,\upgam\hmm M)$.
By Proposition \ref{p:Pasquier}, the subgroup $\upgam\hmm H_{I,M}$ is conjugate to $H_{I,M}$
if and only if $(\upgam\hmm I,\upgam\hmm M)=(I,M)$, as required.
\end{proof}

\begin{proposition}\label{p:conj-(I,M)}
Let $k_0$,  $k$, and $\Gm$  be as in Subsection \ref{ss:mod-Y}.
Let $G$ be a reductive group over $k$, and let $G_0$ be a {\emm quasi-split} $k_0$-model of $G$.
Let $H\subset G$ be a horospherical subgroup with $\Gm$-stable horospherical invariants $(I,M)$,
that is, such that $(\upgam\hmm I,\upgam\hmm M)=(I,M)$ for all $\gamma\in\Gm$.
Then $H$ is conjugate to a subgroup $H_0$ defined over $k_0$,
and  $G/H$ admits the $G_0$-equivariant $k_0$-model  $G_0/H_0$ with a $k_0$-point $1\cdot H_0$.
\end{proposition}

\begin{proof}
Let $B_0\subset G_0$ be a  Borel subgroup of the quasi-split $k_0$-group $G_0$,
and let $T_0\subset B_0$ be a maximal torus.
Consider the Borel pair  $(T,B)=(T_0,B_0)\times_{k_0}k$.
We construct the standard horospherical subgroup $H_{I,M}$
{\em starting with this Borel pair $(T,B)$.}
Then $H$ is conjugate to $H_{I,M}$\hs.
Let $\gamma\in\sG$.
Since the Borel pair $(T,B)$ comes from the Borel pair $(T_0,B_0)$ over $k_0$\hs,
we may take  $g_\gamma=1$  with the notation of Construction \ref{con:*-action}.
Then we have  $\tau=\sigma_\gamma$, and hence
\[\upgam\hm(H_{I,M})=\sigma_\gamma(H_{I,M})=\tau(H_{I,M})=
    H_{\ve_\gamma(I),\hs\ve_\gamma(M)}=H_{\hs\upgam\hm I,\upgam\hm M}\]
(we use formula \eqref{e:H-I-M}\hs).
Since by assumption $(\upgam\hm I,\upgam\hm M)=(I,M)$,
we obtain that  $\upgam\hm(H_{I,M})=H_{I,M}$.
We see that $H_{I,M}$ is $\Gm$-invariant, and hence, defined over $k_0$,
that is, comes from some $k_0$-subgroup $H_0\subset G_0$.
The homogeneous space $G/H$ has a $G_0$-equivariant
$k_0$-model $G_0/H_0$ with a $k_0$-point $1\cdot H_0$
\end{proof}

\begin{remark}
Proposition \ref{p:conj-(I,M)} generalizes \cite[Corollary 3.12]{MJTB},
where the case $k_0=\R$ was considered,
with a similar proof.
\end{remark}

\begin{remark}\label{r:comb-inv}
The combinatorial invariants $(\Xm,\Vm,\Omone,\Omt)$ of
a {\em horospherical} spherical variety $G/H_{I,M}$
can be computed from $(I,M)$. Namely,
\[
\Xm=M\subset\X^*(T),\quad \Vm=\Hom(M,\Q),\quad
\Omone=\big\{\hs(\alpha^\vee|_M,\{\alpha\})\,
\mid\, \alpha\in\Sm\smallsetminus I\big\},\quad \Omt=\varnothing.
\]
We see that Proposition \ref{p:conj-(I,M)} follows
from (difficult) Theorems \ref{t:sphqs} and \ref{t:k0-point},
proved for all spherical homogeneous spaces, not necessarily horospherical.
\end{remark}

\begin{theorem}\label{t:tits-horo}
Let $k_0$, $k$, and $\sG$ be as in Subsection \ref{ss:mod-Y}.
Let $G$ be a  reductive group over $k$.
Let $H\subset G$ be a  horospherical
subgroup with horospherical invariants $(I,M)$.
Let $G_0$ be a $k_0$-model of $G$.
Assume that $H$ has $\Gm$-stable horospherical invariants
with respect to the $*$-action defined by the $k_0$-model $G_0$ of $G$, that is,
$(\upgam\hmm I,\upgam\hmm M)=(I,M)$ for all $\gamma\in\Gm$\hs.
Let $G_\qs$ be a quasi-split inner form of $G_0$\hs.
  Set  $\smash{A_q=\sN_{G_q}(H_q)/H_q}$.
  Then $G/H$ admits  a $G_0$-equivariant $k_0$-model if and only if the image in
  $\smash{H^2(k_0,A_\qs)}$ of the Tits class
  \[t(\Gtil_0)\in H^2(k_0,Z(\Gtil_0))=H^2(k_0,Z(\Gtil_\qs))\] under the map
\[\vktil_*\colon\, H^2(k_0,Z(\Gtil_\qs))\labelto{\pi_*}
    H^2(k_0,Z(G_\qs))\labelto{\vk_*}  H^2(k_0,A_\qs)\]
 is 1.
\end{theorem}

\begin{proof}
Since $G_0$ is an inner form of $G_\qs$, these two $k_0$-models of $G$,
namely, $G_0$ and $G_\qs$\hs, define the same $*$-action of $\Gm$ on  $(\X^*(T), \Sm)$.
It follows that $H$ has $\Gm$-stable horospherical invariants
with respect to the $*$-action defined by the $k_0$-model $G_\qs$ of $G$.
Since the model $G_q$ of $G$ is quasi-split,
by Proposition \ref{p:conj-(I,M)}
the homogeneous space $G/H$ admits a $G_\qs$-equivariant $k_0$-model
of the form $G_q/H_q$, where $H_q\subset G_q$ is a $k_0$-subgroup.
The group $A_\qs$ is abelian
(namely, $A_\qs$ is the $k_0$-torus with character group $M$), and therefore,
 $H^2(k_0,A_\qs)$ has a unique neutral element 1.
Now the theorem follows from Proposition \ref{p:tits}.
\end{proof}

We wish to write explicitly the cohomological condition on the image
of the Tits class in Theorem \ref{t:tits-horo}.

\begin{proposition}\label{p:A1-Bn...}
Let $G$ be a {\emm simply connected, simple} $k$-group of
one of the types $\Af_1$, $\BB_n$ $(n\ge 2)$,
$\CC_n$ $(n\ge 3)$, or $\EE_7$. Let $G_0$ be a $k_0$-model of $G$.
Let $H\subset G$ be a horospherical subgroup.
\begin{enumerate}
\item[\rm (i)] If either $t(G_0)=1$ or $H\supseteq Z(G)$,
then $G/H$ admits a  unique $G_0$-equivariant $k_0$-model.
\item[\rm (ii)] If $t(G_0)\neq 1$ and $H\not\supseteq Z(G)$,
then $G/H$ does not admit a $G_0$-equivariant $k_0$-model.
\end{enumerate}
\end{proposition}

\begin{proof}
We deduce the proposition from Lemma \ref{l:A1-Bn-C} below.

Under our assumptions, the Dynkin diagram of $G$ has no nontrivial automorphisms,
and therefore,  the condition $(\upgam I,\upgam M)=(I,M)$ is trivially satisfied.
Since $G$ is simply connected, the homomorphism $\vktil_*$ coincides with $\vk_*$.
If either $t(G_0)=1$ or $H\supseteq Z(G)$, then $\vk_*(t(G_0))=1$,
and by Theorem \ref{t:tits-horo} $G/H$ admits a $G_0$-equivariant $k_0$-model.
Since $H$ is horospherical, $A_q$ is a torus.
Moreover, under our assumptions the only quasi-split model of $G$ is split,
and therefore $G_0$ is an inner form of a split $k_0$-model of $G$.
By Lemma \ref{l:A1-Bn-C}(i) below, $A_q$ is a {\em split} torus,
and by Hilbert's Theorem 90 we have $H^1(k_0,A_q)=1$.
Since $\Aut^{G_q}(Y_q)=A_q(k)$,
we see that $H^1(\Gm,\Aut^{G_\qs}(Y_\qs)\hs)=1$, and hence
$G/H$ admits a {\em unique} $G_0$-equivariant $k_0$-model, which proves  (i).

Under our assumptions, $Z(G)\simeq\mu_2$, where we write $\mu_2$
for the $k$-group $\{\pm 1\}$ of square roots of unity.
If $H\not\supseteq Z(G)$, then $Z(G)\cap H=\{1\}$, and by Lemma \ref{l:A1-Bn-C}(ii) below,
the homomorphism $\vk_*$ is injective.
If, moreover,  $t(G_0)\neq 1$, then we obtain that $\vk_*(t(G_0))\neq 1$,
and by Theorem \ref{t:tits-horo}
$G/H$ does not admit a $G_0$-equivariant $k_0$-model, which proves (ii).
\end{proof}

\begin{lemma}\label{l:A1-Bn-C}
Let $k$, $k_0$, and $\Gm$ be as in Subsection \ref{ss:mod-Y}.
Let $G$ be a reductive $k$-group.
Let $H\subset G$ be a spherical subgroup such that the group $A=\sN_G(H)/H$ is connected.
Let $G_0$ be $k_0$-model of $G$ and assume that $G_0$ is an inner form
of a {\emm split} $k_0$-model of $G$.
Then
\begin{enumerate}
\item[\rm (i)] $A_q$ is a split $k_0$-torus;
\item[\rm (ii)] If, moreover,  $Z(G)\cap H=\{1\}$,
then the homomorphism $\vk_*\colon H^2(k_0,Z(G_0))\to H^2(k_0,A_\qs)$ is injective.
\end{enumerate}
\end{lemma}

\begin{proof}
Since $G_0$ is an {\em inner form} of a {\em split} $k_0$-model of $G$,
the Galois group $\Gm$ acts (by $*$-action) on $X=\X^*(T)$ trivially,
and hence, it acts on $\Xm\subset X$ trivially.
Since $\X^*(A_q\times_{k_0} k)$ is a quotient group of $\Xm$,
we see that $\Gm$ acts on  it trivially.
Since by hypothesis $A$ is connected, hence a torus,
we conclude that $A_q$ is a  split $k_0$-torus, which proves (i).

If $Z(G)\cap H=\{1\}$, then
the homomorphism $\vk\colon Z(G)\into \sN_G(H)\onto \sN_G(H)/H=A$ is injective.
We identify $Z(G_0)$ with its image in $A_q$\hs.
Consider the short exact sequence
\[ 1\to Z(G_0)\labelto{\vk} A_q\to A_q/Z(G_0)\to 1\]
and the induced cohomology exact sequence
\begin{equation}\label{e:vk-90}
H^1(k_0,A_q/Z(G_0))\to H^2(k_0,Z(G_0))\labelto{\vk_*} H^2(k_0, A_q).
\end{equation}
Since $A_\qs$ is a split $k_0$-torus, the $k_0$-torus $A_q/Z(G_0)$ is split as well,
and by Hilbert's Theorem 90 we have $H^1(k_0,A_q/Z(G_0))=1$.
Now we see from the exact sequence \eqref{e:vk-90}
that the homomorphism $\vk_*$ is injective, which proves (ii).
\end{proof}

\section{Horospherical homogeneous spaces
     over local fields and number fields}
\label{s:horo-l-n}

In this section we consider the case when $k_0$ is $\R$,
or a $p$-adic field, or a number field.
Recall that a $p$-adic field is a finite extension
of  the field $\Q_p$ of $p$-adic numbers,
and a number field is a finite extension of $\Q$.
We retain the notation of the previous section.

\begin{subsec}
Assume that $G$ is a {\em simply connected} semisimple group.
Let $T\subset B\subset G$ be as in \ref{ss:Luna-Vust}.
We write $\Pf $ for $\X^*(T)$; it is the {\em weight lattice of} $G$.
We write $\Qf  \subset \Pf $ for the {\em root lattice of} $G$,
that is, $\Qf  =\X^*(T/Z(G))$.
Then $\X^*(Z(G))\cong\Pf/\Qf$.
The canonical epimorphism $\Pf \to \Pf/\Qf  $ is dual
to the inclusion map $Z(G)\into T$.

Let  $H=H_{I,M}\subset G$  be the standard horospherical subgroup
with horospherical invariants $(I,M)$ with respect to $T$ and $B$.
Note that $\sN_G(H)=P_I\supset T$.
It follows that the canonical homomorphism
\[\vk\colon Z(G)\into \sN_G(H)\onto A=\sN_G(H)/H\]
factors as
\[Z(G)\into T\to A,\]
where the homomorphism $T\to A$ is surjective
because $T\cdot H_{I,M}=P_I=\sN_G(H_{I,M})$.
By duality we obtain a homomorphism $M\to \Pf/\Qf  $, which factors as
\begin{equation}\label{e:M-X-X/Q}
M\into \Pf \onto \Pf/\Qf  ,
\end{equation}
where $M\into \Pf=\X^*(T)$ is the inclusion map, and $\Pf \onto \Pf/\Qf  $
is the canonical epimorphism.

Let $G_0$ be a $k_0$-model of $G$,
an inner form of a quasi-split model $G_q$.
Let $B_q\subset G_q$ be a Borel subgroup, and
$T_q\subset B_q$ be a maximal torus, both defined over $k_0$.
We assume that the standard parabolic subgroup $P_I$
and the standard horospherical subgroup $H_{I,M}$
are constructed starting from $T=T_q\times_{k_0}k$
and $B=B_q\times_{k_0}k$.
We assume also that $\Gm$ preserves $(I,M)$.
Then $\Gm$ acts on $M=\X^*(A)$ and thus defines a $k_0$-model $A_q$ of $A$.
The homomorphisms in \eqref{e:M-X-X/Q} are $\Gm$-equivariant,
and they induce a homomorphism
\begin{equation}
\vk^*\colon M^\Gm\into \Pf ^\Gm\to (\Pf/\Qf  )^\Gm.
\end{equation}
\end{subsec}

\begin{subsec}\label{ss:Theta}
Assume that  $t=t(G_0)\in H^2(k_0,Z(G_q))$ is {\em nontrivial}.
In the case when $k_0$ is $\R$ or $k_0$ is a $p$-adic field,
we write the conditions under which $\vk_*(t)=1 \in H^2(k_0,A_q)$.

If  $k_0$ is a $p$-adic field, then $t\in H^2(\Gm,Z(G))$
defines by cup product a homomorphism
\[t^0\colon (\Pf/\Qf  )^\Gm\to \Br k_0\cong\Q/\Z,\]
where $\Br k_0\coloneqq H^2(k_0,\G_m)$ is the Brauer group of $k_0$;
see formula \eqref{e:Br-char} in Appendix \ref{s:A-TN}.
If $k_0=\R$, then $t$  defines by cup product a homomorphism
\[t^0\colon (\Pf/\Qf  )^\Gm\to \Br k_0\cong\half\Z/\Z\subset\Q/\Z.\]
In both cases we regard $t^0$ as a homomorphism to $\Q/\Z$.
By Proposition \ref{p:t0} in Appendix \ref{s:A-TN}, the condition
 $\vk_*(t)=1\in H^2(\Gm, A)$ is satisfied if and only if the homomorphism
\begin{equation}\label{e:t0-vk}
 t^0\circ \vk^*\colon\, M^\Gm\into \Pf ^\Gm\to
      (\Pf/\Qf  )^\Gm\labelto{t^0} \Q/\Z
\end{equation}
is identically zero.

Since $t\neq 1$, by Lemmas \ref{l:dual-R} and
\ref{l:dual-p} we have $t^0\neq 0$.
Write
\[\Theta=\ker(t^0)\subset (\Pf/\Qf)^\Gm.\]
Since $t^0\neq 0$, we see that $\Theta\neq (\Pf/\Qf)^\Gm$.
Let $\Theta_\Pf\subset \Pf^\Gm$ denote
the preimage in $\Pf^\Gm$ of $\Theta\subset (\Pf/\Qf)^\Gm$.
From \eqref{e:t0-vk} we see that
 $ t^0\circ \vk^*=0$ if and only if
 \begin{equation}\label{e:cohom-cond}
 M^\Gm\subset\Theta_\Pf \hs.\tag{$*$}
 \end{equation}
\end{subsec}

\begin{subsec}\label{ss:simple-R-p}
Now assume that $G$ is a  {\em simply connected, simple}
algebraic group over $k$,
and $G_0$ is a $k_0$-model of $k$.
We compute $\Theta_\Pf \subset \Pf ^\Gm$ for all such $G$ and  $G_0$
over the field $k_0=\R$ and over the $p$-adic fields.
We use the notation of Tits \cite[Table II, starting p.~54]{Tits}.
We consider five cases:

(1) $G$ is of one of the types $\Af_1$, $\BB_n$ $(n\ge 2)$,
$\CC_n$ $(n\ge 3)$, and $\EE_7$\hs.
Then $\Pf/\Qf  $ is of order 2, $\Theta=\{0\}$, $\Theta_\Pf =\Qf$,
and the condition \eqref{e:cohom-cond} on $M$ is
\begin{equation}\label{e:*1}
M\subset \Qf .\tag{$*1$}
\end{equation}
See also Proposition \ref{p:A1-Bn...}.

(2) $G_q$ is of one of the types $\hs^2\Af_{2m-1}$ $(m\ge2)$,
$\hs^2\DD_{2m+1}$ $(m\ge2)$, and $\hs^2\DD_{2m}$ $(m\ge2)$.
Then $(\Pf/\Qf  )^\Gm$ is of order 2, $\Theta=\{0\}$, $\Theta_\Pf =\Qf^\Gm$,
and the condition \eqref{e:cohom-cond} on $M$ is
\begin{equation}\label{e:*2}
M^\Gm\subset \Qf^\Gm.\tag{$*2$}
\end{equation}
This does not imply that $M\subset \Qf$;
see Example \ref{ex:SU(6)} below.
Note that the case $\hs^2\DD_{2m}$ does not appear over $k_0=\R$,
because $H^2(\R,Z(G_q))=1$ in this case.

(3)  $G_q$ is of one the types $\hs^1\Af_{n-1}$ $(n\ge3)$,
$\hs^1\DD_{2m+1}$ $(m\ge2)$, and $\hs^1\EE_6$.
For $\hs^1\DD_{2m+1}$ we put $n=4$.
For $\hs^1\EE_6$ we put $n=3$.
Then $\Pf/\Qf  $ is a cyclic group of order $n$.
The element $t\in H^2(\Gm,Z(G))$ is of order $l>1$ such that $l|n$.
Then
\[\Pf/\Qf  \simeq \tfrac{1}{n}\Z/\Z,\quad \Theta=l\cdot(\Pf/\Qf  )=
(l\Pf +\Qf)/\Qf,\quad \Theta_\Pf =l\Pf +\Qf,\]
and the condition \eqref{e:cohom-cond} on $M$ is
\begin{equation}\label{e:*3}
M\subset l\Pf +\Qf.\tag{$*3$}
\end{equation}
Note that over $k_0=\R$ we have $l=2$,
which excludes the cases  $\hs^1\Af_{n-1}$ with odd $n$
and $\hs^1\EE_6$ with $n=3$ because of the requirement $l|n$.

(4--5) $G_q$ is of the type  $\hs^1\DD_{2m}$ $(m\ge2)$. Write $n=2m$.
We use the notation of Bourbaki \cite[Plate IV]{bou02}.
In particular,
\[\alpha_1,\dots,\alpha_n\in\Qf,\quad
\alpha_1=\ve_1-\ve_2,\ \alpha_2=\ve_2-\ve_3,\
\dots, \ \alpha_{n-1}=\ve_{n-1}-\ve_n,\ \alpha_{n}=\ve_{n-1}+\ve_n\]
are the simple roots, and
$\omega_1,\dots,\omega_{n-1},\omega_n\in \Pf$ are the fundamental weights.
Then $\Qf$ is the sublattice of \, $\oplus\Z\ve_i$ \,
consisting of elements with {\em even} sum of (integer) coordinates.
The homomorphism
\[(\omega_{n-1},\omega_{n})\colon T\to\G_{m,k}\times\G_{m,k}\]
induces  isomorphisms
\[Z(G)\isoto\mu_2\times\mu_2\quad\text{and}\quad
    \Z/2\Z\,\oplus\hs\Z/2\Z\,\isoto\, \Pf/\Qf.\]
We consider
\[t^0\colon \Pf/\Qf  \to \half\Z/\Z\subset\Q/\Z.\]
Up to an automorphism of $G$, we have either
$t^0=(\half,\half)$ or $t^0=(0,\half)$.

If $t^0=(\half,\half)$, then the condition \eqref{e:cohom-cond} on $M$ is
\begin{equation}\label{e:*4}
M\subset \oplus\Z\ve_i\hs.\tag{$*4$}
\end{equation}
If  $t^0=(0,\half)$, then the condition \eqref{e:cohom-cond} on $M$ is
\begin{equation}\label{e:*5}
M\subset \langle \Qf,\omega_{n-1}\rangle.\tag{$*5$}
\end{equation}

Note that in the cases $\EE_8$, $\FF_4$, and $\GG_2$
we have $Z(G)=1$ and hence $t(G_0)=1$; therefore, these cases do not appear.
Moreover, the  cases
$\hs^2\Af_{2m}$, $\hs^2\EE_6$, $\hs^3\DD_4$, and $\hs^6\DD_4$
do not appear either, because  in these cases $\X^*(Z(G))^\Gm=0$,
and by Tate duality (see Subsections  \ref{ss:case-R} and \ref{ss:case-p})
we have  $H^2(\Gm, Z(G))=1$.

We state the results of this subsection as a theorem.
\end{subsec}

\begin{theorem}\label{t:horo-types}
Let $k$, $k_0$, $\Gm$ be as in as in Subsection \ref{ss:mod-Y}.
Moreover, assume that either $k_0=\R$ or $k_0$ is a $p$-adic field.
Let $G$ be a  {\emm simply connected, simple}
algebraic group over $k$, and $G_0$ be a $k_0$-model of $G$.
Let $t=t(G_0)\in H^2(k_0,Z(G_0))$ be the Tits class of $G_0$,
and assume that $t\neq 1$.
Let $H\subset G$ be a horospherical subgroup with invariants $(I,M)$,
and assume that the Galois group $\Gm$ preserves $I$ and $M$.
Then  $G/H$ admits a $G_0$-equivariant $k_0$-model
if and only if $M$ satisfies condition $(*i)$,
where $i=1,2,3,4,5$, depending on the type of $G$ and on $t$.
\end{theorem}

\begin{proof}
By Theorem \ref{t:tits-horo}, $G/H$ admits
a $G_0$-equivariant model  if and only if
$$\vk_*(t)=1\in H^2(k_0,Z(G_q)).$$
From Subsection \ref{ss:Theta} we know that $\vk_*(t)=1$
if and only if condition \eqref{e:cohom-cond} is satisfied.
By Subsection \ref{ss:simple-R-p},  condition \eqref{e:cohom-cond}
is satisfied if and only if condition $(*i)$ is satisfied.
\end{proof}

\begin{remark}
In order to use Theorem \ref{t:horo-types},
one needs to know the Tits class $t(G_0)$
for each $k_0$-model $G_0$ of
a simply connected simple $k$-group $G$,
where $k_0=\R$ or $k_0$ is a $p$-adic field.
For the case $k_0=\R$, the Tits classes
are tabulated in the appendix to \cite{MJTB}.
For the $p$-adic case, $t(G_0)=1$
if and only $G_0$ is quasi-split;
see Proposition \ref{p:q-s} below.
\end{remark}

\begin{proposition}[well-known]
\label{p:q-s}
If $G_0$ is a simply connected semisimple group
over a {\emm $\boldsymbol{p}$-adic} field $k_0$,
then   $t(G_0)=1\in H^2(k_0, Z(G_0)\hs)$ if and only if  $G_0$ is quasi-split.
\end{proposition}

\begin{proof}
For the reader's convenience we provide a short proof.
We know that $G_0$ is an inner form of a quasi-split group, say, $G_0=\hs_c G_q$,
where $G_q$ is a quasi-split group and $c\in Z^1(k_0,G/Z(G)\hs)$.
We have
$$t(G_0)=\delta[c] \,\in \, H^2(k_0, Z(G_q)\hs)=H^2(k_0, Z(G_0)\hs),$$
where $\delta$ is the connecting map in the cohomology exact sequence
\[ H^1(k_0,G_q)\to H^1(k_0,G_q/Z(G_q)\hs)\labelto{\delta}H^2(k_0, Z(G_q)\hs)\]
corresponding to the short exact sequence
\[1\to Z(G_q)\to G_q\to G_q/Z(G_q)\to 1.\]
If $G_0$ is quasi-split, then we may take $G_q=G_0$ and $c=1$.
It follows that $t(G_0)=\delta[1]=1$.
Conversely, if $\delta[c]=t(G_0)=1$,
then $[c]$ lies in the image of $H^1(k_0, G_q)$.
However, since $G_q$ is simply connected and $k_0$
is $p$-adic, by Kneser's theorem
(see Platonov and Rapinchuk \cite[Theorem 6.4]{PR})
we have $H^1(k_0, G_q)=\{1\}$.
It follows that $[c]=1$, hence $G_0\simeq G_q$,
and thus $G_0$ is quasi-split, as required.
\end{proof}

\begin{subsec}\label{ss:number}
We consider the case of a number field $k_0$.
Assuming that $G$ is simple and simply connected,
we reduce the cohomological condition $\vk_*(t(G_0))=1$
over a number field to the similar conditions over local fields,
treated in Theorem \ref{t:horo-types}.

Let $k$, $k_0$, and $\sG$ be as in Subsection \ref{ss:mod-Y},
and assume that $k_0$ is a number field.
We write $\Pl(k_0)$ for the set of places of $k_0$.
For a place $v\in \Pl(k_0)$, let $k_{0,v}$
denote the completion of $k_0$ at $v$.
Let $G$ be a  simply connected, simple $k$-group,
and $G_0$ be a $k_0$-model of $G$.
\end{subsec}

\begin{theorem}\label{t:number-HP}
Let $k$, $k_0$, $\sG$, $G$, $G_0$ be as in Subsection \ref{ss:number},
in particular, $G_0$ be a $k_0$-model of a
{\emm simply connected,  simple} $k$-group $G$, where $k_0$ is a number field.
Let $H\subset G$ be a  horospherical subgroup
with horospherical invariants $(I,M)$, and assume that $\Gm$ preserves $(I,M)$.
Let $t=t(G_0)\in H^2(k_0,Z(G_0)\hs)$ be the Tits class.
For $v\in \Pl(k_0)$, we write $t_v=\loc_v(t)\in H^2(k_{0,v}\hs,\hs Z(G_0)\hs)$,
where
$$\loc_v\colon\hs H^2(k_{0},Z(G_0)\hs)\hs \to\hs H^2(k_{0,v}\hs,\hs Z(G_0)\hs)$$
denotes the localization map.
Then $\vk_*(t)=1\in H^2(k_0,A_q)$ if and only if
$\vk_{*,v}(t_v)=1\in H^2(k_{0,v},A_q)$ for all $v\in \Pl(k_0)$,
where
$$\vk_{*,v}\colon\hs H^2(k_{0,v}\hs,\hs Z(G_0)\hs)\hs
    \to\hs H^2(k_{0,v}\hs,\hs A_q)$$
is the map defined by the $k_{0,v}$\hs-group
$G_{0,v}\coloneqq G_0\times_{k_0} k_{0,v}$.
\end{theorem}

\begin{proof}
We consider the localization map
\[\loc\colon H^2(k_0,A_q)\,\lra \prod_{v\in\Pl(k_0)} H^2(k_{0,v}\hs,\hs A_q)\]
and the Tate-Shafarevich group
\[ \Sha^2(k_0,A_q)=\ker\left[H^2(k_0,A_q)\labelto{\loc}
      \prod_{v\in\Pl(k_0)} H^2(k_{0,v}\hs,\hs A_q)\right].\]
We consider also
\[ \Sha^1(k_0,M)=\ker\left[H^1(k_0,M)\labelto{\loc}
    \prod_{v\in\Pl(k_0)} H^1(k_{0,v}\hs,\hs M)\right].\]

Let  $G_q$ be a quasi-split inner form of $G_0$, $B_q\subset G_q$
be a Borel subgroup (defined over $k_0$),
and $T_q\subset B_q$ be a maximal torus.
The Galois group $\Gm$ acts on
$\Pf \coloneqq \X^*(T_q\times_{k_0}k)$
via the automorphism group
 of the Dynkin diagram $\Dyn(G,T,B)$,
that is, via a group isomorphic to $\{1\}$,
or $\Z/2\Z$, or $\Z/3\Z$, or $S_3$
(the symmetric group on three letters).
Each of these groups is metacyclic
in the sense of Sansuc \cite[p.~13]{Sansuc},
that is, all its Sylow subgroups are cyclic.
Since $M\subset \Pf $, it follows that $\Gm$
acts on $M=\X^*(A_q\times_{k_0}k)$
via a metacyclic quotient group.
By \cite[Lemma 3.4]{Borovoi-Ramanujan},
it follows  that $\Sha^1(k_0,M)=0$.
(Lemma 3.4 of \cite{Borovoi-Ramanujan}
is a straightforward generalization
of Lemmas 1.2 and 1.3 of Sansuc \cite{Sansuc},
who assumed that $M$ is a {\em finite} abelian $\Gm$-group.
The easy proof of Lemma 3.4  is omitted
in the published version of \cite{Borovoi-Ramanujan},
but it can be found in the last arXiv version.)
By Milne \cite[Theorem I.4.20(a)]{Milne-ADT},
the group $\Sha^2(k_0,A_q)$ is dual to $\Sha^1(k_0,M)$,
and hence, $\Sha^2(k_0,A_q)=1$.

If $\vk_*(t)=1$, then from the commutative diagram
\[
\xymatrix@R=27pt@C=50pt{
H^2(k_0, Z(G_q))\ar[r]^-{\vk_*}\ar[d]_-\loc  & H^2(k_0, A_q)\ar[d]^\loc  \\
   \prod_v H^2(k_{0,v}\hs, Z(G_q))\ar[r]^-{\Pi_v\hs\vk_{*,v}}
       &\prod_v H^2(k_{0,v}\hs, A_q)
}
\]
we see that  $\vk_{*,v}(t_v)=1$ for all $v$.
Conversely, if $\vk_{*,v}(t_v)=1$ for all $v$, then we see
from the diagram that $\vk_*(t)\in \Sha^2(k_0, A_q)$.
Since $\Sha^2(k_0, A_q)=1$, we conclude that $\vk_*(t)=1$,  as required.
\end{proof}

\begin{example}\label{ex:SU(6)}
Let $k_0=\Q$, $k=\Qbar$, the algebraic closure of $\Q$ in $\C$.
Let $K/\Q$ be an imaginary quadratic extension,
for example, $K=\Q(\sqrt{-1})$.
Write  $\Gal(K/\Q)=\{1,\tau\}$.
Consider a nondegenerate diagonal Hermitian form in 6 variables
\[ \Hm(x)=\sum_{i=1}^6 \lambda_i\hs x_i \hs^\tau\!\hmm x_i\hs,\]
where $\lambda_i\in\Q^\times$.
Let $G_0=\SU(K^6,\Hm)$, where by abuse of notation
we write $\SU(K^6,\Hm)$ for the algebraic $\Q$-group
{\em with the group of $\Q$-points} $\SU(K^6,\Hm)$.
Then $G_0$ is a simply connected absolutely simple $\Q$-group,
a $\Q$-model of $G=\SL_{6,\Qbar}$.
The group $\Gm=\Gal(\Qbar/\Q)$ acts
on the based root datum $\BRD(G)$
via its quotient group $\Gal(K/\Q)=\{1,\tau\}$.

Let $H\subset G$ be a horospherical subgroup
with horospherical invariants $(I,M)$.
Assume that $^\tau\hm I=I$ (for example,
$I=\varnothing$) and $^\tau\hm M=M$.
We show below that {\em if the number $m$
of negative coefficients $\lambda_i$ is even}
(for example, $m=0$), then
$G/H$ admits a $G_0$-equivariant $\Q$-model
if and only if $M^\tau\subset \Qf^\tau$.

Indeed, by Theorem \ref{t:number-HP}
it suffices to check the local conditions $\vk_{*,v}(t_v)=1$
for all places $v$ of $\Q$.
For those $v$ for which $K\otimes_\Q \Q_v\simeq \Q_v\times \Q_v$,
the group $G_{0,v}\coloneqq G_0\times_\Q \Q_v$
is isomorphic to $\SL_{6,\Q_v}$, and hence, $G_{0,v}$ is split, $t_v=1$,
and there is no local condition on $M$.
For those $v$ for which $K\otimes_\Q \Q_v$ is a field,
the group $G_{0,v}$ is of type $^2\hm\Af_5$.
If $t_v=1$, then we have no local condition at $v$ on $M$.
If $t_v\neq 1$, then by Theorem \ref{t:horo-types}
we have the local condition $(*2)$ at $v$,
the same condition $M^\tau\subset \Qf^\tau$ for all such places $v$.
Observe that the set of  places $v$
for which $t_v\neq 1$ is nonempty, because it contains
the place $\infty$ of $\Q$ corresponding to the completion $\Q_\infty=\R$.
Indeed, since $K/\Q$ be an imaginary quadratic extension,
we have $K\otimes_\Q\R\simeq\C$, hence it is a field.
Moreover, for the group $G_{0,\infty}\simeq\SU(6-m,m)$,
we have $H^2(\hs\R,Z(G_{0,\infty})\hs)\cong\{\pm1\}$
and $t_\infty=(-1)^{3-m}$\hs;
see the tables in the appendix to \cite{MJTB}.
Since $m$ is even, we have $t_\infty=-1\neq 1$.
Thus  $G/H$ admits a $G_0$-equivariant $\Q$-model
if and only if $\vk_*(t)=1$
if and only if $M^\tau\subset \Qf^\tau$, as required.

For example, we claim that the lattice
$M=2\Pf +\Qf$ satisfies the conditions
$^\tau\hm M=M$ and $M^\tau\subset \Qf^\tau$.
Indeed, the equality $^\tau\hm M=M$
follows from $^\tau\Pf=\Pf$ and $^\tau\Qf=\Qf$.
Concerning the condition $M^\tau\subset \Qf^\tau$,
we have $\Pf/\Qf\simeq \Z/6\Z$,
and $\tau$ sends $x\in\Pf/\Qf$ to  $-x$.
It follows that $(\Pf/\Qf)^\tau=\{\bar 0,\bar 3\}$.
Furthermore, $M\supset \Qf$, and
\[M/\Qf=2(\Pf/\Qf)=2\Z/6\Z=\{\bar 0,\bar 2,\bar 4\},\]
 whence
\[(M/\Qf)^\tau=2(\Pf/\Qf)\cap (\Pf/\Qf)^\tau=
    \{\bar 0,\bar 2,\bar 4\}\cap\{\bar 0,\bar 3\}=\{\bar 0\}.\]
Thus $M^\tau\subset \Qf$ and $M^\tau\subset\Qf^\tau$, as claimed.
(Actually we have $M^\tau=\Qf^\tau$,
because $M\supset\Qf$ and hence $M^\tau\supset \Qf^\tau$.
Note that although $M^\tau\subset\Qf^\tau$,
$M$ is not contained in $\Qf$.)
If we  consider the subgroup $\mu_2=\{\pm1\}\subset Z(G)$ and set
$H=U\cdot\mu_2=H_{I,M}$ for  $I=\varnothing$ and $M=2\Pf +\Qf$,
then we conclude that the horospherical
homogeneous space $G/(U\cdot\mu_2)$
does admit a $G_0$-equivariant $\Q$-model.

On the other hand, we claim that the lattice $M'=\Pf$
does not satisfy the condition $M^{\prime\,\tau}\subset \Qf^\tau$.
Indeed, consider the fundamental weight
\[\omega_3=\half(\alpha_1+2\alpha_2+3\alpha_3+2\alpha_4+\alpha_5)\in \Pf,\]
where $\alpha_1,\dots,\alpha_5$ are the simple roots;
see Bourbaki \cite[Plate I]{bou02}.
Since the simple roots $\alpha_1,\dots,\alpha_5$
constitute a basis of the lattice $\Qf$,
we see that  $\omega_3\notin \Qf$.
We have $^\tau\hm\hm\alpha_i=\alpha_{6-i}$
and $^\tau\hmm\omega_i=\omega_{6-i}$\hs,
whence $\omega_3\in \Pf^\tau$.
Thus $\Pf^\tau\not\subset \Qf$, and therefore
$M^{\prime\,\tau}=\Pf^\tau\not\subset\Qf^\tau$, as claimed.
Take $I=\varnothing$; then $H_{I,M'}=H_{\varnothing,\Pf}=U$.
We conclude that if $m$ is even,
then the horospherical homogeneous space $G/U$
does not admit a $G_0$-equivariant $\Q$-model.
\end{example}

\section{Models of  $H^n/\Delta$} \label{s:H x H}

\begin{subsec}\label{ss:H-times-H/Delta}
Let $n\ge 2$ be a natural number.
Let $H$ be a linear algebraic $k$-group, not necessarily reductive,
and set $G=H^n\coloneqq H\times\dots\times H$.
Set $Y=H^{n-1}$, where $G$ acts on $Y$ by
\begin{equation}\label{e:action-*}
 (h_1,h_2,\dots,h_n)*(y_1,\dots,y_{n-1})=(h_1y_1h_2^{-1},\dots,
  h_{n-1}y_{n-1}h_n^{-1})\quad\text{for $h_i,y_j\in H$.}
\end{equation}
Then $G$ acts on $Y$ transitively, and the stabilizer of the point
$(1,\dots,1)\in H^{n-1}= Y$ is the diagonal $\Delta=\{(h,\dots,h) \mid h\in H\}\subset H^n=G$.
Thus $Y\cong G/\Delta$.
Note that $Y$ might not be spherical.

Let $H^{\ste{1}}_0$,\dots,$H^{\ste{n}}_0$ be $n$ $k_0$-models of $H$.
We set $G_0=H^{\ste{1}}_0\times _{k_0}\cdots \times _{k_0} H^{\ste{n}}_0$
and ask whether $Y$ admits a $G_0$-equivariant $k_0$-model.
\end{subsec}

\begin{theorem}\label{t:H-times-H}
  With the notation and assumptions of Subsection~\ref{ss:H-times-H/Delta}, the
  homogeneous space $Y=H^n/\Delta$ admits an
  $H^{\ste{1}}_0\times _{k_0}\cdots \times _{k_0}
  H^{\ste{n}}_0$-equivariant $k_0$-model if and only if each
  $H^{\ste{i}}_0$ for $i=2,\dots,n$ is a pure inner form of
  $H^{\ste{1}}_0$.
\end{theorem}

We need a lemma.

\begin{lemma}[well-known]
\label{l:Stephan}
Let $k_0$, $k$, and $\sG$ be as in Subsection \ref{ss:mod-Y}.
Let $G$ be an algebraic group over $k$, and let $F\subset G$ be an algebraic $k$-subgroup.
Set $Y=G/F$. Let $G_0$ be a $k_0$-model of $G$
with semilinear action $\sigma^0\colon \sG\to\SAut(G)$.
If $Y$ admits a $G_0$-equivariant $k_0$-model, then for any
$\gamma\in\sG$ the subgroup $\sigma^0_\gamma(F)$ is conjugate to $F$ in $G$.
\end{lemma}

\begin{proof} See, for instance, Snegirov \cite[Lemma 5.1]{sni18}.\end{proof}

\begin{proof}[Proof of Theorem \ref{t:H-times-H}]
  Set $G_1=(H^{\ste{1}}_0)^n$; then $Y$ admits
  a $G_1$-equivariant  $k_0$-model $Y_1= (H^{\ste{1}}_0)^{n-1}$
  (with the action \eqref{e:action-*} of  $G_1$).
  Assume that $H_0^{\ste{i}}$ is a pure inner form of
  $H_0^{\ste{1}}$ for each $i=2,\dots, n$; then $G_0$
  is a pure inner form of $G_1$, and by Lemma \ref{l:pure-inner} the
  $G$-variety $Y$ admits a $G_0$-equivariant $k_0$-model.

  Conversely, assume that $Y$ admits an
  $H^{\ste{1}}_0\times _{k_0}\cdots \times _{k_0}  H^{\ste{n}}_0$--equivariant $k_0$-model. First we
  show that then $H_0^{\ste{i}}$ is an \emph{inner form} of $H_0^{\ste{1}}$ for all $i=2,\dots,n$.
  Indeed, let
\[\sigma^{\ste{i}}\colon \sG\to\SAut(H) \]
denote the semilinear action corresponding to the $k_0$-model
$H_0^{\ste{i}}$ of $H$ for $i=1,\dots,n$. Recall that
$\Delta(k)=\big\{(h,\dots,h) \mid h\in H(k)\big\}$. Then for any $\gamma\in \sG$ we
have
\[\smash{(\sigma^{\ste{1}}_\gamma\times\dots\times
  \sigma^{\ste{n}}_\gamma)(\Delta(k))=\big\{\hs(\sigma^{\ste{1}}_\gamma(h),\dots,
  \sigma^{\ste{n}}_\gamma(h))\,\mid\, h\in H(k)\hs\big\}\text{.}}\]
Since $Y$ admits an
$H^{\ste{1}}_0\times _{k_0}\dots\times _{k_0} H^{\ste{n}}_0$-equivariant $k_0$-model,
by Lemma~\ref{l:Stephan} the subgroup
\[(\sigma^{\ste{1}}_\gamma\times\dots\times \sigma^{\ste{n}}_\gamma)(\Delta)\]
is conjugate to $\Delta$ in $G=H^n$.
This means that there exists a tuple
$(h_1,\dots,h_n)\in H(k)^n$ such that for any $h\in H(k)$ there exists
$h'\in H(k)$ such that
\[(\sigma^{\ste{1}}_\gamma(h),\dots, \sigma^{\ste{n}}_\gamma(h))=(h_1 h'
h_1^{-1},\dots, h_n h' h_n^{-1}).\]
Then we have
\[h'=h_i^{-1}\hs\sigma^{\ste{i}}_\gamma(h)\hs h_i\text{ for each
}i=1,\dots,n.\]
It follows that
\[\sigma^{\ste{i}}_\gamma(h)=(h_i  h_1^{-1})\cdot
  \sigma^{\ste{1}}_\gamma(h)\cdot(h_i h_1^{-1})^{-1}\text{.}\]
We see that for any $\gamma\in\sG$ we have
\[\sigma^{\ste{i}}_\gamma=\inn(h_i  h_1^{-1})\circ \sigma^{\ste{1}}_\gamma\hs.\]
This means that $\smash{H^{\ste{i}}_0}$ is an inner form of
$\smash{H^{\ste{1}}_0}$ for each $i=2,\dots,n$.

Now we know that $H^{\ste{i}}_0= {}_{c_i} (H^{\ste{1}}_0)$ for some
1-cocycle $c_i\in Z^1(k_0,\Hbar^{\ste{1}}_0)$,
where $\Hbar^{\ste{1}}_0=H^{\ste{1}}_0/Z(H^{\ste{1}}_0)$. Set
$G_1=(H_0^{\ste{1}})^n$; then $Y_1\coloneqq (H_0^{\ste{1}})^{n-1}$ with the
natural action of $G_1$ is a $G_1$-equivariant $k_0$-model of $Y$.
Moreover,
$G_0=H^{\ste{1}}_0\times _{k_0}\dots \times _{k_0} H^{\ste{n}}_0$ is
the inner twisted form of $G_1$ given by the 1-cocycle
$c=(1,c_2,\dots,c_n)\in Z^1(k_0, G_1)$. Then
\[ \delta[c]\ \in\ H^2(k_0, Z(G_1))=H^2(k_0, Z( H_0^{\ste{1}}))^n\] is
$(1,\delta_H[c_2],\dots,\delta_H[c_n])$, where
\[\delta_H\colon H^1(k_0,\Hbar_0^{\ste{1}})\to H^2(k_0, Z(H_0^{\ste{1}}))\]
is the connecting map. By Theorem \ref{t:twist}, there exists a $G_0$-equivariant $k_0$-model
of $Y$ if and only if
$\vk_*(\delta[c])=1$, that is, if and only if
$\vk_*(1,\delta_H[c_2],\dots,\delta_H[c_n])=1$. An easy calculation
shows that
\[\sN_G(\Delta)=Z(G)\cdot\Delta\]
and
\[A \coloneqq \sN_G(\Delta)/\Delta= Z(G)/Z(\Delta)=Z(H)^n \hs /\hs  Z(\Delta)\text{.}\]
Similarly, over $k_0$ we obtain
\[\sN_{G_1}(\Delta_1)=Z(G_1)\cdot\Delta_1\]
and
\[A_1\coloneqq \sN_{G_1}(\Delta_1)/\Delta_1=
  Z(G_1)/Z(\Delta_1)=Z(H^{\ste{1}}_0)^n\hs /\hs  Z(\Delta_1)\text{.}\]
It is easy to see that the homomorphism of abelian $k_0$-groups
\[ Z(H_0^{\ste{1}})^{n-1}\to A_1,\quad (z_2,\dots,z_n)
        \mapsto (1,z_2,\dots,z_n)\cdot Z(\Delta)\ \ \text{for } z_i\in Z(H_0^{\ste{1}})\]
is an isomorphism.
It follows that the induced map on cohomology
\[ H^2(k_0,Z(H_0^{\ste{1}}))^{n-1}\to H^2(k_0, A_1)\] is an
isomorphism of abelian groups. We see that
$\vk_*(1,\delta_H[c_2],\dots,\delta_H[c_n])=1$ if and only if
$\delta_H[c_i]=1$ for all $i=2,\dots,n$. Therefore, an
$H^{\ste{1}}_0\times _{k_0}\dots \times _{k_0}
H^{\ste{n}}_0$-equivariant $k_0$-model of $Y$ exists if and only if
$\delta_H[c_i]=1$ for all $i=2,\dots,n$, that is, if and only if
$H^{\ste{i}}_0$ is a pure inner form of $H^{\ste{1}}_0$ for all
$i=2,\dots,n$, as required.
\end{proof}

Now we consider the case when $k_0$ is a $p$-adic field.

\begin{lemma}\label{l:Kneser}
Let $H_0$ be a simply connected semisimple group over a  $p$-adic field $k_0$\hs.
Then any pure inner form of $H_0$ is isomorphic to $H_0$.
\end{lemma}

\begin{proof}
  Indeed, by Kneser's theorem we have $H^1(k_0,H_0)=1$; see Platonov
  and Rapinchuk \cite[Theorem 6.4]{PR}.
\end{proof}

\begin{corollary}
  In Theorem \ref{t:H-times-H}, if $k_0$ is a $p$-adic field and $H$
  is a simply connected semisimple group over $k$, then $Y$ admits an
  $H^{\ste{1}}_0\times _{k_0}\dots \times_{k_0}
  H^{\ste{n}}_0$-equivariant $k_0$-model if and only if all
  $\smash{H^{\ste{i}}_0}$ are pairwise isomorphic.
\end{corollary}

\begin{proof}
  Indeed, by Theorem \ref{t:H-times-H}, the variety $Y$ admits an
  $H^{\ste{1}}_0\times _{k_0}\dots\times _{k_0}
  H^{\ste{n}}_0$-equivariant $k_0$-model if and only if
  $\smash{H^{\ste{i}}_0}$ is a pure inner form of $H^{\ste{1}}_0$ for
  all $i=2,\dots,n$, and by Lemma \ref{l:Kneser} any pure inner form
  of $H^{\ste{1}}_0$ is isomorphic to $H^{\ste{1}}_0$.
\end{proof}

\begin{remark}
The similar assertion in the case when $k_0=\R$ is false; see Example \ref{e:transporter}.
\end{remark}

\section{Examples}
\label{s:examples}

In each of the examples below we consider a spherical variety
$Y$ and a $k_0$-model $G_0$ of $G$, and we answer the question
 whether $Y$ admits a $G_0$-equivariant $k_0$-model.

\begin{example}\label{e:transporter}
  Let $k_0=\R$, $k=\C$; then $\sG=\{1,\gamma\}$, where $\gamma$ is the
  complex conjugation. Let $H=\SL_{4,\C}$. Consider the diagonal
  matrices
  \[ I_4=\diag(1,1,1,1)\quad\text{and}\quad
    I_{2,2}=\diag(1,1,-1,-1).\] Consider the real models $\SU_{2,2}$
  and $\SU_4$ of $H$:
\begin{align*}
H^{\ste{1}}_0&=\SU_{2,2},\text{ where }\SU_{2,2}(\R)=\{h\in \SL(4,\C)\,
   \mid\, h\cdot I_{2,2}\cdot\upgam h^\tr=I_{2,2}\},\\
H^{\ste{2}}_0&=\SU_4, \text{ where }\SU_4(\R)=\{h\in \SL(4,\C)\,\mid\,
    h\cdot I_4\cdot\upgam h^\tr=I_4\},
\end{align*}
where $h^\tr$ denotes the transpose of $h$.
Consider the 1-cocycle
\[c\colon\sG\to \SU_{2,2}(\R),\quad 1\mapsto I_4,\ \gamma\mapsto I_{2,2}\hs .\]
A calculation shows that $_c \SU_{2,2}\simeq \SU_4$.
Thus $\SU_4$ is a pure inner form of $\SU_{2,2}$.
By Theorem \ref{t:H-times-H}, there exists an $\SU_{2,2}\times _\R \SU_4$-equivariant
real model $Y_0$ of $Y=(H\times  H)/\Delta.$
We describe this model explicitly.
We may  take for $Y_0$ the \emph{transporter}
\[Y_0=\big\{y\in \SL(4,\C)\, \mid\, y\cdot I_{4}\cdot \upgam y^\tr=I_{2,2}\big\}.\]
Clearly $Y_0$ is defined over $\R$.
It is well known that $Y_0(\C)$ is nonempty but $Y_0$ has no $\R$-points.
The group $G_0\coloneqq H^{\ste{1}}_0\times _\R H^{\ste{2}}_0$ acts on $Y_0$ by
\[(h_1,h_2)*y=h_1\hs y h_2^{-1}.\] It is clear that $Y_0$ is a
principal homogeneous space of both $H^{\ste{1}}_0$ and
$H^{\ste{2}}_0$. Thus $Y_0$ is a $G_0$-equivariant $k_0$-model of $Y$.
\end{example}

\begin{example}\label{e:Sp(n)-Sp(n)}
  Let $k_0=\R$, $k=\C$, $n\ge 2$, $\smash{H=\Sp_{2n,\C}}$ (the
  symplectic group in $2n$ variables). Let $G=H\times H$,
  $Y=(H\times H)/\Delta$. Let $0\le m_1\le n$,
  $H_0^{\ste{1}}=\Sp(m_1, n-m_1)$ (the group of the diagonal
  quaternionic hermitian form in $n$ variables with $m_1$ times $+1$
  and $n-m_1$ times $-1$ on the diagonal).
  Here by abuse of notation we write $\Sp(m_1, n-m_1)$
  for the algebraic $\R$-group {\em with the group of $\R$-points} $\Sp(m_1, n-m_1)$.
  Let $H_0^{\ste{2}}=\Sp(m_2, n-m_2)$. Then $H_0^{\ste{2}}$ is a pure
  inner form of $H_0^{\ste{1}}$, and by Theorem~\ref{t:H-times-H} $Y$
  admits an $H^{\ste{1}}_0\times _{k_0} H^{\ste{2}}_0$-equivariant
  $k_0$-model $Y_0$. One can construct $Y_0$ explicitly as a
  transporter similar to the transporter in Example
  \ref{e:transporter}.
\end{example}

\begin{example}\label{e:Sp(2n-R)}
Let $k_0$, $k$, $H$, $G$, $Y$ be as in Example \ref{e:Sp(n)-Sp(n)}.
Let $H_0^{\ste{2}}=\Sp(m_2, n-m_2)$ be as in Example \ref{e:Sp(n)-Sp(n)},
but $H_0^{\ste{1}}=\Sp_{2n,\R}$.
Then
\[\smash{H^1(\R, H_0^{\ste{1}})=H^1(\R,\Sp_{2n,\R})=1}\text{;}\]
see, for instance, Serre \cite[III.1.2, Proposition 3]{ser97}. It follows that
any pure inner form of $H_0^{\ste{1}}$ is isomorphic to
$H_0^{\ste{1}}$. Since $H_0^{\ste{2}}$ is not isomorphic to
$H_0^{\ste{1}}$, we see that $H_0^{\ste{2}}$ is not a pure inner form
of $H_0^{\ste{1}}$. By Theorem \ref{t:H-times-H}, we conclude that
$(H\times _k H)/\Delta$ does not admit a
$H^{\ste{1}}_0\times _{k_0} H^{\ste{2}}_0$-equivariant $k_0$-model.
This example generalizes \cite[Example~9.15]{BG}, where the case when
$H_0^{\ste{2}}$ is compact was considered.
\end{example}

\begin{example}
Let
 \[k_0=\R,\quad k=\C,\quad  H=\SL_{2,\C}\hs,\quad
             G=H\times_\C H\times_\C H,\quad Y=G/\Delta,\]
 where $\Delta$ is $H$ embedded into $G$ diagonally.
Set
\[G_0=H^{\ste{1}}_0\times_\R H^{\ste{2}}_0 \times_\R H^{\ste{3}}_0\hs,\]
where $H^{\ste{i}}_0$ are $\R$-models of $H=\SL_{2,\C}$.
We show that $Y$ has a $G_0$-equivariant model if and only if
$H^{\ste{1}}_0\simeq H^{\ste{2}}_0\simeq H^{\ste{3}}_0$.

First, assume that there exists a $G_0$-equivariant $\R$-model $Y_0$ of $Y$.
Then, by Theorem~\ref{t:H-times-H}, the $\R$-groups  $H^{\ste{2}}_0$ and $H^{\ste{3}}_0$
are pure inner forms of $H^{\ste{1}}_0$.
There are exactly two non-isomorphic $\R$-models of $\SL_{2,\C}$,
namely, $\SL_{2,\R}$ and $\SU_2$,
and the  $\R$-group $\SU_2$ is not a pure inner form of  $\SL_{2,\R}$
because $H^1(\R,\SL_{2,\R})=1$.
We obtain that $H^{\ste{1}}_0\simeq H^{\ste{2}}_0\simeq H^{\ste{3}}_0$,
and either $H^{\ste{1}}_0\simeq\SL_{2,\R}$ or $H^{\ste{1}}_0\simeq\SU_2$.

Conversely, if
$G_0=H^{\ste{1}}_0\times_\R H^{\ste{1}}_0 \times_\R H^{\ste{1}}_0$,
where $H^{\ste{1}}_0$ is either $\SL_{2,\R}$ or $\SU_2$, we may take
$\Delta_0=H^{\ste{1}}_0$ embedded diagonally in $G_0$, and then
$Y_0\coloneqq G_0/\Delta_0$ is a $G_0$-equivariant $\R$-model of $Y$
(having an $\R$-point $1\cdot \Delta_0$).
\end{example}

\begin{example}\label{SO(10)}
   Let $k_0$, $k$, and $\sG$ be as in Subsection \ref{ss:mod-Y}.
   Let $G=\SO_{10,k}$. Let $\X^*(T)=\Zd^5$ be the standard
  presentation of its character lattice with the simple roots
  \[
    \text{$\alpha_1 = \ve_1-\ve_2$,\quad $\alpha_2 = \ve_2-\ve_3$,
    \quad  $\alpha_3 = \ve_3-\ve_4$,\quad
      $\alpha_4 = \ve_4-\ve_5$,\quad $\alpha_5 = \ve_4+\ve_5$,}
  \]
  where $\ve_1,\dots,\ve_5$ is the standard basis of $\Z^5$;
  see Bourbaki  \cite[Plate IV]{bou02}.
  Consider the spherical subgroup $H=\SO_{9,k}\subset G$.
  Its combinatorial invariants are
  \[
    \sX = \Zd\ve_1\text{,} \quad
    \Sigma = \{2\ve_1\}\text{,} \quad
    \Omone = \{(\alpha_1^\vee|_\sX, \{\alpha_1\})\}\text{,}\quad
    \Omt = \emptyset\text{.}
  \]
  We have $Z(G)\cong\{\pm1\}$, \hs$\sN_G(H)=Z(G)\cdot H$, and the homomorphism
  $Z(G)\to \sN_G(H)/H$ is an isomorphism.

  Let $G_0$ be a $k_0$-model of $G$. By Tits \cite[pp.~56--57]{Tits},
  the algebraic $k_0$-group $G_0$ is isomorphic to either
  $\SO(k_0^{10}, \sQ)$, where $\sQ$ is a non-degenerate quadratic form
  in $10$ variables over $k_0$, or to $\SU(D^5, \sH)$, where $D$ is a
  central division algebra of degree~2 over $k_0$, and $\sH$ is a
  non-degenerate anti-hermitian form in 5~variables over $D$ with
  respect to the canonical involution of the first kind of $D$. In
  each case, the Galois group acts on $\X^*(T)=\Zd^5$ either trivially or by
  multiplying $\ve_5$ by $-1$, depending on the discriminant
  of the corresponding quadratic or anti-hermitian form.
  We see that the Galois action always preserves the
  combinatorial invariants of $G/H$.

For a natural number $n\ge 2$, consider the  Kummer exact sequence
\[ 1\to \mu_n\to\G_{m,k_0}\hs\labelto{x\mapsto x^n}\hs\G_{m,k_0}\to 1,\]
where $\mu_n$ denotes the group of roots of unity of order dividing $n$.
  We consider the Brauer group $\Br k_0\coloneqq H^2(k_0,\G_m)$ and
  denote by $(\Br k_0)_n$ the subgroup of elements of $\Br k_0$ of
  order dividing $n$. Then from the Kummer exact sequence we obtain that
  $H^2(k_0,\mu_n)\cong(\Br k_0)_n$.
  We have
  \[H^2(k_0,\hs\sN_G(H)/H)\hs=\hs H^2(k_0,Z(G))=H^2(k_0,\mu_2)=(\Br k_0)_2\hs.\]

  In the case  $G_0=\SO(k_0^{10}, \sQ)$, the image of $t(G_0)$ in $(\Br k_0)_2$ is $0$;
  see ``The Book of Involutions'' \cite[Example (31.11)]{KMRT}. By Theorem~\ref{t:sphfull'}, the
  homogeneous space $G/H$ admits a $G_0$-equivariant $k_0$-model. We
  construct such a model explicitly. We choose $b\in k_0$, $b\neq 0$;
  then the $G_0$-variety given by the equation $\sQ(x)=b$ is a
  $G_0$-equivariant $k_0$-model of $G/H$
  (this $k_0$-model might have no $k_0$-points).

  In the case $G_0=\SU(D^5, \sH)$, the image of $t(G_0)$ in $(\Br k_0)_2$ is $[D]$,
  and hence is nontrivial; see \cite[Example (31.11)]{KMRT}. By
  Theorem~\ref{t:sphfull'}, the spherical homogeneous space $G/H$
  does not admit a $G_0$-equivariant $k_0$-model.
\end{example}

\begin{example}\label{e:SL3}
Let $k_0$, $k$, and $\sG$ be as in Subsection \ref{ss:mod-Y}.
   Let $G=\SL_{3,k}$. Let $\alpha_1$, $\alpha_2$ be
  its simple roots, and $\omega_1$, $\omega_2$ be the corresponding
  fundamental dominant weights. Consider the spherical subgroup
  $H=\SL_{2,k}\subset G$. Its combinatorial invariants are
  \begin{align*}
    &\sX=\X^*(T) = \Zd\omega_1+\Zd\omega_2\text{,} \qquad
    \Sigma = \{\alpha_1+\alpha_2\}\text{,} \\
    &\Omone = \{\, (\alpha_1^\vee|_\sX, \{\alpha_1\}),\,
    (\alpha_2^\vee|_\sX, \{\alpha_2\})\, \}\text{,}\qquad
    \Omt = \emptyset\text{.}
  \end{align*}
  Let $G_0$ be a $k_0$-model of $G$. The Galois group either
  acts trivially on $X=\X^*(T)$  or swaps $\omega_1$ and $\omega_2$ (and hence
  $\alpha_1$ and $\alpha_2$). We see that the Galois action  always preserves the
  combinatorial invariants of $G/H$. According to Tits
  \cite[p.~55]{Tits}, up to isomorphism there are four cases:
  \begin{enumerate}
  \item $G_0=\SL(3,k_0)$; then we take $H_0=\SL(2,k_0)$ and set $Y_0=G_0/H_0$\hs.
  \item $G_0=\SU(l^3,\sH)$ where $\sH$ is a nondegenerate Hermitian form
    in $3$ variables over a quadratic extension $l/k_0$. Then we choose
    $b\in k_0$, $b\neq 0$, and take for $Y_0$
    the subvariety in $l^3$ given by the equation $\sH(y)=b$.
  \item $G_0=\SL(1,D)$, where $D$ is a central division algebra of
    degree $3$ over $k_0$.
    Then the map
    $\smash{\Ztil}=Z(G)\to\sN_G(H)/H$ is the canonical embedding
    $\mu_3\into \G_m$\hs, and it induces an injective homomorphism
    $H^2(k_0,\mu_3)=(\Br k_0)_3\into\Br k_0= H^2(k_0,\G_m)$. The image
    of $t(G_0)$ in $H^2(k_0, \sN_G(H)/H)=H^2(k_0,\G_m)=\Br k_0$ is, of course,
    $[D]\neq 0$, and by Theorem~\ref{t:sphfull'} there is no
    $G_0$-equivariant $k_0$-model of $G/H$.
  \item $G_0=\SU(1, D, \sigma)$, where $(D,\sigma)$ is a central
    division algebra with an involution of second kind over a quadratic
    extension $l/k_0$. Then, after the base change from $k_0$ to $l$, we
    come to the case (3). Since there is no desired model over $l$,
    there is no model over $k_0$ either.
  \end{enumerate}
\end{example}

\begin{proposition}\label{t:G/U}
Let $k_0$, $k$, and $\sG$ be as in Subsection \ref{ss:mod-Y}.
  Let $G$ be a simply connected semisimple group over $k$.
  Let $B\subset G$ be a Borel
  subgroup, and let $U=R_u(B)$ denote the unipotent radical of $B$.
  Let $G_0$ be a $k_0$-model of $G$. Then the
  homogeneous space $G/U$ admits a $G_0$-equivariant $k_0$-model if and
  only if $t(G_0)=1\in H^2(k_0, Z(G_0))$.
\end{proposition}

Proposition \ref{t:G/U} generalizes Moser-Jauslin and Terpereau \cite[Example~3.22]{MJTB},
where the case $k_0=\R$ was considered.

\begin{proof}
Let $G_q$ be a quasi-split inner form of $G_0$.
Let $B_q\subset G_q$ be a Borel subgroup.
Set $U_q=R_u(B_q)$; then $U_q$ is a $k_0$-subgroup of $G_q$, and $G_q/U_q$
is a $G_0$-equivariant  $k_0$-model of $G/U$.
Set
\[A_q=N_{G_q}(U_q)/U_q=B_q/U_q\cong T_q\hs\text{,}\]
where $T_q$ is a maximal torus in $B_q$.
Write $Z_\qs=Z(G_\qs)$ and $\Ztil_\qs=Z(\Gtil_\qs)$
with the notation of Subsection \ref{ss:Tits}.
Since $G$ is semisimple and simply connected, we have $\Ztil_\qs=Z_\qs$\hs,
and the homomorphism $\vktil \colon \Ztil_\qs\to A_\qs$ is
the inclusion monomorphism $\iota\colon Z_\qs\into T_\qs$.
By Proposition \ref{p:tits}, the homogeneous space $G/U$
admits a $G_0$-equivariant $k_0$-model if and only if
\[\iota_*(t(G_0))=1\in H^2(k_0,T_q)\text{.}\]
By Lemma \ref{l:injective} below, the
homomorphism $\iota_*$ is injective, and therefore, we have $\iota_*(t(G_0))=1$ if
and only if $t(G_0)=1$. Thus the homogeneous space $G/U$ admits a
$G_0$-equivariant $k_0$-model if and only if $t(G_0)=1$.
\end{proof}

\begin{lemma}\label{l:injective}
Let $k_0$, $k$, and $\sG$ be as in Subsection \ref{ss:mod-Y}.
Let $G_q$ be a {\emm quasi-split}    reductive group over  $k_0$,
let $B_q\subset G_q$ be a Borel subgroup defined over $k_0$,
and let $T_q\subset B_q$ be a maximal torus.
Then the canonical homomorphism
\[\iota_*\colon H^2(k_0,Z(G_q))\to H^2(k_0,T_q)\]
is injective.
\end{lemma}

\begin{proof}
  Write
  \[G=(G_q)_k\hs, \quad  B=(B_q)_k,\quad    T=(T_q)_k\hs.\]
  Moreover, write
  \[ Z_q=Z(G_q)\hs,\quad \Gbar_q=G_q/Z_q\hs, \quad\Bbar_q=B_q/Z_q\hs,\quad \Tbar_q=T_q/Z_q\hs.\]
  We write
   \[ Z=(Z_q)_k\hs,\quad \Gbar=(\Gbar_q)_k=G/Z,\quad
          \Bbar=(\Bbar_q)_k=B/Z,\quad  \Tbar=(\Gbar_q)_k=T/Z.\]
   The short exact sequence
\[ 1\to Z_q\labelto{\iota} T_q\lra \Tbar_q\to 1\]
induces a cohomology exact sequence
\begin{equation}\label{e:exact-Sansuc}
\dots\to H^1(k_0,\Tbar_q)\lra H^2(k_0,Z_q)\labelto{\iota_*}H^2(k_0,T_q)\to\dots \
\end{equation}
Since the algebraic subgroups  $\Tbar$ and $\Bbar$
of $\Gbar=(\Gbar_q)_k$ are defined over $k_0$,
the Galois group $\sG=\Gal(k/k_0)$ preserves $\Tbar$ and $\Bbar$ when acting on $\Gbar$.
It follows that the natural action of $\sG$ on $\X^*(\Tbar)$ given by
\[(\upgam\hs\chi)(t)=\upgam\hs(\chi(\hs^{\gamma^{-1}}\hmm t)\hs)
              \quad\text{for }\gamma\in\sG,\ \chi\in\X^*(\Tbar),\ t\in\Tbar(k),\]
 preserves the set of simple roots $\Sm=\Sm(G,T,B)=\Sm(\Gbar,\Tbar,\Bbar)\subset\X^*(\Tbar)$.
Since $\Gbar$ is a semisimple group of adjoint type,
the set of simple roots $\Sm$ is a basis of the character group $\X^*(\Tbar)$;
cf.~Springer \cite[Section 2.15]{Springer-RG}.
It follows that the Galois group permutes the elements of the basis $\Sm$ of $\X^*(\Tbar)$,
hence  $\Tbar_q$ is a quasi-trivial torus, and therefore, $H^1(k_0,\Tbar_q)=1$;
see Sansuc \cite[Lemma 1.9]{Sansuc}.
Now from the exact sequence \eqref{e:exact-Sansuc} we see that $\ker\iota_*=1$, as required.
\end{proof}

\begin{example}
  In Proposition \ref{t:G/U}, let $k_0=\R$, $k=\C$, $G=\SL_{2m,\C}$,
  and $G_0=\SU(s, 2m-s)$, where $0\le s\le 2m)$. Then the homogeneous
  space $G/U$ admits a $G_0$-equivariant $\R$-model if and only if
  $s\equiv m\pmod{2}$. Indeed, we have $H^2(\R,Z_0)=\{\pm 1\}$. By
  the tables in the appendix to \cite{MJTB}, we have
  $t(\SU(s, 2m-s)\hs)=(-1)^{m-s}$, and we apply  Proposition \ref{t:G/U}.
\end{example}

\begin{example}
  Let $G=\SL_3(\C)$, let $B$ be the subgroup of upper triangular
  matrices in $G$, and let $T$ be the subgroup of diagonal matrices.
  Let $\alpha_1,\alpha_2$ be the simple roots. The fundamental weights
  $\omega_1$, $\omega_2$ form a basis of the character lattice
  $\X^*(B)$.

 Let $W_3=\C^3$ with basis $e_0,e_1,e_2$ and coordinates $(x_0,x_1,x_2)$ in this basis.
 Let $W_4=W_3\oplus\C$ with coordinates $(x_0,x_1,x_2,x_3)$.
 The group $G=\SL_3(\C)=\SL(W_3)$ naturally acts on $W_3$, and thus it acts on $W_4$.
 We write
 \[\Pd^3_{(x_0:x_1:x_2:x_3)} =\Pd(W_4)=(W_4\smallsetminus\{0\})/\C^\times;\]
 then $G$ naturally acts on $\Pd^3_{(x_0:x_1:x_2:x_3)}$\,.

Let $W_3'$ denote the dual space to $W_3$ with coordinates $(y_0,y_1,y_2)$.
The group $G=\SL(W_3)$ naturally  acts on $W_3'$; namely, a matrix
$A\in \SL(W_3)=\SL_3(\C)$ acts on $W_3'$ as $(A^\tr)^{-1}$.
We write $\Pd^2_{(y_0:y_1:y_2)}=\Pd(W_3')$ and consider the spherical embedding
\[Y^e =\Pd^3_{(x_0:x_1:x_2:x_3)} \times \Pd^2_{(y_0:y_1:y_2)}\,.\]
Then $G$ naturally acts on $Y^e$.

One can check that $\sX = \langle \omega_1, \omega_2\rangle_\Zd = \X^*(B)$.
We write $(a, b) \in V$ for the point of $V$ with coordinates $a, b$
in  the  basis $\{\alpha^\vee_1,\alpha^\vee_2\}$ of $V$
(which is dual to the basis  $\{\omega_1,\omega_2\}$ of  $\sX$).
We have $B$-invariant divisors
\begin{align*}
  D_1 &= \Vd(y_0) \text{ with valuation $v_1 = (1, 0)$, not $G$-invariant,}\\
  D_2 &= \Vd(x_2) \text{ with valuation $v_2 = (0, 1)$, not $G$-invariant,}\\
  D_3 &= \Vd(x_3) \text{ with valuation $v_3 = (1, -1)$, $G$-invariant,}\\
  D_4 &= \Vd(x_0y_0+x_1y_1+x_2y_2) \text{ with valuation $v_4 = (-1, 0)$, $G$-invariant,}
\end{align*}
where for a bihomogeneous function $f$ on $Y^e$, $\Vd(f)$ denotes the
subvariety of zeros of $f$. We have
\[
\Dm=\{D_1,D_2\}\quad\text{with}\quad
\vs(D_1) = \{\alpha_1\},\  \vs(D_2) = \{\alpha_2\},\
\rho(D_1) = v_1\hs,\   \rho(D_2)=v_2\hs.
\]
We see that $\Omt=\emptyset$
and $\Omega=\Omone = \{(v_1, \{\alpha_1\}), (v_2, \{\alpha_2\})\}$.

The open $G$-orbit is $Y^e \smallsetminus (D_3 \cup D_4)$, and we may take
\begin{align*}
  H \coloneqq \Stab_{G}\rleft(\, (1\hm\hm:\hm\hm0\hm\hm:\hm\hm0\hm\hm:\hm\hm1),\hs
  (1\hm\hm:\hm\hm0\hm\hm:\hm\hm0)\,\rright)
  = \rleft\{\begin{pmatrix} 1&0&0\\0&*&*\\0&*&* \end{pmatrix}\rright\}\text{;}
\end{align*}
hence $G/H$ is the homogeneous space from Example~\ref{e:SL3},
  and we obtain $\Sigma = \{\alpha_1+\alpha_2\}$.

It can be shown that the maximal colored cones in $\CF(Y^e)$ are
\begin{align*}
  (\cone(v_3, v_4), \emptyset) \text{ and } (\cone(v_2, v_4), \{D_2\})\text{.}
\end{align*}

If we consider the  $\R$-model $G_0=\SU(2,1)$ of $\SL_3(\C)$,
then the corresponding $\sG$-action on $\sX=\X^*(B)$
swaps $\omega_1$ and $\omega_2$,
and the corresponding $\sG$-action on $V$
swaps $\alpha_1^\vee$ and $\alpha_2^\vee$.
We see that the $\sG$-action preserves $\sX$, $\Vm$, $\Omone$, and $\Omt$.
By Theorem \ref{t:sphqs}, the open $G$-orbit
admits a $G_0$-equivariant $\R$-model $Y_0$,
because $G_0=\SU(2,1)$ is quasi-split.

On the other hand, since $\Omt=\emptyset$, we see that
the canonical surjective map $\Dm\to\Omega$ is bijective,
and so the $\sG$-action on $\Omega$
lifts {\em uniquely} to a  $\sG$-action on $\Dm$.
Thus $\sG$ naturally acts on the set of possible colored cones in $V$.
Clearly, the set of two cones
\[\{\cone(v_3, v_4),\ \cone(v_2, v_4)\}\]
(where we forget about the colors) is not $\sG$-stable,
and hence, the colored  fan  $\CF(Y^e)$ is not $\sG$-stable.
Thus condition (i) of Theorem \ref{t:main2} is not satisfied.
By Corollary \ref{c:main2}, $Y^e$ does not admit a $G_0$-equivariant $\R$-model.
\end{example}

\begin{example}
Let $G=\SL_{6,\C}$  with root system  $\Af_5$.
Let $\alpha_1,\dots,\alpha_5$ denote the simple roots,
and $\omega_1,\dots,\omega_5$ denote the fundamental weights.
Consider the combinatorial invariants
\begin{align*}
\sX &= \langle \alpha_1,\hs \omega_3,\hs \alpha_5 \rangle\text{,}\\
\Vm &= \{v \in V \hs\mid\hs \langle v, \alpha_1 \rangle \le 0
    \text{ and } \langle v, \alpha_5 \rangle \le 0\}\text{,}\\
\Omone &= \{(\alpha_2^\vee, \{\alpha_2\}), (\alpha_3^\vee, \{\alpha_3\}),
    (\alpha_4^\vee, \{\alpha_4\})\}\text{,}\\
\Omt &= \{(\tfrac{1}{2}\alpha_1^\vee, \{\alpha_1\}),
   (\tfrac{1}{2}\alpha_5^\vee, \{\alpha_5\})\text{.}
\end{align*}
Let $H\subset G$ be a spherical subgroup corresponding to these invariants.

We describe the set of colors $\Dm$.
Let $(v_1, v_3, v_5)$ be the  basis of $V$ dual  to the basis
$(\alpha_1, \hs \omega_3,\hs \alpha_5)$ of $\sX$.
Then we have
\begin{align*}
   &\rho(D_1^+) = \rho(D_1^-) =\half\alpha_1^\vee|_\sX= v_1\hs,
   \quad \varsigma(D_1^+) = \varsigma(D_1^-) = \{\alpha_1\}\text{,}\\
   &\rho(D_5^+) = \rho(D_5^-) =\half\alpha_5^\vee|_\sX= v_5\hs,
   \quad \varsigma(D_5^+) = \varsigma(D_5^-) = \{\alpha_5\}\text{,}\\
   &\rho(D_3)=\alpha_3^\vee|_\sX= v_3\hs, \ \quad\qquad\qquad\ \ \;
   \varsigma(D_3) = \{\alpha_3\}\text{,}\\
   &\rho(D_2)=\alpha_2^\vee|_\sX= -v_1\hs, \ \quad\qquad\qquad \varsigma(D_2) = \{\alpha_2\}\text{,}\\
   &\rho(D_4)=\alpha_4^\vee|_\sX= -v_5\hs, \ \quad\qquad\qquad\varsigma(D_4) = \{\alpha_4\}\text{.}
\end{align*}

Let $\Dm^{\ste{1}}$ and $\Dm^{\ste{2}}$ denote the preimages
in $\Dm$ of $\Omone$ and $\Omt$, respectively.
We have $\Dm=\Dm^{\ste{1}}\cup\Dm^{\ste{2}}$.
We describe the sets $\Dm^{\ste{1}}$ and $\Dm^{\ste{2}}$.
Write
\[\Dbar_2=(\rho(D_2),\vs(D_2))\in \Omone,\
  \Dbar_3=(\rho(D_3),\vs(D_3))\in \Omone,\
  \Dbar_4=(\rho(D_4),\vs(D_4))\in \Omone;\]
 then
\[\Omone=\{\Dbar_2,\Dbar_3,\Dbar_4\}\cong\Dm^{\ste{1}}=\{D_2,D_3,D_4\}.\]
Write
\begin{align*}
\Dbar_1&=(\rho(D_1^+),\vs(D_1^+))=(\rho(D_1^-),\vs(D_1^-))\in \Omt,\\
\Dbar_5&=(\rho(D_5^+),\vs(D_5^+))=(\rho(D_5^-),\vs(D_5^-))\in \Omt;
\end{align*}
 then
\[\Omt=\{\Dbar_1,\Dbar_5\}\quad\text{and} \quad \Dm^{\ste{2}}=\{D_1^+,D_1^-,D_5^+,D_5^-\} .\]

We compute $A$ and $A^\kk$.
We use the notation of Construction \ref{con:Sigma} in Appendix \ref{s:A} below.
We have   $\Sigma = \{\alpha_1, \alpha_5\}$ and hence, $\Sigmat=\Sigma$.
It follows that
\begin{align*}
\Sigmac =  \Sigma = \{\alpha_1, \alpha_5\}\quad\text{ and }
\quad  \Sigma^N = \{2\alpha_1, 2\alpha_5\}\text{.}
\end{align*}
By Proposition \ref{p:Losev} and Theorem \ref{p:ker}, we have
$A(k) \cong \Hom(\sX/\langle\Sigma^N\rangle, k^\times)$ and
\[ A^{\kk}(k) \cong \Hom(\sX/\langle\Sigmac\rangle, k^\times)
     =\Hom(\sX/\langle\Sigma\rangle, k^\times)\cong
     \Hom(\langle\omega_3\rangle,k^\times)\cong k^\times\text{.}\]

Let $G/H\into Y^e$ be the spherical embedding corresponding to the colored fan
with unique maximal colored cone $(\Cm,\Fm)$, where
\begin{align*}
  \Cm=\cone(-v_1+v_3-v_5, v_1, v_5),\quad \Fm=\{D_1^+, D_5^-\}\text{.}
\end{align*}
Let $k_0=\R$, $k=\C$,  and let $G_j=\SU_{6-j,j}$ for $j=0,1,2,3$.
Then $G_j$ is an  $\R$-model of $G=\SL_{6,\C}$.
We ask whether $Y^e$ admits a $G_j$-equivariant $\R$-model.

For any $j=0,1,2,3$, the group $G_j$ is an inner form of the quasi-split group $G_3=\SU_{3,3}$.
We see that all groups $G_j$ define the same (nontrivial) Galois action
on the Dynkin diagram $\Dyn(G)=\Af_5$.
Namely, the complex conjugation $\gamma\in\sG=\Gal(\C/\R)$ acts on $\Af_5$
by the nontrivial automorphism of $\Af_5$.
When acting on $V$, $\gamma$ swaps $v_1$ and $v_5$ and fixes~$v_3$.
We see that $\gamma$ preserves the cone $\Cm$.
When acting on $\Omega$, the complex conjugation $\gamma$ swaps $\Dbar_2$ and $\Dbar_4$,
it swaps $\Dbar_1$ and $\Dbar_5$, and it fixes $\Dbar_3$.
Now we ask whether $\gamma$ preserves the colored cone $(\Cm,\Fm)$,
that is, whether $\gamma$ preserves $\Fm=\{D_1^+,D_5^-\}$.
This depends on a lift to $\Dm$ of the $\sG$-action on $\Omega$.

Clearly, the Galois action on $\Omega$ can be lifted to an action on $\Dm$
that does not preserve the subset $\Fm=\{D_1^+,D_5^-\}$;
namely, we can take the action of the complex conjugation $\gamma$ on $\Dm$
that swaps $D_1^+$ and $D_5^+$ and swaps $D_1^-$ and $D_5^-$.
However, we can also lift the action of $\gamma$ on $\Omega$ to the action on $\Dm$
that swaps $D_1^+$ and $D_5^-$ and swaps $D_1^-$ and $D_5^+$.
Then we obtain a $\sG$-action that does preserve the  subset $\Fm=\{D_1^+,D_5^-\}$,
and hence preserves our colored fan.
We see that condition (i) of Theorem~\ref{t:main2} is satisfied for any $j=0,1,2,3$.

We check condition (ii) of Theorem~\ref{t:main2} for $G_j$.
We have $\Gtil=G$, $Z(\Gtil)=Z(G)=\mu_6=\langle\zeta\rangle$,
where $\zeta$ is a primitive 6th root of unity.
Note that  $\sG$ acts trivially on $Z(\Gtil)$.
By Corollary~\ref{c:T-trivial} below, we have
\[H^2(\sG, Z(\Gtil))=\langle\zeta\rangle/\langle\zeta^2\rangle\cong\{\pm1\}.\]
We have $\X^*(A^\kk)\cong\langle \omega_3\rangle$.
Since the complex conjugation $\gamma$ fixes $\omega_3$,
we see that $A_q^\kk\simeq \G_{m,\R}$.
By Corollary \ref{c:T-Brauer} below, we have
\[H^2(\R,A_q^\kk)\cong H^2(\R,\G_{m,\R})=\R^\times/\R^\times_+\cong\{\pm 1\}.\]
Since $\omega_3$ is the highest weight of the representation
of $G=\SL_{6,\C}$ in $\Lambda^3(\C^6)$,
we see that the restriction of the homomorphism
\[ \omega_3\colon T\to \G_m\]
maps the generator $\zeta$ of $Z(G)=\mu_6$ to $\zeta^3=-1\in\R^\times$,
and hence, it maps the nontrivial class
$[\zeta]\in H^2(\sG, Z(\Gtil))$ to the nontrivial class $[-1]\in H^2(\R,A_q^\kk)$.
It follows that the map
\[\vktil_*\colon H^2(\sG, Z(\Gtil_q))\to H^2(\R,A_q^\kk)\cong\{\pm1\}\]
is an isomorphism.
Thus $\vktil_*(t(\Gtil_j))=1$ if and only if $t(\Gtil_j)=1$.
We see from Table 2 in the appendix to \cite{MJTB} that $t(\Gtil_j)=(-1)^{3-j}$.
Thus condition (ii) of Theorem~\ref{t:main2} for $G_j$
is satisfied if and only if  $j=3$ or $j=1$.
Since the colored fan of $Y^e$ has only one maximal colored cone, the variety $Y^e$
contains exactly one closed $G$-orbit, hence is quasi-projective
by Sumihiro \cite[Lemma~8]{sum74}.
By Theorem \ref{t:main2} and Corollary \ref{c:main2}, the spherical embedding  $Y^e$
admits a $G_j$-equivariant $\R$-model if and only if $j=3$ or $j=1$, that is,
if and only if our group $G_j$ is $\SU_{3,3}$ or $\SU_{5,1}$.

A similar (but easier) calculation shows that $Y^e$ admits an $\SL_{6,\R}$-equivariant $\R$-model,
but has no $\SL_{3,\Hd}$-equivariant $\R$-model,
where $\Hd$ denotes the division algebra of Hamilton's quaternions.
\end{example}

\begin{lemma} \label{l:Tate}
Let $\sG=\{1,\gamma\}$ be a group of order two, and let $A$
be an abstract abelian group written multiplicatively,
with an action of $\sG$. Write $A^\sG$ for the subgroup of $\gamma$-fixed points in $A$.
Then there is  a canonical and functorial in $A$ isomorphism
\[ H^2(\sG,A)\isoto A^\sG/\{a\cdot\upgam a \hs\mid\hs a\in A\}.\]
\end{lemma}

\begin{proof}
  See Atiyah and Wall \cite[Section~8, formulas before Theorem~5]{aw67}.
\end{proof}

\begin{corollary}[well-known]
\label{c:T-trivial}
If in Lemma \ref{l:Tate} the group $\sG$ acts on $A$ trivially, then $H^2(\sG,A)=A/A^2$,
where $A^2=\{a^2 \mid a\in A\}$.
\end{corollary}

\begin{corollary}[extremely well-known]
\label{c:T-Brauer}
If in Lemma \ref{l:Tate} \hs $\sG=\Gal(\C/\R)$, then
\[H^2(\R,\G_{m,\R})=\R^\times/\{z\cdot\upgam z\hs \mid\hs z\in \C^\times\}=
\R^\times/\R^\times_+\cong\{\pm1\}.\]
\end{corollary}

\appendix

\section{Rational points}
\label{s:G/U}

In this appendix we prove the following assertion:

\begin{proposition}\label{p:unip-dense}
Let $k_0$, $k$, and $\sG$ be as in Subsection \ref{ss:mod-Y}.
Let
\[f_0\colon X_0\to Y_0\]
be a  dominant  morphism of $k_0$-varieties.
Assume that a connected unipotent group $U_0$ defined over~$k_0$ acts on $X_0$ such that
for any $y\in f_0(X_0(k))\subset Y_0(k)$,
the fiber $f_0^{-1}(y)$ is an orbit of $U_0(k)$.
We write $X=X_0\times_{k_0} k$ and $Y=Y_0\times_{k_0} k$. If the set
of $k_0$-points $Y_0(k_0)$ is Zariski-dense in~$Y$, then the set of
$k_0$-points $X_0(k_0)$ is Zariski-dense in $X$.
\end{proposition}

To prove the proposition, we need a definition and a lemma.

\begin{definition}\label{d:torsor}
  Let $k_0$, $k$, and $\sG$ be as in Subsection \ref{ss:mod-Y}.
  Let $G_0$ be a linear algebraic group over $k_0$. A \emph{homogeneous
    space} of $G_0$ is a $G_0$-$k_0$-variety $Y_0$ such that $G_0(k)$
  acts transitively on $Y_0(k)$. A \emph{principal homogeneous space
    (torsor)} of $G_0$ is a homogeneous space $P_0$ of $G_0$ such that
  $G_0(k)$ acts simply transitively on $P_0(k)$, that is, the
  stabilizer of a point $p\in P_0(k)$ is trivial. By a torsor
  \emph{dominating} a homogeneous space $Y_0$ we mean a pair
  $(P_0,\alpha_0)$, where $P_0$ is a torsor of $G_0$ and
  $\alpha_0\colon P_0\to Y_0$ is a $G_0$-equivariant morphism.
\end{definition}

\begin{subsec}\label{ss:nonab-H2}
  For a given homogeneous space $Y_0$ of $G_0$, we ask whether there
  exists a torsor dominating it as in Definition~\ref{d:torsor}. Let
  $y\in Y_0(k)$ and set $H=\Stab_G(y)$, where $G=G_0\times_{k_0} k$.
  The homogeneous space $Y_0$ defines a $k_0$-kernel $\kappa$ and a
  cohomology class
\[\eta(Y_0)\in H^2(k_0,H,\kappa);\]
see Springer \cite[Section 1.20]{spr66} or Borovoi \cite[Sections~7.1 and
7.7]{Borovoi-Duke}. There exists a pair $(P_0,\alpha_0)$ dominating $Y_0$ as in
Definition \ref{d:torsor} if and only if the class $\eta(Y_0)$ is
neutral (there can be more than one neutral class).
\end{subsec}

\begin{lemma}\label{l:unip-hom}
Let $k_0$, $k$, and $\sG$ be as in Subsection \ref{ss:mod-Y}.
Let $U_0$ be a connected unipotent group over $k_0$, and let $Y_0$ be a homogeneous space of $U_0$.
Then $Y_0$ has a $k_0$-point, and $k_0$-points are dense in $(Y_0)_k$.
\end{lemma}

\begin{proof}
Write $U=U_0\times_{k_0} k$.
Let $y\in Y_0(k)$ be a $k$-point, and  set $U'=\Stab_U(y)$;
then $U'$ is a unipotent $k$-group.
Consider
\[\eta(Y_0)\in H^2(k_0,U',\kappa)\]
as in \ref{ss:nonab-H2}.
Since $U'$ is unipotent, by Douai's theorem \cite[Theorem IV.1.3]{Douai},
see also \cite[Corollary 4.2]{Borovoi-Duke},
all elements of $H^2(k_0,U',\kappa)$ are neutral.
Since $\eta(Y_0)$ is neutral, there exists a pair $(P_0,\alpha_0)$
as in Definition \ref{d:torsor}.

The torsor $P_0$ of $U_0$ defines a cohomology class
\[ \xi(P_0)\in H^1(k_0, U_0).\]
Since $U_0$ is unipotent and ${\rm char}(k_0)=0$,
we have $H^1(k_0,U_0)=\{1\}$; see Serre  \cite[III.2.1, Proposition 6]{ser97}.
Thus the class $\xi(P_0)$ is neutral, and hence $P_0$ has a $k_0$-point $p_0$.
Then the point $y_0=\alpha_0(p_0)$ is a $k_0$-point of $Y_0$.

Since our field $k_0$ is of characteristic 0, it is perfect and infinite,
and by Borel \cite[Corollary 18.3]{bor91}
the set $U_0(k_0)$ is dense in $U$.
It follows that
\[ \{u\cdot y_0\hs \mid\hs u\in U_0(k_0)\}\subset Y_0(k_0)\]
is a dense set of $k_0$-points in $(Y_0)_k$, as required.
\end{proof}

\begin{remark}
For any unipotent group $U_0$ (which we do not assume to be connected)
over our field $k_0$ of characteristic 0,
there exists the exponential map
\[ {\rm \exp}\colon \Lie(U_0)\to U_0\hs,\]
which is an isomorphism of $k_0$-varieties. It follows
that $U_0$ is isomorphic to the affine space $\Lie(U_0)$ as a $k_0$-variety.
In particular, $U_0$ is connected and the $k_0$-points are dense in $(U_0)_k$.
\end{remark}

\begin{subsec}{\em Proof of Proposition \ref{p:unip-dense}.}
Let $\sU_X\subset X$ be a non-empty  open subset.
We shall show that $\sU_X$ contains a $k_0$-point.
Since $f$ is dominant, $f(\sU_X)$ contains some open subset $\sU_Y$ in $Y$.
Since by assumption $Y_0(k_0)$ is dense in $Y$, there exists a $k_0$-point $y_0\in\sU_Y$.
Then $y_0\in f(X(k))$.

Set $F_0=f^{-1}(y_0)\subset X$; then the fiber $F_0$ is defined over $k_0$.
Write $F= F_0\times_{k_0} k$.
We know that the variety $F$ is non-empty because $y_0\in f(X(k))$.
Since by assumption $F_0$ is a homogeneous space of the unipotent group $U_0$,
by Lemma \ref{l:unip-hom} the set $F_0(k_0)$ is dense in $F$.

The set $\sU_F\coloneqq\sU_X\cap F$ is open in $F$,
and it is non-empty because $y_0\in\sU_Y\subset f(\sU_X)$.
Since $k_0$-points are dense in $F$,   there exists a $k_0$-point $x_0\in\sU_F\subset \sU_X$.
We conclude that the set  $X_0(k_0)$ is dense in $X$.
\qed \end{subsec}

 \section{The automorphism group of a spherical homogeneous space}
 \label{s:AqAq}

\begin{subsec}
  \label{ss:aut}
  In the setting of Subsection~\ref{ss:sph-qs},
  we recall the combinatorial description of the
  automorphism group \[\sA \coloneqq \Aut^G(G/H)\] due to Knop and
  Losev. For every $\alphA \in \sA$ and $\lambda \in \sX$, the
  automorphism $\alphA$ preserves the one-dimensional subspace
  $k(G/H)^{(B)}_\lambda$ of $B$-$\lambda$-semi-invariants in $k(G/H)$,
  that is, it acts on this space by multiplication by a scalar
  $d_{\alphA,\lambda}\in k^\times$. It is easy to see that in this way  we obtain a
  homomorphism
  \begin{align*}
    \iota\colon \sA \to \Hom(\sX, k^\times)\text{,}\quad
    \alphA \mapsto (\lambda \mapsto d_{\alphA,\lambda})\text{.}
  \end{align*}
  Let $T^\Xm$ denote the $k$-torus with character group $\Xm$; then
  the group $\Hom(\sX, k^\times)$ is naturally identified with the
  group $T^\Xm(k)$ of $k$-points of  $T^\sX$.
  Knop \cite[Theorem~5.5]{kno96}  showed that the homomorphism
  $\iota$ is injective and its image is closed in $T^\sX(k)$.
  It follows that this image corresponds to a sublattice
  $\Xi \subset \sX$ such that
  \begin{align*}
    \iota(\sA)=\big\{\phi \in \Hom(\sX, k^\times) \, \mid\, \phi(\lambda)=1
    \text{ for all }\lambda \in \Xi\subset\sX\big\}\text{.}
  \end{align*}
   We shall identify $\sA$ with its image
  $\iota(\sA)\subset T^\sX(k)$.
  Then $\sA$ is the group of $k$-points of a group of multiplicative type $A$
  over $k$ with character group $\X^*(A)=\sX/\Xi$.

  According to Losev \cite[Theorem~2]{los09a},
  there exist integers
  $(c_\gamma)_{\gamma\in\Sigma}$ equal to $1$ or $2$ such that the set
  \begin{align*}
    \Sigma^N =\{c_\gamma\cdot\gamma\}_{\gamma\in\Sigma}\subset \Xi
  \end{align*}
  is a basis of the lattice $\Xi$; see Appendix \ref{s:A} below. It
  follows that we have
  \begin{align*}
    \sA = \big\{\hs\psi \in \Hom(\sX, k^\times) \,\mid\, \psi(\Sigma^N) =
    \{1\}\hs\big\}\text{.}
  \end{align*}
  The coefficients $c_\gamma$ can be computed from the combinatorial invariants of~$G/H$;
  see Construction \ref{con:Sigma} and Proposition \ref{p:Losev} in Appendix \ref{s:A}.
  It follows that any automorphism of  the based root datum $\BRD(G)$
   preserving  $(\sX, \Vm, \Dm)$ also preserves~$\Sigma^N$.
  See Subsection \ref{ss:BRD} for the definition of $\BRD(G)$.
\end{subsec}
Consider  the canonical homomorphism
\[\pi\colon\hs\Gtil\onto [G,G]\into G,\]
where $\Gtil$ is the universal cover of the commutator subgroup $[G,G]$ of $G$.
By abuse of notation, we also denote by $\pi$ the induced homomorphism
of the centers  $Z(\Gtil)\to Z(G)$.
Every element $\ztil \in Z(\Gtil)(k)$ defines a $G$-equivariant
automorphism $\alphA_\ztil\colon G/H \to G/H$, $y \mapsto \pi(\ztil) \cdot y$.

\begin{prop}
  \label{p:111}
  $\iota(\alphA_\ztil)(\lambda)=\lambda(\pi(\ztil))$ \hs for all
  \hs$\ztil\in \Ztil(k),\, \lambda \in \sX\subset \X^*(T)$.
\end{prop}

\begin{proof}
    Let $f \in k(G/H)^{(B)}_\lambda$\hs.
    Write $z=\pi(\ztil)\in Z(G)(k)$.
    Since $z\in B(k)$, for every $y \in G/H$ we have
    \begin{align*}
          (\alphA_\ztil(f))(y) = f(z^{-1}\cdot y) = \lambda(z)f(x)\text{,}
    \end{align*}
    hence $\iota(\alphA_\ztil)(\lambda) = \lambda(z)=\lambda(\pi(\ztil))$, as required.
\end{proof}

\begin{corollary}\label{c:pi-Aq}
The homomorphism of abstract groups
\[Z(\Gtil)(k)\to Z(G)(k)\to\Am\into\Hom(\Xm,k^\times)\]
comes from a homomorphism of algebraic groups $Z(\Gtil)\to T^\Xm$,
where $T^\Xm$ denotes the algebraic $k$-torus with character group $\Xm$.
\end{corollary}

\begin{subsec}
From now on till the end of this appendix, we assume that the $\sG$-action
on $\X^*(B)$ and $\Sm$ preserves the combinatorial invariants $(\Xm,\Vm,\Dm)$ of $G/H$.

  Since $\sG$ acts continuously on $\sX$ and $k^\times$, we obtain a natural algebraic
  $\sG$-action on $T^\sX(k)=\Hom(\sX, k^\times)$ given by
  \[({}^\gamma\phi)(\lambda) =
    {}^\gamma(\phi(\hs^{\gamma^{-1}}\!\lambda))\quad \text{for every \ }
     \gamma \in \gal,\ \phi \in \Hom(\sX, k^\times),\ \lambda \in \sX.\]
   This action defines a $k_0$-model of the torus $T^\sX$.
   According to Subsection~\ref{ss:aut}, since the
  $\sG$-action preserves the combinatorial invariants
  $(\sX, \Vm, \Dm)$, it also preserves the set $\Sigma^N$.
  It follows that the subgroup
  \[A(k)=\iota(\sA)\subset \Hom(\sX, k^\times)=T^\sX(k)\]
  is $\sG$-invariant. We see that the subgroup $A$ of $T^\sX$
  is defined over $k_0$, and hence the
  group $A(k)$ with the induced $\sG$-action is the group of
  $k$-points of an algebraic $k_0$-group, which we denote by $A_q$.

  \begin{prop}
    \label{prop:qsaut}
    Let $G_\qs$ be a quasi-split inner form of $G_0$, and let $Y_\qs$  be a
    $G_\qs$-equivariant $k_0$-model of $G/H$ (it exists by Theorem \ref{t:sphqs}).
    The model $Y_\qs$ induces a
    $\sG$-action on $\sA$.
    Let $\Am_\qs$ denote $\Am$ with the $\sG$-action induced by $Y_\qs$.
    Then the embedding
    $\iota\colon\Am_\qs\into\Hom(\sX, k^\times)$ is $\sG$-equivariant.
  \end{prop}
  \begin{proof}
  We may and shall assume that $G_0=G_\qs$.
  Let $B_\qs\subset G_\qs$ be a Borel subgroup defined over $k_0$.
  We set $B=(B_\qs)_k\subset G$; then $\upgam  B=B$ for all $\gamma\in\sG$.

    Let $f \in k(G/H)$ be $B$-semi-invariant rational function of weight $\lambda \in \sX$.
    The calculation
    \begin{align*}
      ({}^\gamma f)(b^{-1}x) &= {}^\gamma (f(\hs^{\gamma^{-1}}\hmm(b^{-1}x)))
      = {}^\gamma (f((\hs^{\gamma^{-1}}\hmm b^{-1})\cdot(\hs^{\gamma^{-1}}\hmm\hm x)))\\
      &= \hs^\gamma \lambda({}^{\gamma^{-1}}\hmm b)\cdot \hs^\gamma (f(\hs^{\gamma^{-1}}\hmm\hm x))
      = (\hs^\gamma\hm\lambda)(b) \cdot (\hs^\gamma\hmm f)(x)
    \end{align*}
    shows that for every $\gamma \in \sG$, the rational function ${}^\gamma\hmm f$
    is $B$-semi-invariant of weight ${}^\gamma\hm \lambda$.

  Let $\alphA \in \sA$ and let $d_{\alphA,\lambda} \in k^\times$ be as
    in Subsection~\ref{ss:aut}. The $\sG$-action on $\sA$ induced by the
    $k_0$-model $Y_\qs=G_\qs/H_\qs$ of $G/H$ satisfies
    $({}^\gamma\alphA)(f) = {}^\gamma(\alphA(\hs^{\gamma^{-1}}\hmm\hm f))$. It
    follows that if $f\in k(Y)^{(B)}_\lambda$, then
    \[ (\hs^\gamma\alphA)(f)={}^\gamma(\alphA(\hs^{\gamma^{-1}}\hmm\hm f))
    =\hs^\gamma (  d_{\alphA,{}^{\gamma^{-1}}\hmm\lambda}\cdot\hs^{\gamma^{-1}}\hmm\hm f)
    =\hs^\gamma   d_{\alphA,{}^{\gamma^{-1}}\hmm\lambda}\cdot f, \]
    and thus
    \[{}^\gamma\alphA\hs|_{ k(Y)^{(B)}_\lambda} =
    {}^\gamma  d_{\alphA,{}^{\gamma^{-1}}\hmm\lambda}\hs \cdot \hs \mathrm{id}\text{.}\]
    We see that  if we consider  $\iota(\alphA)\in\Hom(\sX, k^\times)$, then
    \begin{align*}
      \iota(\hs^\gamma\hm\alphA)(\lambda) = {}^\gamma d_{\alphA,{}^{\gamma^{-1}}\hmm\lambda}
      = \hs^\gamma (\hs\iota(\alphA)(\hs^{\gamma^{-1}}\hmm\lambda)\hs)
      =(\hs^\gamma(\iota(\alphA)\hs)(\lambda)\text{,}
    \end{align*}
    and thus \[\iota(\hs^\gamma\alphA)=\hs^\gamma(\iota(\alphA))\text{,}\]
    which shows that the embedding $\iota$ is $\sG$-equivariant, as required.
  \end{proof}

By Corollary \ref{c:pi-Aq}, the  homomorphism of abstract groups
\[Z(\Gtil)(k)\to Z(G)(k)\to\Am =A_q(k)\]
comes from a homomorphism of algebraic $k$-groups $Z(\Gtil)\to A$.
By Proposition \ref{prop:qsaut},
the above homomorphism of abstract groups is $\sG$-equivariant,
and hence, it comes from a $k_0$-homomorphism
\[\vktil\colon\, Z(\Gtil_0)=Z(\Gtil_\qs)\to Z(G_\qs)\to A_\qs\hs.\]
We may and shall identify $\Am_\qs$ with $A_q(k)$.
\end{subsec}

Propositions \ref{p:111} and  \ref{prop:qsaut}
show that one can compute the
 $\sG$-group~$\Am_\qs=A_q(k)$ and the homomorphism of $k_0$-groups
 \[\vktil\colon Z(\Gtil_0) \to A_q\]
 from the combinatorial invariants $(\sX, \Vm, \Omone,\Omt)$ of $G/H$
 and the $\sG$-action on $\X^*(B)$ and $\Sm$.

\section{The character group of $\overline{H}/H$}
\label{s:A}
In this appendix we give an alternative proof of Proposition~\ref{l:inj},
not using Theorem~\ref{t:sphqs}.
This proof is based on a calculation of the character group $\X^*(A^\kk)$;
see Theorem \ref{p:ker}.
Here $A^\kk=\overline{H}/H$, where $\overline{H}$
is the spherical closure of $H$.

\begin{subsec}
We consider a certain subset $\Sigmat$ of the set $\Sigma=\Sigma(G/H)$
of the spherical roots of $G/H$.
For a simple root $\alpha\in\Sm$, let $\Dm(\alpha)$
denote the set of colors $D\in\Dm$ such that
the parabolic subgroup $P_\alpha$ moves $D$, that is, $\alpha \in\vs(D)$.
Then it is known that $|\Dm(\alpha)|\le2$, and, moreover, $|\Dm(\alpha)|=2$
if and only if $\alpha\in \Sigma\cap \Sm$.
For $\alpha\in \Sigma\cap \Sm$, write $\Dm(\alpha)=\{D_\alpha^+,D_\alpha^-\}$.
Let $\Sigmat\subset\Sigma\cap \Sm\subset\Sigma$
denote the set of those $\alpha\in \Sigma\cap \Sm$ for which $\rho(D_\alpha^+)=\rho(D_\alpha^-)$.
Then for $\alpha\in \Sigmat$ we have
 \[\rho(D_\alpha^+)=\rho(D_\alpha^-)=\half\alpha^\vee |_\Xm\hs,\quad
 \vs(D_\alpha^+)=\vs(D_\alpha^-)=\{\alpha\}.\]
The map
\[\Sigmat\to\Omt,\quad\  \alpha\,\longmapsto\,\left (\rho(D_\alpha^\pm) ,\vs(D_\alpha^\pm)\right)=
(\half\alpha^\vee |_\Xm,\{\alpha\})\]
is a bijection.
See \cite[Proposition B.3]{BG}
for details and references.
\end{subsec}

\begin{construction}
\label{con:Sigma}
Let  $\Sigmattf$ denote the set of spherical roots $\gamma\in\Sigma$ satisfying
  one of the following conditions (a), (b), (c)
  (these are conditions (2), (3), (4) of Pezzini and Van Steirteghem
  \cite[Proposition 2.7]{PVS}):
  \begin{enumerate}
  \item[\rm(a)] $\gamma = \alpha_1 + \dots + \alpha_n$,
    where $\{\alpha_1, \dots, \alpha_n\} \subset \Sm$
    has type $B_n$ and $\Dm(\alpha_i) = \emptyset$ for every $2 \le i \le n$,
  \item[\rm (b)] $\gamma = 2\alpha_1 + \alpha_2$, where
    $\{\alpha_1, \alpha_2\} \subset \Sm$ has type $G_2$,
    \item[\rm (c)] $\gamma$ is not in the root lattice of $G$.
  \end{enumerate}

Let $\Sigmac\subset \Xm$ denote the set obtained from $\Sigma$
by replacing $\gamma$ by $2\gamma$ for all $\gamma\in\Sigmattf$.

Let $\Sigma^N\subset \Xm$   denote the set obtained from $\Sigma$
by replacing $\gamma$ by $2\gamma$ for all $\gamma\in\Sigmattf\cup\Sigmat$.

Note that $\Sigmat\cap\Sigmattf=\varnothing$.
Note also that the subsets $\Sigmat$ and $\Sigmattf$ of $\Sigma$ are $\Gm$-stable.
It follows that the subsets $\Sigma^N$ and $\Sigmac$ of $\Xm$ are $\Gm$-stable.
We denote by $\langle\Sigma^N\rangle$ and $\langle\Sigmac\rangle$
the ($\Gm$-stable) subgroups of $\Xm$
generated by $\Sigma^N$ and $\Sigmac$, respectively.
Then $\langle\Sigma^N\rangle\subset \langle\Sigmac\rangle$,
and $\langle\Sigmac\rangle/\langle\Sigma^N\rangle$
is a vector space over $\Z/2\Z$ with basis $\Sigmat$.
\end{construction}

\begin{subsec}
Let $\Am=\Aut^G(G/H)$ regarded as an abstract group.
We may write $\Am=A(k)$ for some algebraic subgroup $A\subset T^\Xm$;
see Subsection  \ref{ss:aut} for details.

Let $\Xm'\subset \Xm$ be a subgroup.
We write
\[(\Xm')^\perp=\big\{t\in T^\Xm(k)\, \mid\, \lambda(t)=1
    \text{ for all }\lambda\in\Xm'\big\}.\]
 We regard $(\Xm')^\perp$ as an algebraic subgroup of $T^\Xm$.
\end{subsec}

\begin{proposition}[Losev \cite{los09a}, Theorem 2 and Definition 4.1.1(1)]
\label{p:Losev}
We have $A=\langle\Sigma^N\rangle^\perp$, and hence $\X^*(A)=\Xm/ \langle\Sigma^N\rangle$.\qed
\end{proposition}

By Proposition \ref{p:Losev}, we have
\[A(k)=\Hom(\Xm/\langle\Sigma^N\rangle,k^\times).\]
Let $\gamma\in \Sigmat\subset\Xm$.
Then  $\gamma\notin \langle\Sigma^N\rangle$, but
$2\gamma\in\Sigma^N\subset  \langle\Sigma^N\rangle$,
so we obtain an element  $[\gamma]\in \Xm/\langle\Sigma^N\rangle$ of order 2.
It defines a homomorphism
\[ A(k)=\Hom(\Xm/\langle\Sigma^N\rangle,k^\times)\ \lra\
\Hom(\langle\gamma\rangle/\langle 2\gamma\rangle, k^\times)=\{\pm1\},
   \quad a\mapsto a(\gamma)\coloneqq d_{a,\gamma}.\]

\begin{proposition}\label{p:swaps}
For $\gamma\in\Sigmat$,
an automorphism $a\in A(k)=\Aut^G(G/H)$ swaps $D_\gamma^+$ and $D_\gamma^-$
if and only if $a(\gamma)=-1$.
\end{proposition}

\begin{proof}
Let $\gamma\in \Sigmat$.
Consider the subsets
\[ A(k)_\gamma^+=\{a\in A(k)\, \mid\, a(\gamma)=+1\}\quad\text{and}\quad
       A(k)_\gamma^-=\{a\in A(k)\, \mid\, a(\gamma)=-1\}.\]
Clearly, elements of finite order are dense in $A(k)$,
and hence, elements of finite order are dense
in the open subsets $A(k)_\gamma^+$ and in $ A(k)_\gamma^-$.
Therefore, it suffices to prove the proposition
in the case when $a$ is an element of finite order.

We follow the proof of Theorem B.5 in Gagliardi's appendix to \cite{BG}.
We have
  \begin{align}\label{a:phi}
    a(f_\gamma) =a(\gamma)\cdot f_\gamma\hs,
    \end{align}
 where $f_\gamma\in k(G/H)^{(B)}_\gamma$, $f_\gamma\neq 0$.
 We assume that $a$ is an element of finite order.
  Let $\stH \subset \NGH$ denote the subgroup containing $H$ such that
  \[
  \stH / H =\langle a\rangle\subset\NGH/H=\Aut^G(Y),
  \]
  where $Y = G/H$ and $\langle a\rangle$ denotes the finite subgroup
  generated by the element $a\in A(k)=\NGH/H$. We set $\stY = G/\stH$. We use the
  same notation for the combinatorial objects associated to the
  spherical homogeneous space $\stY$ as for $Y$, but with a tilde
  above the respective symbol. The morphism of $G$-varieties $Y\to
  \stY$ induces an embedding $k(\stY)\into k(Y)$, and $k(\stY)$ is
  the fixed subfield of $a$. Since $k(\stY)$ is a
  $G$-invariant subfield of $k(Y)$, we have $\stsX \subset \sX$.

  If $a(\gamma)=-1$, then by \eqref{a:phi} we have
  $a(f_\gamma) =a(\gamma)\cdot f_\gamma=-f_\gamma\neq f_\gamma$\hs.
  We see that $f_\gamma \notin k(\widetilde Y)$, hence
  $\gamma \in \sX \smallsetminus \stsX$; in particular
  $\gamma \notin \stSigma$. However, $\gamma\in\Sm$ and we may
  consider $\stsD(\gamma)\subset \stsD$. Note that the image of
  $\Dm(\gamma)$ under the map $Y\to\stY$ is $\stsD(\gamma)$. Since
  $\gamma\notin\stSigma\cap\Sm$, by Luna \cite[Sections 2.6 and
  2.7]{lun01}, see also Timashev \cite[Section 30.10]{tim11}, we have
  $|\stsD(\gamma)|\le 1$, hence the two colors in $\sD(\gamma)$ are
  mapped to one color by the map $Y \to \stY$, that is, $a$~swaps
  $D_\gamma^+$ and $D_\gamma^-$.

  On the other hand,  if $a(\gamma)=+1$, then by \eqref{a:phi} we have
  $a(f_\gamma) =a(\gamma)\cdot f_\gamma=f_\gamma$ and hence,
  $f_\gamma\in k(\stY)$ and $\gamma\in\stsX$.
  Since $\gamma$ is a primitive element of $\sX$, it is a primitive
  element of $\stsX\subset\sX$. The natural map $V \to
  \widetilde V$ induced by $Y \mapsto \stY$ is bijective and
  identifies $\sV$ and $\stsV$ (see Knop
  \cite[Section~4]{kno91}). Since $\gamma \in \Sigma$ is dual to a
  wall of $-\sV$, it is dual to a wall of $-\stsV=-\sV$.
   It follows that $\gamma \in \Sm \cap \stSigma$; hence $|\stsD(\gamma)|
  = 2$, and the two colors in $\sD(\gamma)$ are mapped to distinct
  colors under $Y \to \stY$, that is, $a$ fixes
  $D_\gamma^+$ and $D_\gamma^-$.
\end{proof}

\begin{theorem}\label{p:ker}
We have $A^\kk=\langle\Sigmac\rangle^\perp$, and hence
$\X^*(A^\kk)=\Xm/\langle\Sigmac\rangle$.
\end{theorem}

\begin{proof}
By Proposition \ref{p:Losev},
\[A(k)=\big\{a\in T^\Xm(k)\, \mid\, a(\lambda)=1
     \text{ for all }\lambda\in\langle\Sigma^N\rangle\big\}.\]
By definition, an automorphism $a\in A(k)=\Aut^G(G/H)$
is contained in $A^\kk(k)$ if and only if
$a$ fixes $D_\gamma^+$ and $D_\gamma^-$ for all $\gamma\in\Sigmat$.
By Proposition \ref{p:swaps}, this holds if and only if
\[ a(\gamma)=1\text{ for all }\gamma\in\Sigmat.\]
Hence
\[ A^\kk(k)=\big\{a\in T^\Xm(k)\, \mid\, a(\lambda)=1
   \text { for all }\lambda\in\langle\Sigma^N\rangle
    \text{ and for all }\lambda\in\Sigmat\big\}.\]
Since $\langle\Sigma^N\rangle + \langle\Sigmat\rangle=\langle\Sigmac\rangle$,
we conclude that
\[ A^\kk(k)=\big\{a\in T^\Xm(k)\, \mid\, a(\lambda)=1
    \text { for all }\lambda\in\langle\Sigmac\rangle\big\},\]
that is, $A^\kk=\langle\Sigmac\rangle^\perp$, as desired.
\end{proof}

\begin{proposition}[Proposition \ref{l:inj}]
The homomorphism
\[i_*\colon H^2(k_0, A_q^{\kk})\to  H^2(k_0, A_q)\]
induced by the inclusion homomorphism
$i\colon A_q^{\kk}\into A_q$
is injective.
\end{proposition}

\begin{proof}
We write  $T_q^\Xm$ (resp.~$T_q^N$, resp.~$T_q^\scc$)
for the $k_0$-torus with character group $\Xm$
(resp.~$\langle\Sigma^N\rangle$, resp.~$\langle\Sigmac\rangle$),
where the $\Gm$-action on $\Xm$ preserving  $\Sigma^N$ and $\Sigmac$
corresponds to the quasi-split $k_0$-model $G_\qs$ of $G$.
By Proposition \ref{p:Losev} and Theorem \ref{p:ker},
we have commutative diagrams with exact rows
\[
\xymatrix{
0\ar[r] &\langle\Sigma^N\rangle \ar[r]\ar[d]
       &\Xm \ar[r]\ar[d]^{\id}  &\X^*(A_q)\ar[r]\ar[d] &0\\
0\ar[r] &\langle\Sigmac\rangle \ar[r]
       &\Xm \ar[r]                  &\X^*(A_q^\kk)\ar[r]   &0
}
\]
and
\begin{equation}\label{e:second-diag}
\begin{aligned}
\xymatrix{
1\ar[r] &A_q^\kk\ar[r]\ar[d]^i &T_q^\Xm\ar[r]\ar[d]^{\id}  &T_q^\scc\ar[r]\ar[d] &1\\
1\ar[r] &A_q\ar[r]           &T_q^\Xm\ar[r]              &T_q^N\ar[r]          &1
}
\end{aligned}
\end{equation}
in which the arrows are $\Gm$-equivariant
(see Proposition \ref{prop:qsaut}).
Note that the $\sG$-module $\langle\Sigma^N\rangle$
has a $\sG$-stable $\Z$-basis $\Sigma^N$,
and hence,  $T_q^N$ is a quasi-trivial torus.
It follows that  $H^1(k_0, T_q^N)=1$; see Sansuc \cite[Lemma 1.9]{Sansuc}.
Similarly, $T_q^\scc$ is a quasi-trivial torus, and therefore,
$H^1(k_0, T_q^\scc)=1$.
Now from the diagram \eqref{e:second-diag}
we obtain a commutative diagram with exact rows
\[
\xymatrix{
1=H^1(k_0,T_q^\scc)\ar[r]\ar[d] &H^2(k_0,A_q^\kk)\ar[r]\ar[d]^{i_*}
      &H^2(k_0,T_q^\Xm)\ar[d]^{\id} \\
1=H^1(k_0,T_q^N)   \ar[r]       &H^2(k_0,A_q)   \ar[r]
      &H^2(k_0,T_q^\Xm)
}
\]
which shows that the homomorphism
\[i_*\colon H^2(k_0,A_q^\kk)\to H^2(k_0,A_q)\]
is injective, as required.
\end{proof}

\section{Tate duality}
\label{s:A-TN}

In this appendix we consider  cup product pairings and Tate duality.

\begin{subsec}\label{ss:A0-A}
Let $k_0$, $k$, and $\sG$ be as in Subsection \ref{ss:mod-Y}.
Let $A_0$ be a $k_0$-group of multiplicative type
(that is, a $k_0$-group isomorphic to a $k_0$-subgroup of a $k_0$-torus).
Write
\[A=A_{0,k}\coloneqq A_0\times_{k_0} k,\quad A^X=\X^*(A)\coloneqq \Hom(A,\G_{m,k}).\]
Then $A^X$ is a finitely generated abelian group.
The Galois group $\Gm=\Gal(k/k_0)$ acts continuously on the discrete group $A^X$
by
\[(\hs\upgam\chi)(a)=\upgam(\chi(\hs\hs^{\gamma^{-1}}\hm\hm a))
         \quad \text{for }\chi\in A^X,\ a\in A(k).\]
We consider $H^0(\Gm,A^X)=(A^X)^\Gm$ (the group of $\Gm$-fixed points in $A^X$).
We have a $\Gm$-equivariant $\Z$-bilinear pairing
\[  A^X\times A\to k^\times, \quad (\chi,a)\mapsto \chi(a),   \]
which induces a pairing on cocycles
\begin{equation}\label{e:AX-A-pairing}
(A^X)^\Gm \times Z^2(\Gm,A)\hs\to\hs Z^2(\Gm,k^\times),
\quad (\chi, a_{\gamma,\delta})\mapsto \chi(a_{\gamma,\delta})
\end{equation}
and a pairing (the cup product paring)   on cohomology:
\begin{equation}\label{e:cup-product-Br}
H^0(\Gm,A^X)\times H^2(\Gm,A)\labelto\cup H^2(\Gm,k^\times)=\Br k_0\hs,
    \quad (\chi,\alpha)\mapsto \chi\cupdot\alpha,
\end{equation}
where $\Br k_0$ is the Brauer group of $k_0$.
See Atiyah and Wall \cite[Section 7]{aw67} or Harari \cite[Section 2.5]{Harari}.

Let $s\in H^2(\Gm,A)$. By cup product, we obtain a {\em $\Br$-character}
(a homomorphism to the Brauer group)
\begin{equation}\label{e:Br-char}
s^0\colon H^0(\Gm,A^X)\to \Br k_0\hs,\quad \chi\mapsto \chi\cupdot s
     \quad \text{for }\chi\in H^0(\Gm,A^X)=(A^X)^\Gm.
\end{equation}

Now let $\phi\colon C_0\to A_0$ be a homomorphism
of $k_0$-groups of multiplicative type.
Write $C=C_0\times_{k_0} k$, $A=A_0\times_{k_0} k$.
Consider the dual to $\phi$ homomorphism
\[\phi^*\colon A^X\to C^X,\quad \chi_A\mapsto \chi_A\circ\phi\colon\,
    C\labelto{\phi} A\labelto{\chi_A}\G_{m,k}\quad \text{for }\chi_A\in A^X.\]

\begin{lemma}\label{l:c-chi-A}
The following diagram commutes:
\begin{equation*}
 \begin{gathered}
  \xymatrix@C=20pt{
H^0(\Gm,C^X)\!\!\!\ar@{}[r]|-{\times}             &\!\!\! H^2(\Gm,C)\,\ar[r]^\cup\ar[d]^{\phi_*}
                                                       &{\, \Br k_0}\ar@{=}[d]\\
H^0(\Gm,A^X)\!\!\!\ar@{}[r]|-{\times} \ar[u]^{\phi^*} &\!\!\! H^2(\Gm,A)\,\ar[r]^\cup
                                                      &{\, \Br k_0}
 }
 \end{gathered}
\end{equation*}
In other words,
let $(c_{\gamma,\delta})\in Z^2(\Gm,C)$ be a 2-cocycle,
and let $[c_{\gamma,\delta}]\in H^2(\Gm,C)$ denote its cohomology class.
Let $\phi_*[c_{\gamma,\delta}]\in H^2(\Gm,A)$ denote the image
of the class $[c_{\gamma,\delta}]$ under $\phi_*$\hs.
Let $\chi_A\in H^0(\Gm,A^X)$, and
let $\phi^*(\chi_A)\coloneqq \chi_A\circ\phi\in H^0(\Gm,C^X)$
denote the image of the character $\chi_A$ under $\phi^*$.
Then
\begin{equation*}\label{e:chi-c-phi}
\chi_A\cupdot\phi_*[c_{\gamma,\delta}]=\phi^*(\chi_A)
             \cupdot [c_{\gamma,\delta}]\,\in\, \Br k_0\hs.
\end{equation*}
\end{lemma}

\begin{proof}
By \eqref{e:AX-A-pairing}, we have
\begin{equation*}\label{e:chi-c-1}
\chi_A\cupdot\phi_*[c_{\gamma,\delta}]=[\chi_A(\phi(c_{\gamma,\delta}))]
=(\chi_A\circ\phi)\cupdot [c_{\gamma,\delta}]=\phi^*(\chi_A)
    \cupdot [c_{\gamma,\delta}].\qedhere
\end{equation*}
\end{proof}

\begin{corollary}\label{c:t-t0}
Let $\phi\colon C_0\to A_0$ be a homomorphism
of $k_0$-groups of multiplicative type.
Let $t\in H^2(\Gm,C)$ and let
\[t^0\colon H^0(\Gm,C^X)\to \Br k_0\hs,\quad\chi_C\mapsto\chi_C\cupdot
     t\ \text{ for }\chi_C\in H^0(\Gm, C^X),\]
denote the corresponding $\Br$-character.
Consider $\phi_*(t)\in H^2(\Gm,A)$ and the corresponding $\Br$-character
\[\phi_*(t)^0\colon H^0(\Gm,A^X)\to\Br k_0\hs,\quad
          \chi_A\mapsto\chi_A\cupdot \phi_*(t) \text\
{ for }\,\chi_A\in H^0(\Gm,A^X).\]
Then
\[\phi_*(t)^0=t^0\circ\phi^*\colon\, H^0(\Gm,A^X)\labelto{\phi^*}
 H^0(\Gm,C^X)\labelto{t^0}\Br k_0\hs.\]
\end{corollary}
\begin{proof}
Indeed, by Lemma \ref{l:c-chi-A}  the $\Br$-character $\phi_*(t)^0$ is given by
\[\chi_A\mapsto \phi^*(\chi_A)\cupdot t\quad
    \text{for }\chi_A\in H^0(\Gm,A^X).\qedhere\]
\end{proof}
\end{subsec}

\begin{proposition}\label{p:t0}
With the assumptions and notation of Subsection \ref{ss:A0-A},
assume that either $k_0=\R$ or $k_0$ is a $p$-adic field.
Let $\phi\colon C_0\to A_0$ be a homomorphism
of $k_0$-groups of multiplicative type.
Then for $t\in H^2(\Gm,C)$, we have  $\phi_*(t)=0\in H^2(\Gm,A)$
if and only if  $t^0\circ\phi^*=0$.
\end{proposition}

\begin{proof}
If $k_0=\R$, then by Lemma \ref{l:dual-R} below
we have $\phi_*(t)=0$  if and only if $\phi_*(t)^0=0$.
Similarly, if $k_0$ is a $p$-adic field,
then by Lemma \ref{l:dual-p} below we again have
 $\phi_*(t)=0$  if and only if $\phi_*(t)^0=0$.
By Corollary \ref{c:t-t0} we have  $\phi_*(t)^0=t^0\circ\phi^*$,
and the proposition follows.
\end{proof}

In the rest of this appendix we prove Lemmas \ref{l:dual-R} and \ref{l:dual-p},
which were used in the proof of Proposition \ref{p:t0}.

\begin{subsec}\label{ss:case-R}
Let $k_0=\R$ and $k=\C$.
Then $\Gm=\{1,\gamma\}$, there $\gamma$ is the complex conjugation.
We have $\Br \R\cong\half\Z/\Z$.

Let  $A_0$, $A$, and $A^X$ be as in Subsection \ref{ss:A0-A}.
 We consider the Tate modified zeroth cohomology group
\[\Hnul(\Gm,A^X)=H^0(\Gm,A^X)/\{\chi+\upgam\hm\chi\hs\mid\hs\chi\in A^X\}.\]
The cup product pairing \eqref{e:cup-product-Br} induces a pairing
\begin{equation}\label{e:Tate-R}
\Hnul(\Gm,A^X)\times H^2(\Gm,A)\to \Br \R =\half\Z/\Z;
\end{equation}
see Atiyah and Wall \cite[Section 7]{aw67} or Harari \cite[Section 2.6]{Harari}.
This pairing is perfect (Tate duality), that is, the induced
homomorphism of finite  abelian groups
\[ H^2(\Gm,A)\to \Hom(\Hnul(\Gm,A^X),\hs\Br \R)\]
is an isomorphism;
see Milne \cite[Theorem I.2.13(b)]{Milne-ADT}.
\end{subsec}

Let $s\in H^2(\Gm,A)$.
The pairings \eqref{e:cup-product-Br} and \eqref{e:Tate-R} define  $\Br$-characters
\[ s^0\colon H^0(\Gm,A^X)\to \Br \R\quad\text{and}
     \quad  \hat s\hs^0\colon \Hnul(\Gm,A^X)\to \Br \R\]
fitting into the commutative diagram
\begin{equation}\label{e:vert-R}
\xymatrix{
H^0(\Gm,A^X)\ar[r]^-{s^0}\ar@{->>}[d]  &\Br \R\ar@{=}[d]\\
\Hnul(\Gm,A^X)\ar[r]^-{\hat s\hs^0} &\Br \R
}
\end{equation}

\begin{lemma}\label{l:dual-R}
Let $k_0=\R$. Let  $A_0$, $A$, and $A^X$ be as in Subsection \ref{ss:A0-A},
and let $s\in H^2(\Gm,A)$.
We have $s=0$ if and only if\, $\hat s\hs^0=0$ if and only if $s^0=0$.
\end{lemma}

\begin{proof}
The first equivalence follows from the fact
that the pairing \eqref{e:Tate-R} is perfect.
The second equivalence follows from the fact that the left-hand vertical arrow
in the diagram \eqref{e:vert-R} is surjective.
\end{proof}

\begin{subsec}\label{ss:case-p}
Now assume that $k_0$ is a $p$-adic field.
Let $k$ be a fixed algebraic closure of $k_0$, and let $\Gm=\Gal(k/k_0)$.
Then $\Br k_0\coloneqq H^2(\Gm,k^\times)$ is canonically isomorphic to $\Q/\Z$;
see Serre \cite[Section 1.1]{Serre-ANT}
or Harari \cite[Theorem 8.9]{Harari}.

Let  $A_0$, $A$, and $A^X$ be as in Subsection \ref{ss:A0-A}.
The cup product pairing \eqref{e:cup-product-Br}
induces a continuous pairing
\begin{equation}\label{e:Tate-p}
H^0(\Gm,A^X)^\wedge \times H^2(\Gm,A)\to \Br k_0=\Q/\Z,
\end{equation}
where the compact group $H^0(\Gm,A^X)^\wedge$ is the profinite completion
of the finitely generated abelian group $H^0(\Gm,A^X)=(A^X)^\Gm$.
This pairing is perfect (Tate duality), that is,
the induced homomorphism of discrete abelian groups
\[  H^2(\Gm,A)\to\Hom_{\rm cont.\!}(H^0(\Gm,A^X)^\wedge, \Q/\Z)\]
is an isomorphism, where $\Hom_{\rm cont.}$
denotes the group of continuous homomorphisms.
See Serre \cite[Section II.5.8, Theorem 6(a)]{ser97}
or Milne \cite[Corollary I.2.4]{Milne-ADT}.

Let $s\in H^2(\Gm, A)$, and let
\[s^0\colon H^0(\Gm, A^X)\to \Br k_0\]
be as in \eqref{e:Br-char}.
We consider the induced continuous homomorphism
\[ s^{0\wedge}\colon H^0(\Gm, A^X)^\wedge\to \Br k_0=\Q/\Z \]
fitting into the commutative diagram
\begin{equation}\label{e:vert-p}
\xymatrix{
H^0(\Gm,A^X)\ar[r]^-{s^0}\ar[d]  &\Br k_0\ar@{=}[d]\\
H^0(\Gm,A^X)^\wedge\ar[r]^-{ s^{0\wedge}} &\Br k_0
}
\end{equation}
\end{subsec}

\begin{lemma}\label{l:dual-p}
Let $k_0$ be a $p$-adic field. Let  $A_0$, $A$,
and $A^X$ be as in Subsection \ref{ss:A0-A},
and let $s\in H^2(\Gm,A)$.
We have $s=0$ if and only if\, $s^{0\wedge}=0$ if and only if $s^0=0$.
\end{lemma}

\begin{proof}
The first equivalence follows from the fact
that the pairing \eqref{e:Tate-p} is perfect.
The second equivalence follows from the facts
that the left-hand vertical arrow
in the diagram \eqref{e:vert-p} has dense image
and the homomorphism  $s^{0\wedge}$ is continuous.
\end{proof}

\begin{remark}
By Lemmas \ref{l:dual-R} and \ref{l:dual-p},
if $k_0=\R$ or $k_0$ is a $p$-adic field,
then for $s\in H^2(\Gm,A)$ we have $s=0$ if and only if $s^0=0$.
However, the equivalence ``$s=0$ if and only if $s^0=0$\hs''
does not hold for an arbitrary field $k_0$ of characteristic 0.
Indeed, let us take $k_0=\Q$ and $A=\Z/3\Z$
(with the trivial action of $\Gm$ on $\Z/3\Z$).
Then $A^X\simeq\mu_3$ (with a nontrivial action of $\Gm$),
and hence $H^0(\Gm,A^X)=(A^X)^\Gm=0$.
On the other hand, $H^2(\Gm,A)\neq 0$, because $\Sha^2(\Q,\Z/3\Z)$ is finite
and $H^2(\Q_p,\Z/3\Z)\neq 0$ for infinitely many primes $p$.
It follows that there exists an element $s\neq 0$ in $H^2(\Gm,A)$,
and of course we have
$s^0=0\in \Hom(H^0(\Gm,A^X),\hs \Br k_0)$, because $H^0(\Gm,A^X)=0$.
(We thank Ofer Gabber for this counter-example.)
\end{remark}

\textsc{Acknowledgements.}
The authors are very grateful to Ronan Terpereau for a stimulating question;
this article was written in response to his question.
The authors are grateful to Ofer Gabber, David Harari,  Boris Kunyavski\u\i,
Lucy Moser-Jauslin, and Stephan Snegirov for helpful discussions.
The authors thank Victor Petrov and Remy van Dobben de Bruyn
for answering Borovoi's MathOverflow questions.
We thank the referee for thorough refereeing the paper and numerous useful comments,
which helped to improve the exposition.
This article was revised during the visit of the first-named author to
the Institut des Hautes \'Etudes Scientifiques (IHES) in the fall of  2019,
and he is grateful to this institute for support and excellent working conditions.

\end{document}